\documentclass[11pt]{amsart}
\usepackage{amsmath, amsfonts, amsthm, amssymb}
\makeatletter
\@namedef{subjclassname@2020}{%
  \textup{2020} Mathematics Subject Classification}
\makeatother
\usepackage{mathrsfs}
\usepackage{setspace}
\usepackage[margin=1in]{geometry}
\usepackage{enumitem} 
\usepackage{color}
\usepackage{tikz}
\usetikzlibrary{cd}
\usetikzlibrary{arrows}
\usetikzlibrary{patterns}
\usepackage[pdftex]{hyperref}
\usepackage{soul}

\usepackage[OT2,T1]{fontenc}
\DeclareSymbolFont{cyrletters}{OT2}{wncyr}{m}{n}
\DeclareMathSymbol{\be}{\mathalpha}{cyrletters}{"62}
\DeclareMathSymbol{\Be}{\mathalpha}{cyrletters}{"42}
\DeclareMathSymbol{\Sha}{\mathalpha}{cyrletters}{"58}

\newtheorem{theorem}{Theorem}[subsection] 
\newtheorem{proposition}[theorem]{Proposition} 
\newtheorem{lemma}[theorem]{Lemma} 
\newtheorem{corollary}[theorem]{Corollary} 
\newtheorem{conjecture}[theorem]{Conjecture}

\newtheorem{thm}{Theorem}

\theoremstyle{definition}
\newtheorem{definition}[theorem]{Definition}
\newtheorem{remark}[theorem]{Remark}
\newtheorem{example}[theorem]{Example} 
\newtheorem{notation}[theorem]{Notation}

\DeclareMathOperator{\Hom}{Hom}
\DeclareMathOperator{\supp}{supp}

\DeclareMathOperator{\End}{End}

\DeclareMathOperator{\im}{im}

\newcommand{\textdef}[1]{\textit{\textbf{#1}}}
\newcommand{\colim}{\mathop{\text{\normalfont colim}}}

\newcommand{\N}{{\mathbb N}}
\newcommand{\R}{{\mathbb R}}

\newcommand{\Cat}{\mathcal{A}}
\newcommand{\fac}{\mathsf{sf}}
\newcommand{\Sub}{\mathsf{Sub}}
\newcommand{\vect}{\mathrm{vec}}
\newcommand{\Vect}{\mathrm{Vec}}
\newcommand{\field}{\Bbbk}
\newcommand{\veck}{\vect(\field)}
\newcommand{\Veck}{\Vect(\field)}
\newcommand{\con}{\mathsf{con}}
\newcommand{\Fun}{\mathsf{Fun}}
\renewcommand{\supp}{\mathsf{supp}}
\newcommand{\wS}{\widetilde{S}}
\renewcommand{\L}{\mathcal{L}}
\renewcommand{\S}{\mathcal{S}}
\newcommand{\T}{\mathcal{T}}
\newcommand{\e}{\varepsilon}
\newcommand{\IT}{\mathsf{IT}(\mathbb{R})}
\newcommand{\proj}{\overset{\pi}{\sim}}
\newcommand{\pers}{\uparrow\uparrow}
\newcommand{\seg}{\mathsf{seg}}
\newcommand{\cov}{\mathsf{seg}}

\def\subobneq{\raisebox{-2.5pt}{\mbox{
\begin{tikzpicture}
    \draw[draw opacity =0] (0,0)-- (0,.25) -- (.25,.25) -- (.25,0) -- (0,0);
    \draw (0,.04) -- (.25,.04);
    \draw (.09,0) -- (.17,.08);
    \draw (.25,.3) -- (0,.3) -- (0,.1) -- (.25,.1);
\end{tikzpicture}
}}}
\def\subob{\raisebox{-2.5pt}{\mbox{
\begin{tikzpicture}
    \draw[draw opacity =0] (0,0)-- (0,.25) -- (.25,.25) -- (.25,0) -- (0,0);
    \draw (0,.04) -- (.25,.04);
    \draw (.25,.3) -- (0,.3) -- (0,.1) -- (.25,.1);
\end{tikzpicture}
}}}

\title[Composition series of arbitrary cardinality]{Composition series of arbitrary cardinality in modular lattices and abelian categories}
\author{Eric J.\ Hanson}
\address{(at time of journal submission) Department of Mathematical Sciences, Norwegian University of Science and Technology (NTNU), 7491 Trondheim, NORWAY}
\email{ejhanso3@ncsu.edu}
\author{J.\ Daisie Rock}
\address{Department W16, Ghent University, 9000 Ghent, East Flanders, BELGIUM}
\email{jobdrock@gmail.com}
\date{August 2, 2023}
\subjclass[2020]{Primary: 06C05, 18E10. Secondary: 16G20, 18B35, 55N31.}
\begin{document}

Adv. Math. {\bf 433} (2023), \href{https://doi.org/10.1016/j.aim.2023.109292}{DOI:10.1016/j.aim.2023.109292}.

\bigskip

\maketitle

\vspace{-0.75cm}

\begin{abstract}For a certain family of complete modular lattices, we prove a ``Jordan--H\"older--Schreier-like'' theorem with no assumptions on cardinality or well-orderedness. This family includes both lattices with are both join- and meet-continuous, as well as the lattices of subobjects of any object in an abelian category satisfying properties related to Grothendieck's axioms (AB5) and (AB5$^*$).
We then give several examples of objects in abelian categories which satisfy these axioms, including pointwise finite-dimensional persistence modules, presheaves, and certain Pr\"ufer modules.
Moreover, we show that, over an arbitrary ring, the infinite product of isomorphic simple modules both fails to satisfy our axioms and admits at least two composition series with distinct cardinalities.
We conclude by giving a lattice-theoretic proof that any object which is locally finitely generated and satisfies our axioms can be expressed as a direct sum of indecomposable subobjects. We conjecture that this decomposition is unique. \end{abstract}

\tableofcontents

\vspace{-0.75cm}

\section{Introduction}

The Jordan--H\"older Theorem (sometimes called the Jordan--H\"older--Schreier Theorem) remains one of the foundational results in the theory of modules. More generally, abelian length categories (in which the Jordan--H\"older Theorem holds for every object) date back to Gabriel \cite{length} and remain an important object of study to this day. See e.g. \cite{krause, Infinite Length Cats,LL}. 
The importance of the Jordan--H\"older Theorem in the study of groups, modules, and abelian categories has also motivated a large volume work devoted to establishing when a ``Jordan--H\"older-like theorem'' will hold in different contexts. Indeed, it was already observed in Dedekind's classical work~\cite{dedekind} that any two maximal chains in a finite (semi)modular lattice will have the same length. More recently, Gr\"atzer and Nation \cite{GN} used the theory of \emph{up-down perspectivity} to formalize a version of the ``uniqueness of composition factors'' for semimodular lattices, which was later strengthened in \cite{CS} (see also \cite{JH for lattices, JH for lattices2}). Another recent example is the exploration of ``Jordan--H\"older-like theorems'' for exact categories~\cite{BHT,enomoto}, which are a type of additive category which often contain objects whose subobject lattices are not (semi)modular. A commonality amongst all of these examples is that the ``composition series'' in question are assumed to be of finite length, as is the case for the classical Jordan--H\"older Theorem.

In the present paper, we extend the notion of a composition series so that there is no longer any assumption on the cardinality of the chain. While our primary motivation and examples come from the theory of abelian categories, most of our results hold more generally, and are therefore formulated in the language of modular lattices. More precisely, for a bounded lattice $\L$ (for example, the lattices of subobjects of an object in an abelian category), we consider (totally ordered) chains $\Delta \subseteq \L$. We then consider ``successive subfactors'' only at those points in $\Delta$ which contain an immediate successor or predecessor. When each of these subfactors is simple (or, in lattice terms, corresponds to a prime segment) and $\Delta$ satisfies certain closure conditions with respect to supremums and infimums, we call $\Delta$ a compositions series (or, in lattice terms, a prime chain). See  Section~\ref{sec:subobjectChains} for precise definitions and examples.

The consideration of infinite composition series goes back to at least 1934.
Indeed, in \cite{Birkhoff}, Birkhoff shows that if $\Delta$ and $\Gamma$ are composition series (in the above sense) for some group $G$ and are well-ordered, then their successive subfactors are the same up to permutation and isomorphism. The same paper, however, also shows that there is no hope for a fully general ``Jordan--H\"older-like theorem'' beyond this. Indeed, for distinct primes $p$ and $q$, we have that
$$ 0 \subsetneq \cdots \subsetneq p^2\mathbb{Z} \subsetneq p\mathbb{Z} \subsetneq \mathbb{Z}\qquad \text{and} \qquad 0\subsetneq \cdots \subsetneq q^2\mathbb{Z} \subsetneq q\mathbb{Z} \subsetneq \mathbb{Z}$$
are both composition series for $\mathbb{Z}$ (as a $\mathbb{Z}$-module) in the sense that their successive subfactors are simple. On the other hand, the successive subfactors of the first series are all cyclic of order $p$, while those of the second are cyclic of order $q$.

In this paper, we consider modular lattices (and objects in abelian categories) which satisfy an analog of the Grothendieck's properties (AB5) and (AB5$^*$). We call these lattices and objects (weakly) Jordan--H\"older--Schreier (Definition~\ref{JHSobjectsLattice}). We then establish a ``Jordan--H\"older-like theorem'' for this class of objects. This class includes objects of finite length, but is in general much larger.
For example, pointwise finite-dimensional persistence modules over arbitrary fields (Theorem~\ref{thm:pwf}) are (at least weakly) Jordan--H\"older--Schreier. More generally, in the category of right modules over an arbitrary ring, the (weakly) Jordan--H\"older--Schreier objects are precisely the so-called (AB5$^*$)-modules (Remark~\ref{rem:moduleJHS3}(4)). See Section~\ref{sec:examples} for details about these and additional examples.


\subsection{Motivation}
Our first motivation comes from the study of pointwise finite-dimensional persistence modules.
As defined by Botnan and Crawley-Boevey \cite{BC-B}, a \textdef{pointwise finite-dimensional persistence module} is a functor from a small category $\mathcal C$ to the category of finite dimensional vector spaces over an algebraically closed field $\field$.
Pointwise finite-dimensional persistence modules are primarily studied in topological data analysis via persistent homology. One key property in this setting is that, in many applications, the simple persistence modules are in bijection with the objects in the category $\mathcal C$.
This means there is a well-defined concept of support;
that is, for any pointwise finite-dimensional persistence module $M$, there is a set $\{x\in\mathcal C: M(x)\neq 0\}$ that we call the \textdef{support} of $M$.
In Section~\ref{sec:pwf} we show that if $\mathcal{C}$ is acyclic (for example, when $\mathcal{C}$ is a poset category), then the composition factors of a pointwise finite-dimensional persistence module $M$ over $\mathcal{C}$ are precisely the simple modules $S$ corresponding to the objects in the support of $M$.

Another source of motivation is Pr\"ufer modules. A module $M$ over an arbitrary ring $R$ is called \textdef{Pr\"ufer} if there exists a locally nilpotent surjective endomorphism $\varphi:M\rightarrow M$ so that $\ker(\varphi)$ has finite length. It follows that $M = \bigcup_{n\in\mathbb{N}}\ker(\varphi^n)$ is a filtered colimit of finite length modules, but is not itself of finite length.

In the representation theory of finite-dimensional (associative) algebras, Pr\"ufer modules are closely related to generic modules, and themselves contain information about the category of finitely-generated modules. See for example \cite{ringel}. It is therefore natural to try to include the Pr\"ufer modules when studying modules of finite length. As we show in Section~\ref{sec:preliminary}, there are many examples of Pr\"ufer modules which fit into our framework.

More generally, we note in Remark~\ref{rem:moduleJHS3} that a module over an arbitrary ring is weakly Jordan--H\"older--Schreier if and only if it is an ``(AB5$^*$)-module''. These modules have received considerable attention, especially in their relationship with with linear compactness and Morita duality, see \cite{anh} and the references therein. It has recently been shown that every (AB5$^*$)-module $M$ can be expressed as a direct sum of indecomposable modules \cite[Theorem~2.3]{IY}, and if one additionally assumes that $M$ has the ``finite exchange property'' then each indecomposable direct summand has local endomorphism ring and the decomposition is essentially unique \cite[Theorem~2.4]{IY}. This offers further rationale that one can see the (weakly) Jordan--H\"older--Schreier objects introduced in this paper as a natural generalization of finite length objects. In particular, we conjecture that it should be possible to extend the results of \cite{IY} to those (weakly) Jordan--H\"older--Scheier objects in an arbitrary well-powered abelian category which are ``locally finitely generated'' (see Definition~\ref{def:local_fin_gen_2}). We give a lattice-theoretic proof of the existence portion of this conjecture in Section~\ref{sec:directSum}.

As an explicit example, and returning to our first source of motivation, we note that direct sum decompositions are particularly relevant in the study of persistence modules. See e.g. \cite[Sections~3.1 and~4.1]{BL} and the references therein. Indeed, in the recent paper \cite{BC-B}, it is shown that every pointwise-finite dimensional persistence module (over any small category $\mathcal{C}$) decomposes uniquely as a direct sum of indecomposable persistence modules with local endomorphism rings. In \cite[Section~3.6]{GR}, this statement is also discussed, and its proof outlined, without explicit reference to persistence theory.


\subsection{Organization and Main Results}

The organization of this paper is as follows.
In Section~\ref{sec:background}, we recall background information on the Jordan--H\"older and Schreier Theorems for modular lattices of finite length, subobjects in abelian categories, length categories, and (co)limits in functor categories.
In Section~\ref{sec:subobjectChains}, we define the notion of a \textdef{bicomplete chain} in a bounded lattice (Definition~\ref{def:bicompleteChain}) and formalize our notions of \textdef{prime chains} and composition series as special cases (Definitions~\ref{def:filtration} and~\ref{def:filtration_lattice}). We then prove our first main theorem.

\begin{thm}[Theorem \ref{prop:compExists}]\label{thm:intro:compExists}\
    \begin{enumerate}
        \item Let $\L$ be a complete bounded lattice. Then the prime chains of $\L$ and the maximal chains in $\L$ coincide. In particular, for any chain $\Delta \subseteq \L$, there exists a prime chain $\Delta'$ which is a refinement of $\L$.
        \item Let $X$ be an object of a well-powered abelian category $\Cat$, and suppose the lattice of subobjects of $X$ is complete. Then the composition series of $X$ and the maximal chains of subobjects of $X$ coincide. In particular, for any chain $\Delta$ of subobjects of $X$, there exists a composition series $\Delta'$ of $X$ which is a refinement of $\Delta$.
    \end{enumerate}
\end{thm}

\begin{remark}
    Note that in the statement of Theorem~\ref{thm:intro:compExists}(2), we require only that the subobjects of $X$ form a complete lattice; that is, we do not require the limit or colimit of an infinite chain of subobjects to exist. See Section~\ref{sec:limits_colimits} for further discussion.
\end{remark}

\begin{remark}
    We note that our notions of bicomplete (subobject) chains and composition series also come with ``one-sided'' variants, where we assume the chain is closed under either greatest lower bounds or least upper bounds, but not necessarily both. We maintain, however, that the ``two-sided'' variant is the correct generalization due to the fact that every ``one-sided'' composition series $\Delta$ can be extended into a ``two-sided'' composition series. See Example~\ref{ex:refineableCompSeries}.
\end{remark}

\begin{remark}
Theorem~\ref{thm:intro:compExists} only guarantees the existence of a composition series, but does not imply any notion of uniqueness. Indeed, $\mathbb{Z}$ (as a $\mathbb{Z}$-module) admits a composition series under our definition. See Example~\ref{ex:compSeries}.
\end{remark}

We begin Section~\ref{sec:JHSobjects} by defining what it means for two (bicomplete) chains to be \textdef{up-down perspective} (Definition~\ref{def:factorEquivalenceLattice}). For (bicomplete) chains of subobjects in an abelian category, this descends to \textdef{subfactor equivalence} (Definition~\ref{def:factorEquivalence}). In both cases, these definitions extend the classical notions used in the statements of the (finite length) Jordan--H\"older Theorems. We then introduce
\textdef{(weakly) Jordan--H\"older--Schreier} lattices and objects (Definition~\ref{JHSobjectsLattice}). This definition requires not only for least upper bounds and greatest lower bounds to exist, but also to satisfy properties related to Grothendieck's axioms (AB5) and (AB5$^*$).

Once establishing the definition and basic properties, Section~\ref{sec:JHS} is devoted to proving the main result of this paper.

\begin{thm}[Theorem \ref{thm:Jordan--Holder}]\label{thm:intro:Jordan--Holder}\ 
\begin{enumerate}
        \item Let $\L$ be a weakly Jordan--H\"older--Schreier (complete bounded modular) lattice. Then there exists a prime chain in $\L$ and any two primes chains in $\mathcal{L}$ are up-down perspective. Moreover, the map inducing the up-down perspectivity between two given prime chains is unique.
        \item Let $X$ be a weakly Jordan--H\"older--Schreier object in a well-powered abelian category.
   Then there exists a composition series for $X$ and any two composition series for $X$ are subfactor equivalent.
   \end{enumerate}
\end{thm}

In light of this theorem, the (multi)set of composition factors associated to a (weakly) Jordan--H\"older--Schreier object $X$, denoted $\fac(X)$, is well-defined.

Our proof of Theorem \ref{thm:intro:Jordan--Holder} is similar to that of the classical (and well-ordered) Jordan--H\"older Theorem given by ``conflating'' two chains of subobjects (see e.g. \cite[Section~I.3]{Lang}, \cite{Transfinite}). In particular, it relies on a generalization of Schreier's classical ``refinement theorem'' (see Theorems~\ref{thm:refinement},~\ref{thm:JH_lattices}(2), and~\ref{thm:schreier}).

As a consequence of Theorem~\ref{thm:intro:Jordan--Holder}, we show that (weakly) Jordan--H\"older--Schreier objects have many properties in common with those of finite length. In particular, we prove the following, which also lends itself to a lattice-theoretic generalization given explicitly in Theorem~\ref{cor:inducedInclusion}.

\begin{thm}[Theorem~\ref{cor:inducedInclusion}, simplified]\label{thm:intro:inducedInclusion}
    Let $X$ be a weakly Jordan--H\"older--Schreier object in a well-powered abelian category $\Cat$. Let
    $0\rightarrow Y \xrightarrow{f} X \xrightarrow{g} Z\rightarrow 0$
    be a short exact sequence in $\Cat$. Then $Y$ and $Z$ are also weakly Jordan--H\"older--Schreier and there is an induced bijection (of multisets) $\fac(Y)\sqcup\fac(Z) \cong \fac(X)$. Moreover, if $Y$ (respectively $Z$) is Jordan--H\"older--Schreier, then $Z = 0$ (respectively $Y = 0$) if and only if the induced inclusion $\fac(Y)\hookrightarrow \fac(X)$ (respectively $\fac(Z) \hookrightarrow \fac(X)$) is an isomorphism.
\end{thm}

In Section~\ref{sec:limits_colimits}, we compare the completeness of lattices of subobjects to the existence of exact bounded direct limits. In particular, we show that if $\mathcal{A}$ and $\mathcal{A}^{op}$ have $X$-bounded direct limits for some object $X$, then $X$ must be weakly Jordan--H\"older--Schreier (Corollary~\ref{cor:catComp_AB5}). We further show that the categories $\Cat$ and $\Cat^{op}$ both have exact bounded direct limits if and only if every object in $\Cat$ is weakly Jordan--H\"older--Schreier. (Corollary~\ref{cor:catComp_AB5_2}).

In Section~\ref{sec:examples}, we give examples and non-examples of (weakly) Jordan--H\"older--Schreier objects. Our examples include objects of finite length and (AB5$^*$) modules (Section~\ref{sec:preliminary}), as well as functor categories with a length category as a target (Section \ref{sec:pwf}). The latter in particular includes categories of pointwise definite-dimensional persistence modules. Furthermore, when the source category is directed (for example, when the functors are presheaves over some topological space), we describe the composition factors explicitly. See Sections~\ref{sec:fun cats} and~\ref{sec:pwf} for precise definitions and the more general version of our final main theorem, which we state in simplified form below.

\begin{thm}[Proposition~\ref{prop:FunctorCategoryIsJHS}, Theorem~\ref{thm:pwf}, and Corollary \ref{cor:pwf}, simplified]\label{thm:intro:pwf}\ 
\begin{enumerate}
    \item Let $\mathcal{C}$ be a directed small category and let $\field$ be an arbitrary field. Then any pointwise finite-dimensional $\mathcal{C}$-persistence module $M:\mathcal{C}\rightarrow \veck$ is Jordan--H\"older--Schreier (in the category of covariant functors from $\mathcal{C}$ to finite-dimensional $\field$-vector spaces). Moreover, for each object $X$ of $\mathcal{C}$, the simple module with support at $X$ is a composition factor of $M$ with multiplicity $\dim_\field M(X)$.
    \item Let $\Cat$ be a well-powered abelian category such that every object in $\Cat$ is Jordan--H\"older--Schreier.
    Then any presheaf on a topological space $X$ with values in $\Cat$ is Jordan--H\"older--Schreier (in the category of contravariant functors from the poset of open sets of $X$ to~$\mathcal{A}$).
\end{enumerate}
\end{thm}

We conclude Section~\ref{sec:examples} with two additional examples. First, we show that every object of Igusa and Todorov's category of representations of $\mathbb{R}$ \cite{IT15} are weakly Jordan--H\"older--Schreier, even though this category contains no simple objects (Section \ref{sec:ITreps}). Finally, we show that, over any ring, the infinite product of copies of some simple module is not weakly Jordan--H\"older--Schreier and admits two composition series with different cardinalities (Section~\ref{sec:grothendieck}).

Lastly, in Section~\ref{sec:directSum}, we discuss existence and uniqueness properties for direct sum decompositions, and more generally \textdef{independent decompositions} (Definition~\ref{def:independent}), for Jordan--H\"older--Schreier objects and lattices. In particular, we give partial results extending the recent work \cite{IY} on direct sum decompositions of (AB5$^*$)-modules (Theorem~\ref{thm:decomp_existence}). We also describe a plan for continuing this work and give a formal conjecture regarding the uniqueness of independent decompositions (Conjecture~\ref{conj:decomp_unique}).

\subsection{Related work on other additive categories}
While we have focused mainly on applications of our results to abelian categories, there has been recent work studying the (finite) Jordan--H\"older property in other additive categories, including exact categories \cite{BHT,enomoto} and the more general extriangulated categories \cite{BHST,WWZZ}. It is made explicit in \cite[Section~2.4]{enomoto} that if $X$ is an object in exact category $\mathcal{A}$ and both (i) $X$ admits a (finite) composition series, and (ii) the admissible subobjects of $X$ form a modular lattice, then every composition series of $X$ will have the same length as a result of the Jordan--H\"older Theorem for modular lattices. The present paper could thus be used to extend this result by (i) removing the assumption that $X$ admits a finite composition series, and (ii) assuming that the lattice of admissible subobjects of $X$ is ``weakly Jordan--H\"older--Schreier''. On the other hand, it is made explicit in \cite[Section~5.2]{BHT} that there are examples of objects in exact categories which satisfy a Jordan--H\"older property, but whose posets of admissible subobjects are not (modular) lattices. It would be an interesting problem to see whether the posets in these examples fit into a larger family which satisfy an order-theoretic Jordan--H\"older property.

\subsection*{Acknowledgments}
EH was partially supported by the Canada
Research Chairs program (CRC-2021-00120) and NSERC Discovery Grants (RGPIN-2022-03960
and RGPIN/04465-2019). JDR was partly supported at Ghent University by BOF grant 01P12621. Portions of this work were completed while EH was a graduate student at Brandeis University and a postdoc at Universit\'e du Qu\'ebec \`a Montr\'eal and Universit\'e de Sherbrooke, and while JDR was supported by the Hausdorff Research Institute for Mathematics. The authors thank these institutions for their support and hospitality. The authors are thankful to Aslak Bakke Buan for helpful comments. They are also grateful to an anonymous referee for their thoughtful suggestions for simplifying our axioms, expanding our results to modular lattices, and exploring direct sum decompositions and quasi-coherent sheaves.
JDR thanks Sergio Estrada for enlightening discussions regarding quasi-coherent sheaves.

\section{Background}\label{sec:background}

We begin by recalling several definitions about (modular) lattices. We refer to \cite{gratzer_book} as a standard reference for this material. Let $\L = (\L,\leq)$ be a poset. Given a subset $\Delta \subseteq \L$, we denote by $\bigvee \Delta$ and $\bigwedge\Delta$ the supremum (or join) and infimum (or meet) of $\Delta$, if they exist. The poset $\L$ is a called a \textdef{lattice} if for all $x, y \in \L$, both $x \vee y := \bigvee\{x,y\}$ and $x \wedge y := \bigwedge\{x,y\}$ exist. We say a lattice $\L$ is \textdef{modular} if $x \vee (y \wedge z) = (x \vee y) \wedge z$ for any $x, y, z \in \L$ such that $x \leq z$. (Note that satisfying the modular law is a weaker condition than satisfying the distributive law(s).) Finally, we say a lattice $\L$ is \textdef{bounded} if there exists a unique maximal element $\hat{1} \in \L$ and a unique minimal element $\hat{0} \in \L$. By convention, we set $\bigvee \emptyset:= \hat{0}$ and $\bigwedge \emptyset := \hat{1}$ for any bounded lattice. Unless otherwise stated, the symbol $\L$ will always refer to a bounded modular lattice.

If $\L$ is a lattice, it is straightforward to show that $\bigvee\Delta$ and $\bigwedge\Delta$ both exist for any nonempty finite subset $\Delta \subseteq \L$. If, moreover, $\bigvee\Delta$ (respective $\bigwedge\Delta$) exists for all $\Delta \subseteq \L$, the lattice $\L$ is called \textdef{complete} (respectively \textdef{cocomplete}). It is well-known that a bounded lattice is complete if and only if it is cocomplete. Indeed, for $\Delta \subseteq \L$, one will have $\bigwedge\Delta = \bigvee\{y \in \L \mid (\forall x \in \Delta) \ y \leq x \}$ and vice versa.

A subset $\mathcal{K} \subseteq \L$ is called a \textdef{sublattice} if given any $x, y \in \mathcal{K}$, one has $x \vee y, x \wedge y \in \mathcal{K}$. If in addition $\L$ is complete and for any nonempty subset $\Delta \subseteq \mathcal{K}$ one has $\bigvee \Delta, \bigwedge \Delta \in \mathcal{K}$, then $\mathcal{K}$ is called a \textdef{complete sublattice}. For example, given $s_0 \leq s_1 \in \L$, the \textdef{segment} from $s_0$ to $s_1$ is the subset $[s_0,s_1]:= \{x \in \L \mid s_0 \leq x \leq s_1\}$. It is straightforward to show that $[s_0,s_1]$ is a complete sublattice with minimal element $s_0$ and maximal element $s_1$. Finally, we recall that $[s_0,s_1]$ is a \textdef{prime segment} if $|[s_0,s_1]| = 2$. In this case, one says that $s_0 \leq s_1$ is a \textdef{cover relation}.


\subsection{Subobjects in abelian categories}\label{sec:subobjects}
We now recall the definition and basic properties of subobjects in an abelian category. A detailed and well-written treatise for this theory can be found in~\cite{abelian lecture notes}.

Fix an abelian category $\Cat$. We recall that in $\Cat$, finite limits and finite colimits always exist. Moreover, arbitrary limits will commute with one another whenever they exist, and likewise for colimits.

Let $X$ be an object in $\Cat$. A \textdef{subobject} of $X$ is a pair $(Y,\iota_Y)$ where $Y$ is also an object in  $\Cat$ and $\iota_Y:Y\rightarrow X$ is a monomorphism. The subobjects of $X$ form a category, which we denote $\Sub(X)$. Morphisms in this category are given by
$$\Hom_{\Sub(X)}((Y,\iota_Y),(Z,\iota_Z)) = \{h \in \Hom_\Cat(Y,Z)| \iota_Y = \iota_Z\circ h\}.$$
Necessarily, we see that if $h \in \Hom_{\Sub(X)}((Y,\iota_Y),(Z,\iota_Z))$, then $h$ is a monomorphism in $\Cat$. Moreover, we have that $|\Hom_{\Sub(X)}((Y,\iota_Y),(Z,\iota_Z))| \leq 1$.

We will assume throughout this paper that $\Cat$ is well-powered, which by definition means that $\Sub(X)$ is skeletally small for every object $X$. Thus (any skeleton of) $\Sub(X)$ has the structure of a poset under the relation $(Y,\iota_Y) \subob (Z,\iota_Z)$ if $\Hom((Y,\iota_Y),(Z,\iota_Z)) \neq \emptyset$. If we denote this unique map by $f_{Y,Z}$, then we can view $(Y,f_{Y,Z})$ as an object in $\Sub(Z)$. As a result, we will also use the notation $\subob$ to mean ``is a subobject of''. We will freely move between using $\Sub(X)$ to denote both the category of subobjects and the poset obtained by taking a skeleton. The above discussion can then by summarized as saying that given $Z \subob X$, the lattice $\Sub(Z)$ can be identified with the interval $[0,Z] \subseteq \Sub(X)$.

In the case that $\Cat$ is a module category, the relation $\subob$ coincides with the usual notion of containment for submodules. Nevertheless, we have chosen to use the notation $\subob$ for this relation to avoid confusion with refinements of subobject chains, which are actually subsets (see Definitions~\ref{def:bicompleteChain} and ~\ref{def:refinement}).

We adopt the common notation of omitting the inclusion map from the description of a subobject when this data is implied. We caution that, as in the module case, this means there may be subobjects $Y$ and $Z$ of $X$ which are isomorphic in $\Cat$ but which are not isomorphic in $\Sub(X)$. To distinguish this, we will say that $Y = Z$ when they are isomorphic as subobjects of $X$.

The poset $\Sub(X)$ is known to be a bounded modular lattice. Given $Y,Z \in \Sub(X)$, their supremum is $Y \vee Z = Y+Z$ and their infimum is $Y \wedge Z = Y\cap Z$. In module categories, these coincide with the usual notions of sums and intersections. Categorically, we have that $Y \cap Z$ and $Y + Z$ are the kernel and image of the morphism $Y\oplus Z \xrightarrow{[\iota_Y \ \iota_Z]} X$, respectively. In particular, let $j_Y: Z\cap Y \rightarrow Y$ and $J_Z:Z\cap Y \rightarrow Z$ be the inclusion maps. Then $Y\cap Z$ and $Y+ Z$ are the pullback and pushout of the respective diagrams $$Z\xrightarrow{\iota_Z} X\xleftarrow{\iota_Y} Y \qquad \text{and} \qquad Z\xleftarrow{j_Z} Z\cap Y \xrightarrow{j_Y} Y.$$

If $\Cat$ has arbitrary coproducts, we can more generally speak of infinite intersections and sums. Indeed, in this case the lattice $\Sub(X)$ is complete and satisfies $\bigvee \Delta = \sum \Delta$ for any $\Delta \subseteq \Sub(X)$. On the other hand, it is possible that $\Sub(X)$ is complete even if $\Cat$ does not have arbitrary coproducts. Due to this fact, we will typically use the notation $\bigvee \Delta$ rather than $\sum \Delta$. Likewise, we will use $\bigwedge \Delta$ rather than $\bigcap \Delta$.

It should be noted that there is a dual theory for quotient objects in an abelian category $\Cat$, which can be viewed as the presented theory in $\Cat^{op}$.
As we shall see in later sections, it is sometimes useful to move between the lattice of subobjects and that of quotient objects. To that end, we formulate the third and fourth isomorphism theorems in the language of this paper.

\begin{proposition}\label{prop:subQuotients}
    Let $X$ be an object in $\Cat$.
    \begin{enumerate}
        \item Let $Z \subob Y \subob X$. Then the segment $[Z,Y] \subseteq \Sub(X)$ is isomorphic (as a complete lattice) to $\Sub(Y/Z)$ via the association $W \mapsto W/Z$. In particular, $[Z,Y]$ is a prime segment if and only if $Y/Z$ is simple.
        \item Let $\Sub^{op}(X)$ denote the poset of subobject of $X$ in $\Cat^{op}$. Then there is an anti-isomorphism of complete lattices $\Sub(X) \rightarrow \Sub^{op}(X)$ given by $Y \mapsto X/Y$. Moreover, given a segment $[Z,Y] \subseteq \Sub(X)$, one has that $Y/Z$ is the kernel of the quotient map $(X/Y) \twoheadrightarrow (X/Z)$.
    \end{enumerate}
\end{proposition}
Finally, we recall that an object $Y$ in $\Cat$ is called a \textdef{subquotient} of $X$ if there exists a subobject $Z \subob X$ and an epimorphism $Z \twoheadrightarrow Y$. This notion is also self-dual, in that $Y$ is a subquotient of $X$ if and only if there exists some epimorphism $X \twoheadrightarrow Z'$ such that $Y \in \Sub(Z')$. In this language, Proposition~\ref{prop:subQuotients}(1) in particular says that $Y$ is a subquotient of $X$ if and only if $\Sub(Y)$ can be identified with a segment in $\Sub(X)$.

\subsection{Composition series and length}\label{sec:length background}
In this section, we discuss the notions of composition series and length in abelian categories. We then discuss the extension of these notions to modular lattices in the following section. Recall that $\Cat$ denote a well-powered abelian category, and let $X$ be an object in $\mathcal{A}$. 

Consider
$$\Delta = \{0=X_0 \subobneq X_1 \subobneq \cdots\subobneq X_{n-1} \subobneq X_n=X\}$$
a finite sequence of subobjects of $X$.
We refer to the objects $X_i/X_{i-1}$ for $1 \leq i \leq n$ as the (successive) subfactors of $\Delta$. If all of these subfactors are simple, then $\Delta$ is called a \textdef{composition series}. We say $\Delta$ is \textdef{subfactor equivalent} to another finite chain of subobjects
$$\Gamma = \{0 = Y_0 \subobneq Y_1 \subobneq \cdots\subobneq Y_{m-1}\subobneq Y_m = X\}$$
if $m = n$ and there exists a permutation $\sigma$ on $\{1,\ldots,n\}$ so that $X_i/X_{i-1} \cong Y_{\sigma(i)}/Y_{\sigma(i)-1}$ for each $i$. This allows us to state the well-known Jordan--H\"older Theorem for abelian categories.

\begin{theorem}[Jordan--H\"older Theorem]\label{thm:jh}
    Let $X$ be an object in $\Cat$ and let $\Delta$ and $\Gamma$ be composition series of $X$. Then $\Delta$ and $\Gamma$ are subfactor equivalent.
\end{theorem}

Based upon this theorem, if $\Delta$ is a composition series of $X$, then the length of $\Delta$ is referred to as the \textdef{length} of $X$ and the subfactors of $\Delta$ are referred to as the \textdef{composition factors} of $X$. The abelian category $\Cat$ is then called a \textdef{length category} if every object has finite length (and thus well-defined length and composition factors).

Closely related to the Jordan--H\"older Theorem is the Schreier Refinement Theorem, which says that one may essentially conflate the data of two finite chains of subobjects.

\begin{theorem}[Schreier Refinement Theorem]\label{thm:refinement}
    Let $X$ be an object in $\Cat$ and let $\Delta$ and $\Gamma$ be finite filtrations of $X$.
    Then there exist filtrations $\Delta'$ and $\Gamma'$ of $X$ such that:
    \begin{enumerate}
        \item Each object that appears in $\Delta$ appears in $\Delta'$ and each object that appears in $\Gamma$ appears in $\Gamma'$.
        \item The filtrations $\Delta'$ and $\Gamma'$ are (subfactor) equivalent.
    \end{enumerate}
\end{theorem}

Indeed, it is common to prove the Schreier Refinement Theorem first and obtain the Jordan--H\"older Theorem as a corollary. See e.g.~\cite[Section~I.3]{Lang} in the setting of groups. We will adopt a similar strategy in proving our generalization of the Jordan--H\"older Theorem.

\subsection{The Jordan--H\"older and Schreier Theorems for modular lattices} \label{sec:jh_for_lattices}
As mentioned in the introduction, both Theorem~\ref{thm:jh} and Theorem~\ref{thm:refinement} can be deduced from more general lattice-theoretic statements. We will review these more general statements in this section.

Fix a bounded modular lattice $\L$.
For each segment $\S = [s_0,s_1] \subseteq \L$, there is an order-preserving map $\pi_\S: \L \rightarrow \S$ given by
$$\pi_\S(x) := (x \vee s_0) \wedge s_1 = (x \wedge s_1) \vee s_0.$$
Two segments $\S, \T \subseteq \L$ are said to be \textdef{up-down perspective}, and one writes $\S \pers \T$, if the maps $\pi_{\S}|_\T$ and $\pi_\T|_\S$ are inverse bijections. The transitive closure of the relation $\pers$ is an equivalence relation $\proj$ known as \emph{projective equivalence}. We denote by $\S_\pi$ the equivalence class of a segment $\S$ with respect to $\proj$.

More generally, consider
    $$\Delta = \{\hat{0} = x_0 \lneq x_1 \lneq \cdots \lneq x_n = \hat{1}\}$$
a finite chain in $\L$. We refer to the segments $[x_{i-1}, x_i]$ for $1 \leq i \leq n$ as the (successive) subsegments of $\Delta$. If each of these is a prime segment, then $\Delta$ is called a \textdef{maximal chain}. We say $\Delta$ is \textdef{up-down perspective} to another finite chain
    $$\Gamma = \{\hat{0} = y_0 \lneq y_1 \lneq \cdots \lneq y_m = \hat{1}\}$$
if $m = n$ and there exists a permutation $\sigma$ on $\{1,\ldots,n\}$ so that $[x_{i-1},x_i] \pers [y_{\sigma(i)-1},y_i]$ for each $i$. This allows us to state versions of both the Jordan--H\"older Theorem and Schreier Refinement Theorem for modular lattices of finite length. To make the statement precise, we refer to $\sigma$ above as a bijection which \textdef{induces an up-down perspectivity}. We also recall that a subset $\Delta \subseteq \L$ is called a \textdef{chain} if it is totally ordered under the relation inherited from $\L$.

\begin{theorem}[Jordan--H\"older and Schreier Theorems for modular lattices of finite length]\label{thm:JH_lattices}
    Let $\L$ be a modular lattice and suppose that every (maximal) chain in $\L$ is finite. Let $\Delta, \Gamma \subseteq \L$ be arbitrary chains. Then the following hold.
    \begin{enumerate}
        \item If $\Delta$ and $\Gamma$ are maximal, then there is a unique bijection $\sigma: \Delta \rightarrow \Gamma$ which induces an up-down perspectivity.
        \item There exist chains $\Delta', \Gamma' \subseteq \L$ which are up-down perspective and satisfy $\Delta \subseteq \Delta'$ and $\Gamma \subseteq \Gamma'$.
    \end{enumerate}
\end{theorem}

\begin{remark}\
    \begin{enumerate}
        \item Theorem~\ref{thm:JH_lattices} can more generally be formulated for (upper) \emph{semi}modular lattices of finite length. This amounts to weakening the modular law to instead only require that if $[x\wedge y,x]$ is a prime segment, then $[x,x\vee y]$ is also. At this level of generality, the uniqueness of $\sigma$ in Theorem~\ref{thm:JH_lattices} is due to \cite[Theorem~1.3]{CS} and the rest of the theorem is due to \cite[Theorem~2]{GN}.
        \item It is well-known that a lattice $\mathcal{L}$ is modular if and only if it is both upper and lower semimodular; i.e. if and only if both $\mathcal{L}$ and $\mathcal{L}^{op}$ are upper semimodular. Since our main results are based upon many self-dual assumptions (see e.g. Definition~\ref{JHSobjectsLattice}), we have thus chosen to work exclusively over modular lattices in the present paper.
    \end{enumerate}
\end{remark}

The uniqueness of the bijection $\sigma$ in Theorem~\ref{thm:JH_lattices}(1) is of particular note, as it is not present in the classical statement of the Jordan--H\"older Theorem for abelian categories. To clarify this distinction, and to describe the relationship between up-down perspectivity and subfactor equivalence, we need the following lemma.

\begin{lemma}\label{lem:proj}
    Let $\L$ be a modular lattice, and let $\S = [s_0,s_1]$ and $\T = [t_0,t_1]$ be segments of $\L$.
    Then $\S \pers\T$ if and only if (i) $t_1 \vee s_0 = s_1 \vee t_0$, (ii) $t_1 \wedge (t_0 \vee s_0) = t_0$, and (ii') $s_1 \wedge (t_0 \vee s_0) = s_0$.
\end{lemma}

\begin{proof}
    Suppose $\S\pers\T$. Then $\pi_\S|_\T$ is an order-preserving bijection, and so $(t_0 \vee s_0) \wedge s_1 = \pi_{\S}(t_0) = s_0$, (proving (ii')) and $(t_1 \vee s_0) \wedge s_1 = \pi_{\S}(t_1) = s_1$. In particular, $t_1 \vee s_0 \geq s_1$, and so $t_1 \vee s_0 \geq t_0 \vee s_1$. By symmetry, we also have that $(s_0 \vee t_0) \wedge t_1 = t_0$ and $t_1 \vee s_0 \leq t_0 \vee s_1$. This shows that (i) and (ii) both hold.
    
    Now suppose that (i), (ii), and (ii') all hold. Then
    \begin{align*}
        \pi_{\T}\circ \pi_{\S}(x) &= (((x \vee s_0) \wedge s_1)\vee t_0)\wedge t_1 &\\
        &= ((x \vee s_0) \wedge (s_1 \vee t_0)) \wedge t_1 & \text{since $t_0 \leq x \vee s_0$}\\
        &= (x\vee s_0) \wedge ((t_1 \vee s_0) \wedge t_1)) & \text{by (i)} \\
        &= (x \vee s_0) \wedge t_1 &\\
        &= (x \vee (t_0 \vee s_0)) \wedge t_1 & \text{since $t_0 \leq x$}\\
        &= x \vee ((t_0 \vee s_0) \wedge t_1) & \text{since $x \leq t_1$}\\
        &= x \vee t_0 & \text{by (ii)}\\
        &= x. &
    \end{align*}
    By symmetry, we conclude that $\pi_{\S}|_\T$ and $\pi_{\T}|_\S$ are inverses; that is, that $\S\pers\T$.
\end{proof}

As a consequence, we obtain the following.

\begin{proposition}\label{prop:subfactor_proj}
    Let $X$ be an object in $\mathcal{A}$, and let $[Y_0, Y_1]$ and $[Z_0, Z_1]$ be segments in $\Sub(X)$. Suppose $[Y_0,Y_1] \pers [Z_0,Z_1]$. Then $Y_1/Y_0 \cong Z_1/Z_0$.
\end{proposition}

\begin{proof}
    The result is a consequence of Lemma~\ref{lem:proj} and the second isomorphism theorem. Indeed, we have
    $$\frac{Y_1}{Y_0} \cong \frac{Y_1}{Y_1\cap(Z_0 + Y_0)} \cong \frac{Z_0 + Y_1}{Z_0 + Y_0} \cong \frac{Z_1 + Y_0}{Z_0 + Y_0} \cong \frac{Z_1}{Z_1 \cap (Z_0 + Y_0)} \cong \frac{Z_1}{Z_0}$$
    where the first, third, and fifth isomorphisms are due to conditions (ii), (i), and (ii') in Lemma~\ref{lem:proj} and the other two isomorphisms are due to the second isomorphism theorem.
\end{proof}

As a special case, and as a motivating example for the definition of up-down perspective, we recall that the Diamond Isomorphism Theorem says that for any $x, y \in \L$, one has $[x \wedge y, x] \pers [y, y \vee x]$. Combining this with Proposition~\ref{prop:subfactor_proj} yields precisely the second isomorphism theorem.

We conclude with an example which highlights both (i) the fact that the converse of Proposition~\ref{prop:subfactor_proj} does not hold, and (ii) that applying Theorem~\ref{thm:JH_lattices} to the lattice $\Sub(X)$ results in a stronger statement than Theorem~\ref{thm:jh}.

\begin{example}\label{ex:subfact_vs_proj}
    Suppose $\mathcal{A}$ contains a simple object $S$. We consider the object $X = S\oplus S$ and subobjects $Y_1 = (S, \iota_1)$ and $Z_1 = (S, \iota_2)$, where $\iota_1$ and $\iota_2$ are the inclusion maps. Then $X$ has composition series
    $\Delta = \{0 = Y_0 \subobneq Y_1 \subobneq Y_2 = X\}$ and $\Gamma = \{0 = Z_0 \subobneq Z_1 \subobneq Z_2 = X\}$. Now let $\sigma_1, \sigma_2: \{1,2\} \rightarrow \{1,2\}$ be the identity permutation and the transposition (1,2), respectively. Then for $i, j \in \{1,2\}$, we have $Y_i/Y_{i -1} \cong Z_{\sigma_j(i)}/Z_{\sigma_j(i)-1}$, as desired. However, we claim that only $\sigma_2$ induces an up-down perspectivity. Indeed, if $\sigma_1$ induced an up-down perspectivity, Lemma~\ref{lem:proj} and the assumption that $[Y_1,X] \pers [Z_1,X]$ would imply $X = X \cap (Y_1 + Z_1) = Y_1$, a contradiction. Note in particular that this also shows that the chains $\{0 \subobneq Y_1\}$ and $\{0 \subobneq X_1\}$ are subfactor equivalent, but not up-down perspective. On the other hand, $\sigma_2$ does induce an up-down perspectivity. For example, we have $X + 0 = Z_1 + Y_1$, $X \cap (Y_1 + 0) = Y_1$, and $Z_1 \cap (Y_1 + 0) = 0$. By Lemma~\ref{lem:proj}, this shows that $[Y_1,X] \pers [0,Z_1]$. 
\end{example}


\subsection{Functor categories}\label{sec:fun cats}
We now discuss categories of functors with target in an abelian category. We again refer to~\cite{abelian lecture notes} for more details.

Let $\mathcal{C}$ be a small category, and recall that $\Cat$ is an abelian category. Then the category $\Fun(\mathcal{C},\Cat)$ of covariant functors from $\mathcal{C}$ to $\Cat$ is once again abelian. This is sometimes referred to as the category of $\Cat$-representations of $\mathcal{C}$.

We recall that colimits and limits can be computed ``pointwise'' in functor categories, in the following sense. Let $\mathfrak{D} = (\mathcal{F},\mathcal{N})$ be a diagram in $\Fun(\mathcal{C},\Cat$) (with set of objects $\mathcal{F}$ and set of morphisms $\mathcal{N}$). For $X$ an object of $\Cat$, define 
$$C(X) := \colim(\{M(X)|M \in \mathcal{F}\},\{\eta_X|\eta \in \mathcal{N}\})$$
if this colimit exists. If $C(X)$ exists for all objects $X$, then for each morphism $X\xrightarrow{f} Y$ in $\mathcal{C}$, there is a natural morphism $C(X)\xrightarrow{C(f)}C(Y)$. This makes $C$ into a functor from $\mathcal{D}$ to $\Cat$, and we have $C = \colim \mathfrak D$. The limit of $\mathfrak{D}$ is computed similarly (when all of the pointwise limits exist).

To conclude this section, we recall known results about the category $\Fun(\mathbb{R},\veck)$ where $\field$ is a field and $\mathbb{R}$ is the category with objects the real numbers and precisely one morphism $x \rightarrow y$ if and only if $x \leq y$. The category $\Fun(\mathbb{R},\veck)$ is often referred to as the category of  \textdef{pointwise finite-dimensional ($\mathbb{R}$-)persistence modules}. It serves as the basis for many of our examples, as well as a large part of our motivation.

We now review results about certain indecomposable objects in $\Fun(\mathbb{R},\veck)$ and the morphisms between them. These results are discussed, and their proofs outlined, by Gabriel and Ro\u{\i}ter \cite[Section~3.6]{GR}, and were proven explicitly by Crawley-Boevey~\cite{C-B}. Much of the theory also extends to the case where $\mathbb{R}$ is given a partial order in place of its standard order. See~\cite{BC-B,GR,Continuous A 1}. In Section~\ref{sec:ITreps}, we will also consider a certain subcategory of $\Fun(\mathbb{R},\veck)$ which was first studied by Igusa and Todorov~\cite{IT15}.

For $M \in \Fun(\mathbb{R},\veck)$ and $x \leq y$ in $\mathbb{R}$, we denote by $M(x,y)$ the result of applying $M$ to the unique morphism $x\rightarrow y$.

For every (open, closed, or half-open) interval $I\subseteq\mathbb{R}$ and for all $x,y \in \mathbb{R}$ with $x \leq y$, denote
        \begin{equation}\label{eqn:intIndec} M_I(x) = \begin{cases}\field & x \in I\\0& x \notin I\end{cases} \qquad\qquad M_I(x,y) =
        \begin{cases}1_{\field} &x\leq y \in I\\0 & \text{otherwise}.\end{cases}\end{equation}
        The functor $M_I$ is often referred to as the \textdef{interval indecomposable representation} associated to $I$. Up to isomorphism, these are precisely the indecomposable objects in $\Fun(\mathbb{R},\veck)$. Moreover, for two intervals $I,J$, there is a monomorphism $M_I\rightarrow M_J$ in $(\mathbb{R},\veck)$ if and only if (a) $I \subseteq J$, (b) the right endpoint of $I$ is the same as the right endpoint of $J$, and (c) the intervals are either both open on the right or both closed on the right. Moreover, if these conditions are met then this monomorphism is unique up to scalar multiplication.
        
Finally, let $\Delta$ be a set of intervals in $\mathbb{R}$ which is totally ordered with respect to the relation $I\preceq J$ if there is a monomorphism $M_I\rightarrow M_J$. Then we have
\begin{equation}\label{eqn:intLims}\bigvee_{I \in \Delta}M_I = M_{\left(\bigcup_{I \in \Delta} I\right)}\qquad\qquad \bigwedge_{I \in \Delta}M_I = M_{\left(\bigcap_{I \in \Delta} I\right)}.\end{equation}

\begin{remark}\label{rem:sums}
    Note that Equation~\ref{eqn:intLims} holds even if $\Delta$ is infinite. In particular, $\bigvee_{I \in \Delta} M_I$, which may reasonably also be denoted $\sum_{I \in \Delta} M_I$, exists in the category $\Fun(\mathbb{R},\veck)$, even though $\bigoplus_{I \in \Delta} M_I$ may not. As an explicit example, take $\Delta = \{[a,1] \mid 0 \leq a \leq 1\}$. Then one can form $\bigoplus_{I \in \Delta} M_I$ in the category of functors from $\mathbb{R}$ to \emph{arbitrary} $\field$-vector spaces, but this functor will not be pointwise finite-dimensional. On the other hand, one has $\bigvee_{I \in \Delta} M_I = M_{[0,1]}$, which is pointwise finite-dimensional.
\end{remark}

\section{Chains and Composition Series}\label{sec:subobjectChains}

The purpose of this section is to introduce composition series and ``prime chains'' in lattices which do not have finite length. In Section~\ref{sec:subobjectLims}, we study completeness and cocompleteness properties on chains. In Section~\ref{sec:compSeries}, we use these properties to give formulate generalized definitions of composition series, filtrations, and ``prime chains''. Finally, in section~\ref{sec:existence}, we prove Theorem~\ref{thm:intro:compExists}, which shows that, in any bounded complete lattice, prime chains (or composition series) and maximal chains coincide.

 Unless otherwise stated, our conventions in this section are that $\L = (\L,\leq)$ denotes a bounded lattice, that $\Cat$ denotes a well-powered abelian category, and that $X$ denotes an object in $\Cat$. Readers interested only in applications to abelian categories may freely take $\L = \Sub(X)$.


\subsection{Completeness and Cocompleteness of Chains}\label{sec:subobjectLims}

In this section, we study completeness and cocompleteness properties for chains in $\L$. We recall that a subset $\Delta \subseteq \L$ is called a \textdef{chain} if it is totally ordered under the partial order inherited from $\L$. We will sometimes refer to a chain in $\Sub(X)$ as a \textdef{subobject chain}. In any case, we say a chain $\Delta$ is \textdef{spanning} if it contains both $\hat{0}$ and $\hat{1}$.

\begin{remark}
    In general, our framework will require that each chain in $\L$ admits a supremum and an infimum. This is well-known to be equivalent to the assumption that $\L$ is complete (and thus also cocomplete). For convenience, we give a proof of this fact in the appendix (Proposition~\ref{prop:complete}).
\end{remark}

As many of our examples come from objects in abelian categories, we state the following definition.

\begin{definition}\label{def:subobjectComplete}
    Let $X$ be an object in $\mathcal{A}$. We say that $X$ is \textdef{subobject complete} if the lattice $\Sub(X)$ is complete.
\end{definition}

\begin{remark}\
    \begin{enumerate}
        \item If $\mathcal{A}$ has arbitrary coproducts, then $\Sub(X)$ is complete. Moreover, in this case we know that $\bigvee \Delta = \sum_{Y \in \Delta} Y$ and $\bigwedge \Delta = \bigcap_{Y \in \Delta} Y$ for any $\Delta \subseteq \mathcal{A}$. On the other hand, even if $\mathcal{A}$ does not have arbitrary coproducts, it is still possible that the lattice $\Sub(X)$ will be complete. See Remark~\ref{rem:sums} as an example. Because of this, we prefer to use the notations $\bigvee \Delta$ and $\bigwedge \Delta$ in place of $\sum \Delta$ and $\bigcap \Delta$, respectively.
        \item We defer discussion regarding the relationship between subobject completeness the existence of categorical (co)limits to Section~\ref{sec:limits_colimits}.
        \item Observe that, by Proposition~\ref{prop:subQuotients}, 
        subobject-completeness is a self-dual notion. That is, $X$ is subobject complete in $\Cat$ if and only if it is subobject complete in $\Cat^{op}$. See also Remark~\ref{rem:selfDual}
    \end{enumerate}
\end{remark}

Before proceeding, we fix the following notational convention. Suppose $\Delta' \subseteq \Delta \subseteq \L$. When we write $\bigvee \Delta'$ (resp. $\bigwedge \Delta'$), we will always mean the supremum (respectively infimum) of $\Delta'$ computed in $\L$. Indeed, Example~\ref{ex:sup_inf} below shows that $\Delta'$ may admit a supremum or infimum in $\Delta$ which differs from the one computed in $\L$. If we wish to refer to such an object, we will always refer to it in words as the supremum (respectively infimum) of $\Delta'$ in $\Delta$.

\begin{example}\label{ex:sup_inf}
    Consider the lattice $\L = \N \cup \{\alpha,\beta\}$, where we set $n < \alpha < \beta$ for all $n \in \N$. Let $\Delta = \N \cup \{\beta\}$ and $\Delta' = \N$. Then the supremum of $\Delta'$ in $\L$ is $\bigvee \Delta' = \alpha \notin \Delta'$. On the other hand, $\beta$ is the supremum of $\Delta'$ in $\Delta$. Moreover, if we replaced the element $\alpha$ with a copy of the opposite poset of $\mathbb{N}$, we would have that $\bigvee \Delta'$ does not exist while $\beta$ is still the supremum of $\Delta'$ in $\Delta$.
\end{example}

We are now prepared to state the main definition of this section.

\begin{definition}\label{def:bicompleteChain}
    Let $\emptyset \neq \Delta \subseteq \L$ be a nonempty subset.
    \begin{enumerate}
        \item We say that $\Delta$ is \textdef{complete} if for any $\emptyset \neq \Delta' \subseteq \Delta$, the supremum $\bigvee \Delta'$ exists (in $\mathcal{L})$ and is an element of $\Delta$.
        \item We say that $\Delta$ is \textdef{cocomplete} if for any $\emptyset \neq \Delta' \subseteq \Delta$, the infimum $\bigwedge \Delta'$ exists (in $\mathcal{L})$ and is an element of $\Delta$.
        \item We say that $\Delta$ is \textdef{bicomplete} if it is both complete and cocomplete.
    \end{enumerate}
\end{definition}

We emphasize that with this definition, the subset $\Delta$ in Example~\ref{ex:sup_inf} is \emph{not} complete, even though every nonempty subset of $\Delta$ admits a supremum in $\Delta$. On the other hand, this subset is cocomplete since it is well-ordered. In particular, when $\Delta \neq \mathcal{L}$, the definitions of complete and cocomplete do not necessarily coincide.

An immediate consequence of Definition~\ref{def:bicompleteChain} is the following.

\begin{proposition}\label{prop:bicompleteIsLattice}\
    \begin{enumerate}
        \item Suppose $\mathcal{L}$ is complete, and let $\emptyset \neq \Delta \subseteq \L$ be bicomplete. Then $\Delta$ is a complete bounded sublattice of $\L$. In particular,  if $\L$ is modular then $\Delta$ is a complete bounded modular lattice.
        \item Let $X$ be a subobject complete object in $\Cat$. Then every subquotient of $X$ is subobject bicomplete.
    \end{enumerate}
\end{proposition}

\begin{proof}
    It suffices to prove (1), as (2) will then follow from (1) and Proposition~\ref{prop:subQuotients}. Thus suppose $\L$ is complete and let $\emptyset\neq \Delta \subseteq \L$ be bicomplete. It is clear from the definition that $\Delta$ is a complete sublattice of $\L$. Moreover, by assumption we have $\bigvee \Delta, \bigwedge \Delta \in \Delta$, which means $\L$ is bounded.
\end{proof}

We conclude this section with several examples.

\begin{example}\label{ex:subobjectChains}\
\begin{enumerate}
    \item If $|\Delta| < \infty$, then $\Delta$ is bicomplete.
    \item If $\Delta$ is a segment, then $\Delta$ is bicomplete.
    \item If $\mathcal{A}$ admits arbitrary coproducts, then every object of $\mathcal{A}$ is subobject complete. Indeed, in this case we recall that supremums coincide with sums and infimums correspond with intersections.
    \item Consider $\mathbb{Z}$ as a $\mathbb{Z}$-module. For $p$ a prime, define
        $$\Delta_p = \{0\mathbb{Z}\} \cup \left\{p^\alpha \mathbb{Z}\mid \alpha \in \mathbb{N}\right\} \subseteq \Sub(\mathbb{Z}).$$
        with the natural inclusion maps. Then $\Delta_p$ is a bicomplete subobject chain of $\mathbb{Z}$.
    \item Consider the category $\Fun(\mathbb{R},\veck)$ for some field $\field$.
        We define three subobject chains of $M_{(0,1)}$:
        \begin{eqnarray*}
            \Delta_1 &=& \left\{M_{(a,1)} \mid a \in [0,1)\right\}\cup\{0\}\\
            \Delta_2 &=& \left\{M_{[a,1)} \mid a \in (0,1)\right\}\cup\left\{0,M_{(0,1)}\right\}\\
            \Delta_3 &=& \Delta_1\cup\Delta_2.
        \end{eqnarray*}
        (See Equation~\ref{eqn:intIndec} for an explanation of the notation $M_{(a,1)}$ etc.)
        We note that as posets, $\Delta_1$ and $\Delta_2$ are isomorphic to $[0,1]$ and $\Delta_3$ is isomorphic to the lexicographical order on $[0,1]\times\{0,1\}$ without the elements $(0,0)$ and $(1,1)$.
        
        Now for $i \in \{1,2,3\}$ and $\Delta'\subseteq\Delta_i$, let
        $$C = \bigcup_{M_I \in \Delta'}I,\qquad L = \bigcap_{M_I \in \Delta'}I$$
        and note that $\colim\Delta' = M_C$ and $\lim\Delta' = M_L$ (see Equation~\ref{eqn:intLims}). Since the intervals defining $\Delta_1$ are closed under unions but not intersections, this means $\Delta_1$ is a complete subobject chain of $M_{(0,1)}$ which is not cocomplete. Likewise, $\Delta_2$ is a cocomplete subobject chain of $M_{(0,1)}$ which is not complete, and $\Delta_3$ is a bicomplete subobject chain of $M_{(0,1)}$.
    
\end{enumerate}
\end{example}

\subsection{Subfactor Multisets, Successive Subsegments, and Composition Series}\label{sec:compSeries}

In this section, we generalize the definition of a composition series to include subobject chains of arbitrary cardinality. We continue to primarily work in the framework of bounded lattices (in which case we use the term ``prime chain'' in place of ``composition series''), but most of our examples will come from objects in abelian categories.

As a starting point, this requires us to precisely describe the (successive) subsegments of a chain. To that end, we fix the following notation.

\begin{notation}\label{not:pred}
    Let $\Delta$ be a subset of $\L$. For $y \in \Delta$, we denote by
    \begin{eqnarray*}
        y^-_\Delta &:=& \bigvee\{z \in \Delta \mid  z \lneq y\} \\
        y^+_\Delta &:=& \bigwedge\{z \in \Delta \mid  y \lneq z\}
    \end{eqnarray*}
    when these supremums and infimums exist. In particular if $\hat{0} \in \Delta$, we have $\hat{0}^-_{\Delta} = \hat{0}$ and if $\hat{1} \in \Delta$, we have $\hat{1}^+_\Delta = \hat{1}$. When the subset $\Delta$ can be inferred, we will sometimes write $y^-$ for $y^-_\Delta$ and $y^+$ for~$y^+_\Delta$.
\end{notation}

One can consider $y_\Delta^-$ and $y_\Delta^+$ as the ``predecessor'' and ``successor'' of $y$ in $\Delta$, respectively. For example, suppose $\Delta$ is bicomplete (and so $y^-_\Delta$ and $y^+_\Delta$ both exist for every $y \in \Delta$). If $\Delta$ is a continuous well-ordered chain, then $y^+_\Delta$ will truly be the successor of $y$ in $\Delta$. (This is the approach used to prove ``Jordan--H\"older-like'' and ``Schreier-like'' theorems for well-ordered subobject chains in \cite{Transfinite}.) The case where $y^+_\Delta = y$ similarly corresponds to when $y$ has no successor. Example~\ref{ex:factorMultisets}(3) shows an example of this.

The following definition extends this generalization of predecessors and successors to a generalization of what is meant by the successive subsegments of an arbitrary chain. Recall that a multiset is allowed to contain multiple distinct copies of the same element.

 \begin{definition}\label{def:factorMultiset}\
        \begin{enumerate}
        \item Let $\Delta \subseteq \L$ be a complete chain. Then the \textdef{lower subsegment set} of $\Delta$ is
        $$\seg^-(\Delta) := \left\{[y^-_\Delta,y] \mid  y \in \Delta, y \neq y^-_\Delta\right\}$$
        and the \textdef{lower $\pi$-subsegment multiset} is $\seg_\pi^-(\Delta) := \left\{\mathcal{S}_\pi \mid \mathcal{S} \in \seg^-(\Delta)\right\}$.
        \item Let $\Delta \subseteq \L$ be a cocomplete chain. Then the \textdef{upper subsegment set} of $\Delta$ is
        $$\seg^+(\Delta) := \{[y, y^+_\Delta] \mid y \in \Delta, y \neq y^+_\Delta\}$$
        and the \textdef{upper $\pi$-subsegment multiset} is $\seg_\pi^+(\Delta) := \{\mathcal{S}_\pi \mid \mathcal{S} \in \seg^+(\Delta)\}.$
    \end{enumerate}
\end{definition}

\begin{remark}
    In the proof of Lemma~\ref{lem:unique_proj}, we will show that if $\S$ and $\T$ are distinct segments in $\cov^-(\Delta)$ or $\cov^+(\Delta)$, then there does not exist a segment $\mathcal{R} \subseteq \L$ such that $\S \pers \mathcal{R}$ and $\mathcal{R} \pers \T$. It is not, however, immediately clear whether it is possible to have $\S \proj \T$. This is why we consider $\cov_\pi(\L)$ as a multiset rather than as a set.
\end{remark}

Similarly, in the case $\L = \Sub(X)$ we can define the following.

\begin{definition}\label{def:factorMultiset2}\
    \begin{enumerate}
        \item Let $X$ be an object of $\Cat$ and let $\Delta$ be a complete subobject chain of $X$. Then the \textdef{lower subfactor multiset} of $\Delta$ is
        $$\fac^-(\Delta) := \left\{Y/Y^-_\Delta \mid [Y^-_\Delta, Y] \in \seg^-(\Delta)\right\}.$$
        \item Let $X$ be an object of $\Cat$ and let $\Delta$ be a cocomplete subobject chain of $X$. Then the \textdef{upper subfactor multiset} of $\Delta$ is
        $$\fac^+(\Delta) := \left\{Y^+_\Delta/Y \mid [Y, Y^+_\Delta] \in \seg^+(\Delta)\right\}.$$
    \end{enumerate}
\end{definition}

\begin{example}\label{ex:factorMultisets}\
    \begin{enumerate}
        \item Let $X$ have finite length and let $\Delta$ be a composition series of $X$. Then $\Delta$ is bicomplete and $\fac^-(\Delta) = \fac^+(\Delta)$ is the multiset of composition factors of $X$ (with multiplicity).
        \item For $p$ a prime, let $\Delta_p$ be as in Example~\ref{ex:subobjectChains}(2). Then $\fac^+(\Delta_p) = \fac^-(\Delta_p)$ consists of countably many copies of $\mathbb{Z}/p\mathbb{Z}$.
        \item Let $\Delta_1,\Delta_2$, and $\Delta_3$ be as in Example~\ref{ex:subobjectChains}(3).
         Now for $a \in (0,1)$, we have
         \begin{eqnarray*}
            \left(M_{(a,1)}\right)^-_{\Delta_1} &=& \bigvee\left\{M_{(b,1)} \mid  b \in (a,1)\right\} = M_{(a,1)}\\
            \left(M_{(a,1)}\right)^-_{\Delta_3} &=& \bigvee\left\{M_{(b,1)},M_{[b,1)} \mid b \in (a,1)\right\} = M_{(a,1)}\\
            \left(M_{[a,1)}\right)^-_{\Delta_3} &=& \bigvee
            \left(\left\{M_{(b,1)},M_{[b,1)} \mid b \in (a,1)\right\}\cup\left\{M_{(a,1)}\right\}\right) = M_{(a,1)}
         \end{eqnarray*}
        This shows that $\fac^-\Delta_1 = \emptyset$. and $\fac^-\Delta_3 = \{M_{[a,a]} \mid  a \in (0,1)\}$. A similar argument shows that $\fac^+\Delta_2 = \emptyset$ and $\fac^+\Delta_3 = \fac^-\Delta_3$.
    \end{enumerate}
\end{example}

We note that for those subobject chains in Example~\ref{ex:factorMultisets} which are bicomplete, the upper and lower subfactor multisets coincide. The following shows that this is completely general.

\begin{proposition}\label{prop:upperLowerFactors}
    Let $\Delta \subseteq \L$ be a bicomplete chain.
    \begin{enumerate}
        \item If $y \in \Delta$ and $y \neq y^-$, then $(y^-)^+ = y$.
        \item If $y \in \Delta$ and $y\neq y^+$, then $(y^+)^- = y$.
        \item There is an equality of sets $\seg^-(\Delta) = \seg^+(\Delta)$.
        \item If $\L = \Sub(X)$, then there is an equality of multisets $\fac^-(\Delta) = \fac^+(\Delta)$.
    \end{enumerate}
\end{proposition}

\begin{proof}
   (1) Let $\Omega = \{y' \in \Delta \mid y^- \lneq y'\}$. By assumption, we have that $y \in \Omega$. Moreover, if $y'' \in \Delta$ with $y'' \lneq y$, then $y'' \leq y^-$ by definition. We conclude that $y = \bigvee \Omega = (y^-)^+$. The proof of (2) is dual.
   
   (3) Let $[y^-,y] \in \seg^-(\Delta)$. Then $y^- \in \Delta$ by assumption and $[y^-,y] = [y^-,(y^-)^+] \in \seg^+(\Delta)$ by~(1). The reverse inclusion analogously follows from (2).
   
   Finally, (4) is an immediate consequence of (3).
\end{proof}

\begin{notation}
From now on, when the hypotheses of Proposition \ref{prop:upperLowerFactors} are satisfied, we will write $\seg(\Delta)$ instead of $\seg^-(\Delta)$ or $\seg^+(\Delta)$. Analogously, we will write $\fac(\Delta)$ instead of $\fac^-(\Delta)$ or $\fac^+(\Delta)$. We emphasize that while $\seg(\Delta)$ is a set, $\fac(\Delta)$ is a multiset.
\end{notation}

We are now ready to define generalized notions of composition series and filtrations over a subcategory. To help motivate the lattice-theoretical versions, we give the versions for objects in abelian categories first.

\begin{definition}\label{def:filtration}
Let $\mathcal{D}$ be a subcategory of $\Cat$ which is closed under isomorphisms. Let $X$ be an object of $\Cat$ and let $\Delta$ be a spanning subobject chain of $X$.
    \begin{enumerate}
        \item We say that $\Delta$ is a \textdef{lower $\mathcal{D}$-filtration} if $\Delta$ is complete and every object in $\fac^-(\Delta)$ is in $\mathcal{D}$. If in addition every object in $\mathcal{D}$ is simple in $\Cat$, we say that $\Delta$ is a \textdef{lower composition series}.
        \item We say that $\Delta$ is an \textdef{upper $\mathcal{D}$-filtration} if $\Delta$ is cocomplete and every object in $\fac^+(\Delta)$ is in $\mathcal{D}$. If an addition every object in $\mathcal{D}$ is simple in $\Cat$, we say that $\Delta$ is an \textdef{upper composition series}.
        \item We say that $\Delta$ is a \textdef{$\mathcal{D}$-filtration} (respectively a \textdef{composition series}) if $\Delta$ is subobject bicomplete and is a lower $\mathcal{D}$-filtration (respectively a lower composition series).
    \end{enumerate}
\end{definition}

\begin{definition}\label{def:filtration_lattice}
    Let $\mathfrak{S} \subseteq 2^\L$ be a set of segments which is closed under up-down perspectivity, and let $\Delta \subseteq \L$ be a spanning chain.
    \begin{enumerate}
        \item We say that $\Delta$ is a \textdef{lower $\mathfrak{S}$-filtration} if $\Delta$ is complete and $\seg^-(\Delta) \subseteq \mathfrak{S}$. If in addition every segment in $\mathfrak{S}$ is a prime, we say that $\Delta$ is a \textdef{lower prime chain}.
        \item We say that $\Delta$ is an \textdef{upper $\mathfrak{S}$-filtration} if $\Delta$ is cocomplete and $\seg^+(\Delta) \subseteq \mathfrak{S}$. If in addition every segment in $\mathfrak{S}$ is prime, we say that $\Delta$ is an \textdef{upper prime chain}.
        \item We say that $\Delta$ is a \textdef{$\mathfrak{S}$-filtration} (respectively a \textdef{prime chain}) if $\Delta$ is bicomplete and is a lower $\mathfrak{S}$-filtration (respectively a lower prime chain).
    \end{enumerate}
\end{definition}

\begin{remark}
    As an immediate consequence of Proposition~\ref{prop:upperLowerFactors}, we see that a bicomplete subobject chain $\Delta$ is a $\mathcal{D}$-filtration (respectively a composition series) if and only if it is an upper $\mathcal{D}$-filtration (respectively an upper composition series). That is, we could replace both instances of ``lower'' with ``upper'' in Definition~\ref{def:filtration}(3) without actually changing the definitions. The analogous result holds for filtrations and prime chains in $\L$ as well.
\end{remark}

\begin{example}\label{ex:compSeries}
 All of the subobject chains in Example~\ref{ex:factorMultisets} are (lower or upper) composition series. In particular, (finite) composition series in the traditional sense are also composition series under Definition~\ref{def:filtration}. More generally, if $\mathcal{L}$ has finite length, then a subset $\Delta \subseteq \mathcal{L}$ is a prime chain if and only if it is a maximal chain.
\end{example}

\begin{remark}\label{rem:selfDual}
    As a consequence of Proposition~\ref{prop:subQuotients}, all of the notions in this section are self-dual in the following sense. Let $\Delta \subseteq \Sub(X)$, and denote $\nabla := \{X/Y \mid Y \in \Delta\} \subseteq \Sub^{op}(X)$. (Here, $\Sub^{op}(X)$ denotes the category of subobjects of $X$ in $\mathcal{A}^{op}$, or alternatively the category of quotient objects of $X$ in $\mathcal{A}$.) Then $\Delta$ is complete if and only if $\nabla$ is cocomplete and vice versa. Moreover, if $\Delta$ is complete then $\fac^-(\Delta) = \fac^+(\nabla)$. In particular, every lower (respectively upper) $\mathcal{D}$-filtration or composition series $\Delta \subseteq \Sub(X)$ determines an upper (respectively lower) $\mathcal{D}^{op}$-filtration or composition series $\nabla \subseteq \Sub^{op}(X)$ which has the same subfactors, and vice versa.
\end{remark}

Finally, we state the following immediate consequence of the definitions.

\begin{proposition}\label{prop:comp_cover}
    Let $X$ be an object of $\mathcal{A}$, and let $\Delta \subseteq \Sub(X)$ be a subobject chain. Let $\mathcal{D}$ be a subcategory of $\mathcal{A}$ which is closed under isomorphisms, and let $\mathfrak{S} \subseteq 2^\L$ be the set of segments $[Y_0,Y_1] \subseteq \Sub(X)$ for which $Y_1/Y_0 \in \mathcal{D}$. Then $\Delta$ is a (lower or upper) $\mathcal{D}$-filtration if and only if it is a (lower or upper) $\mathfrak{S}$-filtration. In particular, $\Delta$ is a (lower or upper) composition series if and only if it is a (lower or upper) prime chain.
\end{proposition}


\subsection{The Existence of Composition Series and Prime Chains}\label{sec:existence}

We now turn towards determining when a bounded lattice $\L$ admits a prime chain. As we shall see in Theorem~\ref{prop:compExists}, only completeness needs to be assumed in order for the prime chains and maximal chains of $\L$ to coincide.

While we state many of the definitions and results of this section for an arbitrary subset $\Delta \subseteq \L$, we will usually be most interested in the case where $\Delta$ is a chain.

\begin{definition}\label{def:refinement}
    Let $\Delta$ and $\Delta'$ be subsets of $\L$.
    We say that $\Delta'$ is a \textdef{refinement} of $\Delta$ if $\Delta \subseteq \Delta'$ and that $\Delta'$ is a \textdef{proper} refinement of $\Delta$ if $\Delta\subsetneq \Delta'$.
\end{definition}

Note that by definition, a chain $\Delta \subseteq \L$ is maximal if and only if there does not exist a chain $\Delta' \subseteq \L$ which is a proper refinement of $\Delta$.

We now show that when $\L$ is bicomplete (respectively cocomplete, complete), there is a natural way to refine an arbitrary subset of $\L$ into a bicomplete (respectively cocomplete, complete) subset.

\begin{lemma}\label{lem:completions}
    Suppose $\L$ is complete, and let $\{\Delta_\alpha\}_\alpha \subseteq 2^\L$. If each $\Delta_\alpha$ is complete (respectively cocomplete, bicomplete) and $\bigcap_\alpha \Delta_\alpha \neq \emptyset$, then $\bigcap_\alpha \Delta_\alpha$ is complete (respectively cocomplete, bicomplete).
\end{lemma}

\begin{proof}
    We prove the case for each $\Delta_\alpha$ complete, as the other cases follow analogously. Let $\emptyset \neq \Gamma \subseteq \bigcap_\alpha \Delta_\alpha$. The for any $\alpha$, the assumption that $\Delta_\alpha$ is complete means $\bigvee \Gamma \in \Delta_\alpha$. It follows that $\bigvee \Gamma \in \bigcap_\alpha \Delta_\alpha$, and so $\bigcap_\alpha \Delta_\alpha$ is complete.
\end{proof}

Based on Lemma~\ref{lem:completions}, we have the following definition.

\begin{definition}\label{def:bicompletion}
    Suppose $\L$ is complete and let $\emptyset \neq \Delta \subseteq \L$. Let $\mathfrak{B}(\Delta)$ be the set of bicomplete subsets of $\L$ which contain $\Delta$. We define the \textdef{bicompletion} of $\Delta$ to be
    $$\overline{\Delta} := \bigcap_{\Gamma \in \mathfrak{B}(\Delta)} \Gamma.$$
    We define the \textdef{completion} $\overline{\Delta}_\uparrow$ and the \textdef{cocompletion} $\overline{\Delta}_\downarrow$ of $\Delta$ analogously.
\end{definition}

Since $\L$ is complete (and thus bicomplete), we have $\Delta \subseteq \overline{\Delta}$ By Lemma~\ref{lem:completions}, we thus have that $\overline{\Delta}$ is bicomplete and that any bicomplete set containing $\Delta$ must also contain $\overline{\Delta}$. The analogous statements hold for $\overline{\Delta}_\uparrow$ and $\overline{\Delta}_\downarrow$ as well.

In the special case that $\Delta$ is a chain, we can also describe $\overline{\Delta}$ more explicitly.

\begin{proposition}\label{prop:bicompletion}
    Suppose $\L$ is complete, and let $\Delta$ be a chain in $\L$. Then $\overline{\Delta}$, $\overline{\Delta}_\uparrow$, and $\overline{\Delta}_\downarrow$ are also chains, and are given by
    \begin{eqnarray*}
        \overline{\Delta}_{\phantom{\uparrow}} &=& \left\{\bigvee \Gamma, \bigwedge \Gamma \ \middle| \ \Gamma \subseteq \Delta\right\},\\
        \overline{\Delta}_\uparrow &=& \left\{\bigvee \Gamma \ \middle| \ \Gamma \subseteq \Delta\right\},\\
        \overline{\Delta}_\downarrow &=& \left\{\bigwedge \Gamma \ \middle| \ \Gamma \subseteq \Delta\right\}.
    \end{eqnarray*}
\end{proposition}

\begin{proof}
    We prove the result only for $\overline{\Delta}$, as the other cases are similar.
    Let $\Theta = \left\{\bigvee \Gamma, \bigwedge \Gamma \ \middle| \ \Gamma \subseteq \Delta\right\}$. It is clear from the construction that $\Theta \subseteq \overline{\Delta}$. Thus it suffices to show that $\Theta$ is a bicomplete chain.
    
    To see that $\Theta$ is a chain, we note that by construction every element of $\Theta$ is either the supremum or infimum of some chain $\Gamma \subseteq \Delta$. Thus let $\Gamma, \Gamma'$ be two such chains.
    
    We first show that $\bigvee \Gamma$ and $\bigvee \Gamma'$ are comparable. Indeed, if there exists $y \in \Gamma$ such that $y' \leq y$ for all $y' \in \Gamma'$, then we have $\bigvee \Gamma' \leq y \leq \bigvee \Gamma$. Otherwise, by symmetry, we may assume that for all $y \in \Gamma$, there exists $y' \in \Gamma'$ with $y \subseteq y'$ and vice versa. In this case, it is straightforward to see that $\bigvee \Gamma = \bigvee \Gamma'$.
    
    The argument that $\bigwedge \Gamma$ and $\bigwedge \Gamma'$ are comparable is analogous. Therefore consider $\bigvee \Gamma$ and $\bigwedge \Gamma'$. Now if there exist $y' \in \Gamma'$ and $y \in \Gamma$ with $y' \leq y$, then $\bigwedge \Gamma' \leq y' \leq y \leq \bigvee \Gamma$ and we are done. Otherwise, for all $y' \in \Gamma'$ and $y \in \Gamma$, we have $y \leq y'$. It follows that $\bigvee \Gamma \leq \bigwedge \Gamma'$. We conclude that $\Theta$ is a chain.
    
    It remains only to show that $\Theta$ is bicomplete. Let $\widetilde{\Gamma} \subseteq \Theta$. Then $\bigvee \widetilde{\Gamma}$ and $\bigwedge \widetilde{\Gamma'}$ exist since $\L$ is complete. Now denote
    $$\Gamma := \left\{y \in \Delta \ \middle| \ \bigwedge \widetilde{\Gamma} \leq y \leq \bigvee \widetilde{\Gamma}\right\}.$$
    
    First suppose that $\Gamma = \emptyset$. We claim that, in this case, $|\widetilde{\Gamma}| \leq 2$. Indeed, suppose there exist subsets $\Gamma_1, \Gamma_2 \subseteq \Delta$ with $\bigvee \Gamma_1, \bigvee \Gamma_2 \in \widetilde{\Gamma}$. Since $\Theta$ is a chain, we assume without loss of generality that $\bigvee \Gamma_1 \leq \bigvee \Gamma_2$. Now if $\bigvee \Gamma_1 < \bigvee \Gamma_2$, the fact that $\Delta$ is a chain means there exists $y \in \Gamma_2 \setminus \Gamma_1$ such that $\bigvee \Gamma_1 \leq y \leq \bigvee\Gamma_2$. On the other hand, we have that $\bigvee \Gamma_2 \leq \bigvee \widetilde{\Gamma}$ and $\bigvee \Gamma_1 \geq \bigwedge \widetilde{\Gamma}$, which contradicts that $\Gamma = \emptyset$. We conclude that $\bigvee \Gamma_1 = \bigvee \Gamma_2$. Analogous reasoning shows that if there exist subsets $\Gamma_1,\Gamma_2 \subseteq \Delta$ with $\bigwedge \Gamma_1,\bigwedge \Gamma_2 \in \widetilde{\Gamma}$, then $\bigwedge \Gamma_1 = \bigwedge \Gamma_2$. We conclude that $|\widetilde{\Gamma}| \leq 2$, as claimed. It follows that $\bigvee \widetilde{\Gamma},\bigwedge\widetilde{\Gamma} \in \widetilde{\Gamma} \subseteq \Theta$.
    
    If $\Gamma \neq \emptyset$, let $\widetilde{\Gamma}' = \left\{y \in \widetilde{\Gamma} \ \middle| \  y \leq \bigwedge \Gamma\right\}$. By analogous reasoning to before, we see that $|\widetilde{\Gamma}'| \leq 2$. If $\widetilde{\Gamma}' = \emptyset$, then we have $\bigwedge \widetilde{\Gamma} = \bigwedge\Gamma \in \Theta$. Otherwise, we have that $\bigwedge \widetilde{\Gamma} \in \widetilde{\Gamma}' \subseteq \Theta$. The argument that $\bigvee \widetilde{\Gamma} \in \Theta$ is analogous.
\end{proof}

\begin{remark}
    We note that the formulas describing $\overline{\Delta}_\uparrow$ and $\overline{\Delta}_\downarrow$ given in Proposition~\ref{prop:bicompletion} still hold when $\Delta$ is not a chain. On the other hand, the assumption that $\Delta$ is a chain is explicitly used to prove the result for $\overline{\Delta}$. This, and the fact that the remainder of this section focuses only on chains, is our justification for stating Proposition~\ref{prop:bicompletion} at this level of generality.
\end{remark}

\begin{example}
    In Examples~\ref{ex:subobjectChains}(3) and~\ref{ex:factorMultisets}(3), we have that $\Delta_3$ is the bicompletion of both $\Delta_1$ and $\Delta_2$.
\end{example}

We conclude this section by proving our first main theorem.

\begin{theorem}[Theorem~\ref{thm:intro:compExists}]\label{prop:compExists}\
    \begin{enumerate}
        \item Let $\L$ be a complete bounded lattice. Then the prime chains of $\L$ and the maximal chains in $\L$ coincide. In particular, for any chain $\Delta \subseteq \L$, there exists a prime chain $\Delta'$ which is a refinement of $\L$.
        \item Let $X$ be an object of $\Cat$ and suppose $X$ is subobject complete. Then the composition series of $X$ and the maximal chains in $\Sub(X)$ coincide. In particular, for any subobject chain $\Delta \subseteq \Sub(X)$, there exists a composition series $\Delta'$ of $X$ which is a refinement of $\Delta$.
    \end{enumerate}
\end{theorem}

\begin{proof}
    It suffices to prove (1), as (2) will then follow from Proposition~\ref{prop:comp_cover}.
    
    First suppose $\Delta$ is a prime chain, and let $\Gamma \subseteq \L$ be a chain which refines $\Delta$. Let $z \in \Gamma$. We will show that $z \in \Delta$, and so $\Delta = \Gamma$.
    
    Define $\Omega^- = \{y \in \Delta \mid  y \lneq z\}$ and $\Omega^+ = \{y \in \Delta \mid  z \lneq y\}$. Denote $y_0 = \bigvee \Omega^-$ and $y_1 = \bigwedge \Omega^+$, which exist by the assumption that $\L$ is complete (and thus also bicomplete). Since prime chains are bicomplete, we then have that $y_0,y_1 \in \Delta$ and that $(y_1)_\Delta^- = y_0$. This means that either the segment $[y_0,y_1]$ is prime or that $y_0 = y_1$. In either case, this implies that either $z = y_0$ or $z = y_1$, and so $z \in \Delta.$
    
    Now suppose that $\Delta$ is a maximal chain. By Proposition~\ref{prop:bicompletion}, this means that $\Delta = \overline{\Delta}$; i.e., that $\Delta$ is bicomplete. Now let $y \in \Delta$ with $y \neq y^-_\Delta$ and choose some $z\in [y^-_\Delta,y]$. Then $\Delta \cup \{z\}$ is a chain which is a refinement of $\Delta$, and so $z \in \Delta$ by the maximality assumption. It follows that either $z = y^-_\Delta$ or $z = y$. Therefore, the segment $[y^-_\Delta,y]$ is prime.
    
    We have thus shown that the prime chains of $\L$ and the maximal chains of $\L$ coincide. The fact that an arbitrary chain can be refined into a prime chain is then a well-known consequence of Zorn's lemma.
\end{proof}

\begin{remark}\label{rem:refineableCompSeries}
    We note that Theorem~\ref{prop:compExists} immediately implies that any maximal chain is both a lower and upper prime chain (or composition series). However, the ``one-sided'' versions of composition series need not be maximal. Indeed, the subobject chains $\Delta_1$ and $\Delta_2$ in Example~\ref{ex:subobjectChains}(3) are lower and upper composition series, respectively, but both admit $\Delta_1\cup \Delta_2$ as a proper refinement.
\end{remark}

Building upon Remark~\ref{rem:refineableCompSeries}, we conclude with the following example, which shows that the behavior of ``one-sided'' prime chains can be much more pathological than that of their ``two-sided'' counterparts.

\begin{example}\label{ex:refineableCompSeries}
    Let $\L' = [0,1/2) \cup (1/2,1] \subseteq \mathbb{R}$ with the standard partial order. Let $\L = \L' \cup \{a, b, c, d\}$, where we set $t \leq a$ for $t \in [0,1/2)$, $d \leq t$ for $t \in (1/2,1]$, $a \leq b \leq d$ and $a \leq c \leq d$. (We suppose that $b$ and $c$ are incomparable.) It is straightforward to show that $\L$ is a complete bounded (modular) lattice. Then $\Delta = [0,1/2) \cup \{a\} \cup (1/2,1] \subseteq \L$ is a lower prime chain. On the other hand, the bicompletion $\overline{\Delta} = \Delta \cup \{d\}$ is not a (lower) prime chain. Moreover, $\Delta \cup \{b,d\}$ and $\Delta \cup \{c,d\}$ are both (lower) prime chains which refine $\Delta$, but $\seg(\Delta \cup \{b,d\}) = \{[a,b], [b,d]\} \neq \{[a,c], [c,d]\} = \seg(\Delta \cup \{c,d\})$.
\end{example}


\section{(Weakly) Jordan--H\"older--Schreier Lattices and Objects}\label{sec:JHSobjects}
In this section, we address the uniqueness of prime chains and composition series. Throughout this section, we assume that $\mathcal{L}$ is a complete bounded modular lattice and that a chosen object $X$ in the (well-powered) abelian category $\mathcal{A}$ is subobject complete.

\subsection{Axioms and basic properties}\label{sec:JHSobjectsIntro}

In this section, we introduce the set of axioms (Definition~\ref{JHSobjectsLattice}) under which we will prove our main results. We also extend the notions of up-down perspectivity and subfactor equivalence from Section~\ref{sec:background} include lattices and objects which are not of finite length.

\begin{definition}\label{def:factorEquivalenceLattice}
    Let $\Delta, \Gamma \subseteq \L$ be complete chains. We say that $\Delta$ and $\Gamma$ are \textdef{up-down perspective} if there exists a bijection
        $$\Phi: \left\{y \in \Delta \ \middle| \ y \neq y^-_\Delta\right\} \rightarrow \left\{z \in \Gamma \ \middle| \ z \neq z^-_\Gamma\right\}$$
        such that $\left[y^-_\Delta,y\right] \pers \left[\Phi(y)^-_\Gamma,\Phi(y)\right]$ for all $y$ in the domain of $\Phi$. In this case, we say that the bijection $\Phi$ \textdef{induces an up-down perspectivity}. We define up-down perspectivity for cocomplete and bicomplete chains analogously.
\end{definition}
 
\begin{remark}
    It is an immediate consequence of Proposition~\ref{prop:upperLowerFactors} that two bicomplete chains are up-down perspective as complete chains if and only if they are up-down perspective as cocomplete chains.
\end{remark}

\begin{remark}
We emphasize that the domain of the bijection $\Phi$ in Definition~\ref{def:factorEquivalenceLattice} is generally not all of $\Delta$, and for example in the cocomplete case consists of only those elements of $\Delta$ which have a ``successor". When $\Delta$ and $\Gamma$ are continuous well-ordered chains (or in particular finite), every element has a ``successor''. Therefore, up-down perspectivity implies that $\Delta$ and $\Gamma$ have the same cardinality in this case. Based solely on the definition, however, it is not clear in general whether chains with different cardinalities can be up-down perspective.
\end{remark}

We collect the following lemma, which will ultimately imply the uniqueness of $\Phi$ in the generalized Jordan--H\"older theorem for modular lattices.

\begin{lemma}\label{lem:unique_proj}
    Let $\Delta, \Gamma \subseteq \L$ be complete (respectively cocomplete, bicomplete) chains, and suppose that $\Delta$ and $\Gamma$ are up-down perspective. Then there is a unique bijection $\Phi$ which induces an up-down perspectivity.
\end{lemma}

\begin{proof}
    We prove the result in the complete case, as the other cases are similar. Suppose $\Phi, \Psi: \left\{y \in \Delta \ \middle| \ y \neq y_\Delta^-\right\} \rightarrow \left\{z \in \Gamma \ \middle| \ z \neq z_\Gamma^-\right\}$ are two bijections which both induce a up-down perspectivities. Suppose for a contradiction that there some $y$ in the domain of $\Phi$ and $\Psi$ such that $\Phi(y) \neq \Psi(y)$, and denote $\mathcal{Y} = \left[y^-_\Delta,y\right]$. Since $\Gamma$ is a chain, suppose without loss of generality that $\Phi(y) \leq \Psi(y)_\Gamma^-$. Since $\pi_\mathcal{Y}$ is order preserving, this implies that $\pi_\mathcal{Y}(\Phi(y)) \leq \pi_\mathcal{Y}\left(\Psi(y)_\Gamma^-\right)$. On the other hand, since $\Phi$ and $\Psi$ both induce up-down perspectivities, we must have $\pi_\mathcal{Y}(\Phi(y)) = y$ and $\pi_\mathcal{Y}\left(\Psi(y)_\Gamma^-\right) = y_\Delta^-$. Thus $y_\Delta^- = y$, which contradicts that $y$ is in the domain of $\Phi$ and $\Psi$. 
\end{proof}

For objects in abelian categories, we also consider the following definition.

\begin{definition}\label{def:factorEquivalence}
Let $X$ be an object of $\Cat$ and let $\Delta$ and $\Gamma$ be complete subobject chains of $X$. We say that $\Delta$ and $\Gamma$ are \textdef{subfactor equivalent}
if there exists a bijection
    $$\Phi: \left\{Y \in \Delta \ \middle| \ Y \neq Y^-_\Delta\right\} \rightarrow \left\{Z \in \Gamma \ \middle| \ Z \neq Z^-_\Gamma\right\}$$
so that $Y/Y^-_\Delta \cong \Phi(Y)/\Phi(Y)^-_\Gamma$ for all $Y$. We define subfactor equivalence for cocomplete and bicomplete subobject chains analogously.
\end{definition}

\begin{remark}
    It follows from Proposition~\ref{prop:subfactor_proj} that if $\Delta$ and $\Gamma$ are up-down perspective (as chains in $\Sub(X)$), then they are also subfactor equivalent (as subobject chains of $X$). The converse, however, does not hold. See Example~\ref{ex:subfact_vs_proj}.
\end{remark}

\begin{remark}\label{rem:selfDual2}
    It follows from Remark~\ref{rem:selfDual} that two chains $\Delta, \Gamma \subseteq \Sub(X)$ are subfactor equivalent if and only if the corresponding chains of quotient objects (seen as subobject chains in $\Sub^{op}(X)$) are also subfactor equivalent.
\end{remark}

Before proceeding, we fix the following 
notational convention.

\begin{notation}\label{nota:wedge_set}
     Given a subset $\Delta \subseteq \L$ and an element $z \in \L$, we denote $z \wedge \Delta = \{z \wedge y \mid y \in \Delta\}$ and $z \vee \Delta = \{z \vee y \mid y \in \Delta\}$. Likewise in the context of abelian categories we denote $Z \cap \Delta = \{Z \cap Y \mid Y \in \Delta\}$ and $Z + \Delta = \{Z + Y \mid Y \in \Delta\}$.
\end{notation}

We now give the set of axioms under which we will prove our ``Jordan--H\"older-like'' and ``Schreier-like theorems''. Readers may refer to Definition~\ref{def:directed} in the appendix for the definitions of upward and downward direct sets.

\begin{samepage}

\begin{definition}\label{JHSobjectsLattice}
    Let $\L$ be a bounded modular lattice.
    \begin{enumerate}
        \item We say that $\L$ is \textbf{(AB5)} if $\L$ is complete and for every (upward) directed set $\Delta \subseteq \L$ and $z \in \L$, the following relation holds:
        $$z \wedge \left(\bigvee \Delta\right) = \bigvee \left(z \wedge\Delta\right).$$
        \item We say that $\L$ is \textbf{(AB5$^*$)} if $\L$ is complete and for every downward directed set $\Delta \subseteq \L$ and $z \in \L$, the following relation holds:
         $$z \vee \left(\bigwedge \Delta\right)= \bigwedge\left(z \vee \Delta\right).$$
        \item We say that $\L$ is \textdef{weakly Jordan--H\"older--Schreier} (or \textdef{weakly JHS}) if $\L$ is both (AB5) and (AB5$^*$).
        \item We say that $\L$ is \textdef{Jordan--H\"older--Schreier} (or \textdef{JHS}) if $\L$ is weakly JHS and either $|\L| = 1$ or there exists at least one prime chain in $\L$.
        \item For P one of the properties defined in (1-4), we say an object $X$ in $\Cat$ is P whenever $\Sub(X)$ is P.
    \end{enumerate}
\end{definition}

\end{samepage}

\begin{remark}\label{rem:moduleJHS3}\
    \begin{enumerate}
        \item Note that, by definition, a nonzero object $X$ in $\Cat$ is JHS if and only if it is weakly JHS and contains a simple subquotient.
        \item It is well-known (see Proposition~\ref{prop:AB5} in the appendix) that $\L$ is (AB5) (respectively (AB5$^*$)) if and only if the condition in (1) (respectively (2)) holds whenever $\Delta$ is a chain.
        \item Recall that if $\Cat$ has arbitrary coproducts, then Grothendieck's axiom (AB5) is equivalent to requiring that $Z \cap (\sum \Delta) = \sum(Z \cap \Delta)$ for all objects $Z$ and directed systems $\Delta$. In particular, if $\Cat = \mathrm{Mod}\text{-}R$ is the category of (left) modules over a ring $R$, then a module $X \in \Cat$ is weakly JHS if and only if it is an ``(AB5$^*$)-module''. See Section~\ref{sec:examples} for additional examples and non-examples and Section~\ref{sec:directSum} for additional discussion of (AB5$^*$) modules.
        \item Recall that a complete, but not necessarily modular, lattice $\L$ is called \textdef{meet-continuous} (also known as upper-continuous) if it satisfies the (AB5) property. Likewise, $\L$ is said to be \textdef{join-continuous} (also known as lower-continuous) if it satisfies the (AB5$^*$) property. We have opted to instead use the (AB5) and (AB5$^*$) terminology to highlight the connection with Grothendieck's axioms.
        \item  We will show in Proposition~\ref{prop:ARcIsWJHSNotJHS} that there exist weakly JHS objects which are not JHS. In particular, this proposition gives an example of an abelian subcategory $\Cat' \subseteq \Cat$ such that $X$ is JHS in $\Cat$ but not in $\Cat'$.
    \end{enumerate}
\end{remark}

The following will be used throughout the remainder of this paper.

\begin{proposition}\label{prop:subquotientJHS}\
    \begin{enumerate}
        \item Suppose $\L$ is weakly JHS, and let $[s_0,s_1] \subseteq \L$ be a segment. Then $[s_0,s_1]$ is weakly JHS with respect to the partial order inherited from $\L$.
        \item Suppose $X$ is a weakly JHS object in $\Cat$. Then every subquotient of $X$ is weakly JHS in $\Cat$.
    \end{enumerate}
\end{proposition}

\begin{proof}
    These are immediate consequences of Propositions~\ref{prop:bicompleteIsLattice} and~\ref{prop:subQuotients}.
\end{proof}

\begin{remark}
We note that if $X$ is JHS, then it is possible a subquotient $Y$ of $X$ is not JHS. For example, let $Y$ be weakly JHS in an abelian category $\mathcal{B}$ and let $Z$ be JHS in an abelian category $\mathcal{C}$. Then $X=Y\oplus Z$ is JHS in $\mathcal{B}\times\mathcal{C}$ and $Y\subob (Y\oplus Z)$.
Alternatively, suppose $X$ is JHS, that a composition series $\Delta$ of $X$ has exactly one subfactor $Y/Y^-_\Delta$, and that $Y^-_\Delta\neq 0$. Then $Y^-_\Delta$ is weakly JHS but not JHS.
\end{remark}

\subsection{Conflations and Jordan--H\"older--Schreier-like Theorems}\label{sec:JHS}
In this section, we formalize the notion of conflating two chains and prove our ``Jordan--H\"older--like theorem'' (Theorem~\ref{thm:intro:Jordan--Holder} in the introduction). The argument is similar to the proof of the classical Jordan--H\"older Theorem given in \cite[Section~I.3]{Lang} and of the well-ordered case in \cite{Transfinite}.

Part of the basis for this argument is the ``Butterfly Lemma'' of Zassenhaus. For convenience, we prove a lattice theoretical statement of this result here.

\begin{lemma}\label{lem:butterflyLattice}
    Let $\L$ be a modular lattice, and let $\S = [s^-,s]$ and $\T = [t^-,t]$ be segments in $\L$. Define $\S' := \pi_\T(\S)$ and $\T' := \pi_\S(\T)$. Then
    \begin{enumerate}
        \item $\T' = [(t^- \wedge s) \vee s^-, (t \vee s^-) \wedge s]$  and $\S' = [(s^- \wedge t) \vee t^-, (s \vee t^-) \wedge t]$ are both segments.
        \item $\S' = \pi_\T(\T')$ and $\T' = \pi_\S(\S')$.
        \item $\S' \pers \T'$.
    \end{enumerate}
\end{lemma}

\begin{proof}
    (1) Let $x \in [(t^- \wedge s) \vee s^-, (t \vee s^-) \wedge s] \subseteq \S$. We will show that $x = \pi_\S\circ \pi_\T(x) \in \T'$. Indeed, we have
    \begin{align*}
        \pi_\S\circ \pi_\T(x) &= (((x \wedge t) \vee t^-) \wedge s) \vee s^- &\\
        &= ((x \wedge t) \vee (t^- \wedge s)) \vee s^- & \text{ because $x \wedge t \leq s$}\\
        &= ((x \vee (t^- \wedge s)) \wedge t) \vee s^- & \text{ because $t^- \wedge s \leq t$}\\
        &= (x \wedge t) \vee s^- & \text{ because $t^- \wedge s \leq x$}\\
        &= x \wedge (t \vee s^-) & \text{because $s^- \leq x$} \\
        &= x & \text{because $x \leq t \vee s^-$.}
    \end{align*}
    We conclude that $x \in \T'$. The result for $\S'$ then follows by symmetry.
    
    (2) Let $x \in \T' \subseteq \S$. Then $\pi_\T(x) \in \S'$ by definition and $x = \pi_\S\circ \pi_\T(x)$ by the proof of (1). We conclude that $x \in \pi_\S(\S')$. Conversely, we have that $\pi_\S(\S') \subseteq \T'$ because $\T' = \pi_\S(\T)$ and $\S' \subseteq \T'$. The proof that $\pi_\T(\T') = \S'$ is symmetric.
    
    (3) Since $\S' \subseteq \T$, we have $\pi_{\S'} = \pi_{\S'}|_\T \circ \pi_\T$. Likewise, we have $\pi_{\T'} = \pi_{\T'}|_\S \circ \pi_\S$. By (2), this means $\pi_{\S'}|_{\T'} = \pi_\T|_{\T'}$ and $\pi_{\T'}|_{\S'} = \pi_{\S}|_{\T'}$. By the proof of (1), we conclude that $\pi_{\S'}|_{\T'}$ and $\pi_{\T'}|_{\S'}$ are inverses. This proves the result.
\end{proof}

As a consequence, we obtain the Butterfly Lemma for modules.

\begin{lemma}[Zassenhaus Butterfly Lemma]\label{lem:butterfly}
    Let $X$ be an object in  $\Cat$ and let $Y,Y^-,Z,Z^-$ be subobjects of $X$ with $Y^- \subob Y$ and $Z^- \subob Z$. Then
    \begin{equation*}\label{eqn:butterfly} \frac{Y^- + (Y\cap Z)}{Y^- + (Y\cap Z^-)} \cong \frac{Z^- + (Z\cap Y)}{Z^- + (Z\cap Y^-)}.\end{equation*}
\end{lemma}

\begin{proof}
    The result follows from applying Lemma~\ref{lem:butterflyLattice} and Proposition~\ref{prop:subfactor_proj} to the intervals $[Y^-,Y]$ and $[Z^-,Z]$ of $\Sub(X)$.
\end{proof}

\begin{remark}
    If $Y = Y^-$ or $Z = Z^-$, then both sides of the equation in Lemma~\ref{lem:butterfly} are zero.
\end{remark}

We now define a procedure to conflate the data of two subobject chains.

\begin{definition}\label{def:conflation}
   Let $\Delta, \Gamma \subseteq \L$ be spanning chains. We the define the \textdef{lower conflation} of $\Delta$ by $\Gamma$ to be
    $$\con^-(\Delta,\Gamma):=\Delta \cup \left\{y^-_\Delta \vee (y\wedge z) \ \middle| \ (y,z) \in \Delta\times\Gamma,y^-_\Delta \text{ exists}\right\}.$$
    Similarly, we define the \textdef{upper conflation} of $\Delta$ by $\Gamma$ to be
    $$\con^+(\Delta,\Gamma):= \Delta \cup \left\{y \vee (y^+_\Delta\wedge z) \ \middle| \ (y,z) \in \Delta\times\Gamma, y^+_\Delta \text{ exists}\right\}.$$
\end{definition}

\begin{samepage}

\begin{example}\label{ex:conflation}\
\begin{enumerate}
    \item Let $\mathbb{Z}_{24} = \{0,1,\ldots,23\}$ be the cyclic group of order 24 (considered in the category of $\mathbb{Z}$-modules). For each $i \in \mathbb{Z}_{24}$, we denote by $\langle i\rangle$ the subgroup generated by $i$. Now consider the (bicomplete) subobject chains
    $$\Delta = \{\langle 0\rangle,\langle 4\rangle,\langle 1\rangle\}\qquad\text{and}\qquad
        \Gamma = \{\langle 0\rangle,\langle 6\rangle,\langle 1\rangle\}
    $$
    of $\mathbb{Z}_{24}$. We then have
    \begin{eqnarray*}
        \con^-(\Delta,\Gamma) = \con^+(\Delta,\Gamma) &=& \{\langle 0\rangle, \langle 12\rangle,\langle 4\rangle,\langle 2\rangle,\langle 1\rangle\}\\
        \con^-(\Gamma,\Delta) = \con^+(\Gamma,\Delta) &=& \{\langle 0\rangle,\langle12\rangle,\langle6\rangle,\langle2\rangle,\langle1\rangle\}.
    \end{eqnarray*}
    In particular, we note that $\con^-(-,-)$ and $\con^+(-,-)$ are not symmetric in their arguments.
    
    \item Consider the subobject chains $\Delta_1$ and $\Delta_2$ in Example~\ref{ex:subobjectChains}(3). Then $\con^-(\Delta_1,\Delta_2) = \Delta_1 = \con^+(\Delta_1,\Delta_2)$, but for different reasons. We have $\con^-(\Delta_1,\Delta_2) = \Delta_1$ because $Y^-_{\Delta_1} = Y$ for all $Y \in \Delta_1$. On the other hand, we have $\con^+(\Delta_1,\Delta_2) = \Delta_1$ because for $Y \in \Delta_1$, $Y^+_{\Delta_1}$ is only an element of $\Delta_1$ if $Y = 0$, in which case $Y^+_{\Delta_1} = Y$.
    
    \item Let $p$ and $q$ be distinct primes and let $\Delta_p$ and $\Delta_q$ be as in Example~\ref{ex:subobjectChains}(2). Note that for $\alpha \in \mathbb{N}$, we have $(p^\alpha\mathbb{Z})^-_{\Delta_p} =  p^{\alpha+1}\mathbb{Z}$. Moreover, for $\alpha,\beta \in \mathbb{N}$, we have
    $$p^{\alpha+1}\mathbb{Z} + \left(p^\alpha\mathbb{Z}\cap q^{\beta}\mathbb{Z}\right) = p^\alpha \mathbb{Z}.$$
    We then have that $\con^-(\Delta_p,\Delta_q) = \Delta_p$. We can also see this as a consequence of Proposition~\ref{prop:compExists}(1) and the fact that $\Delta_p$ is a composition series. See Remark~\ref{rem:CompSeriesRefinements} below.
\end{enumerate}
\end{example}

\end{samepage}

\begin{remark}
    Note that if $\Delta$ is complete, then for $y \in \Delta$ we can write $y = y^-_\Delta \vee (y\wedge \hat{1})$, where $\hat{1} \in \Gamma$ since $\Gamma$ is spanning. In particular, the ``$\Delta\cup$'' in the definition of $\con^-(\Delta,\Gamma)$ becomes redundant in this case. The behavior is similar in the cocomplete case. Nevertheless, it it necessary to include this in the definition in order for Proposition~\ref{prop:conflation} below to hold. Indeed, in Example~\ref{ex:conflation}(2), removing ``$\Delta_1\cup$'' from the definition of $\con^+(\Delta_1,\Delta_2)$ would result in $\con^+(\Delta_1,\Delta_2) = \{0\}$, which is not a refinement of $\Delta_1$.
\end{remark}

We now prove a series of results about the conflation of two chains. From these, we will deduce our ``Schreier-like'' and ``Jordan--H\"older-like'' theorems (Theorems~\ref{thm:schreier} and~\ref{thm:Jordan--Holder}, respectively).

\begin{proposition}\label{prop:conflation}
    Let $\Delta, \Gamma \subseteq \L$ be spanning chains. Then
    $\con^-(\Delta,\Gamma)$ and $\con^+(\Delta,\Gamma)$ are
    spanning chains which are refinements of $\Delta$.
\end{proposition}

\begin{proof}
We prove the result only for $\con^-(\Delta,\Gamma)$. The proof for $\con^+(\Delta,\Gamma)$ is similar.

    Since $\Delta \subseteq \con^-(\Delta,\Gamma)$, it is clear that $\con^-(\Delta,\Gamma)$ is spanning. Let $w, w' \in \con^-(\Delta,\Gamma)$. We will show that either $w' \leq w$ or $w \leq w'$, and therefore that $\con^-(\Delta,\Gamma)$ is a chain.
    
    If both $w$ and $w'$ are elements of $\Delta$, then we are done. If $w \notin \Delta$, write $w = y^-_\Delta \vee (y\wedge z)$ with $y \in \Delta$ and $z \in \Gamma$. Now, there exists $y' \in \Delta$ so that either $w' = y'$ or $(y')^-_\Delta$ exists and $(y')^-_\Delta \leq w' \leq y'$. We see that if $y' \lneq y$ (respectively $y \lneq y'$) then $w' \leq w$ (respectively $w' \leq w$). Since $\Delta$ is a chain, the only other possibility is that $y = y'$. Since $y$ has a predecessor, this means there exists $z' \in \Gamma$ so that $w' = y^-_\Delta \vee (y\wedge z')$. Then since $\Gamma$ is a chain, we have without loss of generality that $z' \leq z$ and thus $w' \leq w$.
\end{proof}

\begin{remark}\label{rem:CompSeriesRefinements}
    If $\Delta$ is a prime chain, then Theorem~\ref{prop:compExists} and Proposition~\ref{prop:conflation} immediately imply that $\con^-(\Delta,\Gamma) = \Delta = \con^+(\Delta,\Gamma)$ for any $\Gamma$.
\end{remark}

\begin{samepage}

\begin{lemma}\label{lem:bicompleteConflation}
    Let $\Delta, \Gamma \subseteq \L$ be spanning chains.
    \begin{enumerate}
        \item If $\L$ is (AB5) and both $\Delta$ and $\Gamma$ are complete, then $\con^-(\Delta,\Gamma)$ is complete.
        \item If $\L$ is (AB5$^*$) and both $\Delta$ and $\Gamma$ are cocomplete, then $\con^+(\Delta,\Gamma)$ is complete.
        \item If $\L$ is weakly JHS and both $\Delta$ and $\Gamma$ are bicomplete, then $\con^-(\Delta,\Gamma) = \con^+(\Delta,\Gamma)$ and this chain is bicomplete.
    \end{enumerate}
\end{lemma}

\end{samepage}

\begin{proof}
    (1) Let $\Theta \subseteq \con^-(\Delta,\Gamma)$ and let $\Theta_\Delta = \left\{y \in \Delta \ \middle| \  \exists z \in \Gamma: y^-_\Delta \vee (y\wedge z) \in \Theta\right\}$. Denote $y_0  = \bigwedge \Theta_\Delta$. Note that since $\Delta$ is complete, we have $y_0 \in \Delta$.
    
    If $y_0 \notin \Theta_\Delta$, then for all $y \leq y_0$ in $\Delta$, there exists $y' \in \Delta$ and $z' \in \Gamma$ so that $w:=(y')^-_\Delta \vee (y' \wedge z') \in \Theta$ and
    $y \leq w \leq y_0.$
    We conclude that $\bigvee \Theta = y_0 \in \Delta \subseteq \con^-(\Delta,\Gamma)$.
    
    Otherwise, let $\Theta_\Gamma = \left\{z \in \Gamma \ \middle| \  (y_0)^-_\Gamma \vee (y_0\wedge z) \in \Theta\right\}$. By the (AB5) property, we then have
    \begin{eqnarray*}
        \bigvee \Theta &=& \bigvee \left\{(y_0)^-_\Delta \vee (y_0\wedge z) \ \middle| \ z \in \Theta_\Gamma\right\}\\
        &=& (y_0)^-_\Delta \vee \left(y_0\wedge \bigvee \Theta_\Gamma\right).
    \end{eqnarray*}
    Now since $\Gamma$ is complete, we have $\bigvee \Theta_\Gamma \in \Gamma$ and therefore $\bigvee \Theta \in \con^-(\Delta,\Gamma)$.
    
    (2)Let $\Theta \subseteq \con^+(\Delta,\Gamma)$ and let $\Theta_\Delta = \left\{y \in \Delta \ \middle| \ \exists z \in \Gamma: y \vee (y^+\wedge z) \in \Theta\right\}$. Denote $y_0=\bigvee\Theta_\Delta$. Note that since $\Delta$ is cocomplete, we have $y_0 \in \Delta$.
    
    If $y_0 \notin \Theta_\Delta$, then for all $y_0\leq y$ in $\Delta$, there exists $y' \in \Delta$ and $z' \in \Gamma$ so that $w:=y' \vee ((y')^+_\Delta \wedge z') \in \Theta$ and $y_0\leq w \leq y$. We conclude that $\bigwedge\Theta = y_0 \in \Delta\subseteq\con^+(\Delta,\Gamma)$.
    
    Otherwise, let $\Theta_\Gamma = \left\{z \in \Gamma \ \middle| \ y_0 \vee ((y_0)^+_\Delta \wedge z) \in \Theta\right\}$. By the (AB5$^*$) property, we then have
    \begin{eqnarray*}
        \bigwedge\Theta &=& \bigwedge\left\{y_0 \vee ((y_0)^+_\Delta\wedge z) \ \middle| \ z \in \Theta_\Gamma\right\}\\
        &=& y_0 \vee \left((y_0)^+_\Delta \wedge \bigwedge \Theta_\Gamma\right).
    \end{eqnarray*}
    Now since $\Gamma$ is cocomplete, we have that $\bigwedge\Theta_\Gamma \in \Gamma$ and therefore $\bigwedge\Theta \in \con^+(\Delta,\Gamma)$.
    
    (3) Let $y \in \Delta$ and $z \in \Gamma$. If $y^-_\Delta = y$, then $$y^-_\Delta \vee (y\wedge z) = y = y \vee \left(y^+_\Delta \wedge 0\right).$$
    Otherwise, we have
    $$y^-_\Delta \vee (y\wedge z) = y^-_\Delta \vee \left(\left(y^-_\Delta\right)^+ \wedge z\right)$$
    by Proposition~\ref{prop:upperLowerFactors}. Therefore in either case, we have that $y^-_\Delta \vee (y\wedge z) \in \con^+(\Delta,\Gamma)$. By symmetry, we conclude that $\con^-(\Delta,\Gamma) = \con^+(\Delta,\Gamma)$. This subobject chain is then bicomplete as a consequence of (1) and (2).
\end{proof}

\begin{example}\label{ex:conflationNotComplete}
    We now present an example of a conflation between a bicomplete chain and a complete chain which does not result in a bicomplete chain.
    Consider the category $\Fun(\R,\veck)$ and define the subobject chains $\Delta$ and $\Gamma$:
    \begin{eqnarray*}
        \Delta &=& \left\{0, M_{[1,1]}, M_{[\frac{1}{2},1]}, M_{[0,1]}\right\}\\
        \Gamma &=& \left\{0, M_{[1,1]}, M_{[0,1]} \right\} \cup \left\{ M_{(a,1]} \mid a\in\R,\, 0< a < 1 \right\}.
    \end{eqnarray*}
    Note that $\Delta$ is bicomplete and $\Gamma$ is complete.
    Thus it makes sense to take $\con^-(\Delta,\Gamma)$, which is complete by Lemma \ref{lem:bicompleteConflation}.
    We see $\con^-(\Delta,\Gamma)$ is given by
    $$\con^-(\Delta, \Gamma) = \left\{0, M_{[1,1]}, M_{[\frac{1}{2},1]}, M_{[0,1]}\right\} \cup \left\{ M_{(a,1]} \mid a\in\R,\, 0< a < 1 \right\} = \Delta\cup\Gamma.$$
    We see $\Delta\cup\Gamma$ is still missing uncountably many infimums and so it is not cocomplete.
    Thus, $\con^-(\Delta,\Gamma)$ is not bicomplete.
\end{example}

\begin{samepage}

\begin{lemma}\label{lem:maxRep}
     Let $\Delta \subseteq \L$ be a complete spanning chain and let $\Gamma\subseteq \L$ be a cocomplete spanning chain.
     
     \begin{enumerate}
         \item Suppose $\L$ is (AB5$^*$), let $w \in \con^-(\Delta,\Gamma)$, and denote $$\Omega(w) = \left\{(y,z) \in \Delta\times\Gamma \ \middle| \  y^-_\Delta \vee \left(y \wedge z\right) = w\right\}.$$ Then $\Omega(w)$ contains a minimal element with respect to lexicographical ordering of $\Delta\times\Gamma$.
         \item Suppose $\L$ is (AB5), let $w \in \con^+(\Gamma,\Delta)$, and denote
         $$\Omega(w) = \left\{(z,y) \in \Gamma\times\Delta \ \middle| \  z \vee \left(z^+_\Gamma\wedge Y\right) = w\right\}.$$ Then $\Omega(w)$ contains a maximal element with respect to the lexicographical ordering of $\Gamma\times\Delta$.
     \end{enumerate}
\end{lemma}

\end{samepage}

\begin{proof}
    We prove only (1) since the proof of (2) is dual. Let $\Theta_\Delta = \{y \in \Delta \mid  \exists z \in \Gamma:(y,z) \in \Omega(w)\}$. We first claim that $|
    \Theta_\Delta| \leq 2$. Indeed, suppose $y' \lneq y$ and there exist $z, z' \in \Gamma$ such that $(y,z)$ and $(y',z')$ are in $\Omega(w)$. Note that $y^-_\Delta, (y')^-_\Delta \in \Delta$ since $\Delta$ is complete.
    Then
    $$(y')^-_\Delta \leq (y')^-_\Delta \wedge \vee(y'\wedge z') \leq y' \lneq y^-_\Delta \lneq y^-_\Delta \vee (y \wedge z)$$
    implies that $y' = y^-_\Delta$, and so $\Theta_\Delta = \{y,y'\}$.
    Now let $y_0 = \bigwedge \Theta_\Delta$ and note that $y_0 \in \Theta_\Delta$ since this set contains at most two elements. Define $\Theta_\Gamma = \{z \in \Gamma \mid  (y_0,z) \in \Theta\}$. Since $\Gamma$ is cocomplete, we note that $z_0:=\bigwedge \Theta_\Gamma \in \Gamma$. Then by the (AB5$^*$) property, we have
    \begin{eqnarray*}
        (y_0)^-_\Delta \vee (y_0 \wedge z_0) &=& (y_0)^-_\Delta \vee \left(Y_0 \wedge \bigwedge \Theta_\Gamma\right)\\
        &=& \bigwedge((y_0)^-_\Delta \vee (y_0 \wedge \Theta_\Gamma))\\
        &=& \bigwedge\{w\}\\
        &=& w.
    \end{eqnarray*}
    We conclude that $(y_0,z_0) \in \Omega(w)$. Moreover, if $y \in \Delta$ and $z \in \Gamma$ with either (a) $y \lneq y_0$, or (b) $y = y_0$ and $z \lneq z_0$, then $(y,z) \notin \Omega(w)$. This means that $(y_0,z_0) = \min\Omega(w)$ in the lexicographic ordering.
\end{proof}

\begin{example}
We note that as a $\mathbb{Z}$-module, $\mathbb{Z}$ is subobject bicomplete and (AB5). To see that $\mathbb{Z}$ is not (AB5$^*$), let $p$ and $q$ be primes and let $\Delta_p$ and $\Delta_q$ be as in Example~\ref{ex:subobjectChains}. Let $\Delta_p' = \Delta_p \setminus\{0\mathbb{Z}\}$. Then for any $\beta \in \mathbb{N}$ with $\beta \neq 0$, we have
\begin{eqnarray*}
    \lim(\Delta'_p + q^\beta\mathbb{Z}) &=& \lim_{\alpha \rightarrow\infty} (p^\alpha \mathbb{Z} + q^\beta\mathbb{Z}) =  \mathbb{Z}\\
    \lim(\Delta'_p) + q^\beta\mathbb{Z} &=& \left(\lim_{\alpha\rightarrow\infty}p^\alpha\mathbb{Z}\right) + q^\beta\mathbb{Z} = q^\beta\mathbb{Z}.
\end{eqnarray*}
Now let $\alpha \neq 0$. Working in $\con^-(\Delta_p,\Delta_q)$, we then have by Example~\ref{ex:conflation}(3) that
$$\Omega(p^\alpha\mathbb{Z}) = \{(p^\alpha\mathbb{Z},q^\beta\mathbb{Z}) \mid \beta \neq 0\} \cup \{(p^{\alpha-1}\mathbb{Z},0\mathbb{Z})\}.$$
We note that $\Omega(p^\alpha\mathbb{Z})$ has $(p^{\alpha-1}\mathbb{Z},0\mathbb{Z})$ as its maximal element, but has no minimal element.
\end{example}

\begin{lemma}\label{lem:maxSuccessor}
Suppose $\L$ is weakly JHS, let $\Delta \subseteq \L$ be a complete spanning chain, and let $\Gamma$ be a cocomplete spanning chain.
\begin{enumerate}
    \item Let $w \in \con^-(\Delta,\Gamma)$ and write $w = y^-_\Delta \vee (y \wedge z)$ with $(y,z) \in \Delta\times\Gamma$ minimal in the sense of Lemma~\ref{lem:maxRep}(1). Then
    $$w^-_{\con^-(\Delta,\Gamma)} =\begin{cases} w & y^-_\Delta = y\\y^-_\Delta \vee \left(y\wedge z^-_\Gamma\right) & y^-_\Delta \neq y.\end{cases}$$
    
    \item Let $w \in \con^+(\Gamma,\Delta)$ and write $w = z \vee (z^+_\Gamma \wedge y)$ with $(z,y) \in \Gamma\times\Delta$ maximal in the sense of Lemma~\ref{lem:maxRep}(2). Then
    $$w^+_{\con^+(\Gamma,\Delta)} = \begin{cases} w & z^+_\Gamma = z\\z \vee \left(z^+_\Gamma\wedge y^+_\Delta\right) & z^+_\Gamma \neq z.\end{cases}$$
\end{enumerate}
\end{lemma}

\begin{proof}
    We prove only (1) as the proof of (2) is dual. We first observe that since $\L$ is (AB5$^*$) and $\Gamma$ is cocomplete, it is indeed possible to write $w = y^-_\Delta \vee (y\wedge z)$ with $(y,z) \in \Delta\times\Gamma$ minimal in the sense of Lemma~\ref{lem:maxRep}(1).
    
    Now if $y^-_\Delta = y$, then we must have $z = 0$. Then for any $w' \lneq w$, we have $w' \lneq y$ and so $w^-_{\con^-(\Delta,\Gamma)} \leq y$. Moreover, if $y' \lneq y$, then because $\Delta$ is complete we have $y' = (y')^-_\Delta \vee (y' \wedge \hat{1}) \in \con^-(\Delta,\Gamma)$. In particular, since $y' \leq y_\Delta^-$, this means that $y\leq w^-_{\con^-(\Delta,\Gamma)}$. We conclude that $w^-_{\con^-(\Delta,\Gamma)} = y$.
    
    Otherwise, we have $y^-_\Delta \lneq w \leq y$. In particular, $z \neq 0$. By the (AB5) property, this means
    \begin{eqnarray*}
        w^-_{\con^-(\Delta,\Gamma)} &=& \bigvee\left\{y^-_\Delta \vee (y\wedge z') \ \middle| \ z'\in \Gamma, z' \lneq z\right\}\\
        &=& y^-_\Delta \vee \left(y\wedge \left(\bigvee\{z' \in \Gamma \mid z' \lneq z\}\right)\right)\\
        &=& y^-_\Delta \vee \left(y\wedge z^-_\Gamma\right).
    \end{eqnarray*}
    
    \vspace{-20pt} \qedhere
\end{proof}

\vspace{0.1cm}

We are now ready to prove our ``Schreier-like'' theorem.

\begin{theorem}\label{thm:schreier}\
    \begin{enumerate}
        \item Suppose $\L$ is weakly JHS, and let $\Delta, \Gamma \subseteq \L$ be chains. Then there exist bicomplete (spanning) chains $\Delta' \supseteq \Delta$ and $\Gamma' \supseteq \Gamma$ which are up-down perspective.
        \item Let $X$ be a weakly JHS object of $\Cat$ and let $\Delta$ and $\Gamma$ be spanning subobject chains of $X$. Then there exist bicomplete spanning subobject chains $\Delta' \supseteq \Delta$ and $\Gamma' \supseteq \Gamma$ which are subfactor equivalent.
    \end{enumerate}
\end{theorem}

\begin{proof}
It suffices to prove (1), as (2) will then be a consequence of Proposition~\ref{prop:subfactor_proj}.

Let $\Delta' = \overline{\Delta \cup \left\{\hat{0},\hat{1}\right\}}$ and $\Gamma' = \overline{\Gamma \cup \left\{\hat{0}, \hat{1}\right\}}$ be the bicompletions of $\Delta \cup \left\{\hat{0},\hat{1}\right\}$ and $\Gamma \cup \left\{\hat{0},\hat{1}\right\}$, respectively. Note that, by Proposition~\ref{prop:bicompletion}, $\Delta'$ and $\Gamma'$ are spanning chains which refine $\Delta$ and $\Gamma$, respectively.
We will show that $\con^-(\Delta',\Gamma')$ and $\con^-(\Gamma',\Delta')$ are up-down perspective.

Let $w \in \con^-(\Delta',\Gamma')$. Since $\L$ is (AB5$^*$), we can write $w = y_{\Delta'}^- \vee (y \wedge z)$ with $(y,z) \in \Delta'\times\Gamma'$ minimal in the sense of Lemma~\ref{lem:maxRep}(1). Now if $y_{\Delta'}^- = y$, then $w^-_{\con^-(\Delta',\Gamma')} = w$ by Lemma~\ref{lem:maxSuccessor}. In particular, this means $\left[w^-_{\con^-(\Delta',\Gamma')}, w\right] \notin \seg(\con^-(\Delta',\Gamma'))$. Otherwise, again by Lemma~\ref{lem:maxSuccessor}, we have that 
    $$w^- := w^-_{\con^-(\Delta',\Gamma')} = y^-_{\Delta'} \vee \left(y \wedge z_{\Gamma'}^-\right).$$
    We denote $v:= z_{\Gamma'}^- \vee (z \wedge y)$ and $v' := z_{\Gamma'}^- \vee \left(z \wedge y_{\Delta'}^-\right)$. To simplify notation,
denote $\mathcal{Y} = [y_{\Delta'}^-, y]$, $\mathcal{W} = [w^-, w]$, $\Omega = [v', v]$, and $\mathcal{Z} = [z_{\Gamma'}^-,z]$. By construction, we have $\mathcal{W} = \pi_\mathcal{Y}(\mathcal{Z})$ and $\Omega = \pi_\mathcal{Z}(\mathcal{Y})$. Thus we can apply Lemma~\ref{lem:butterflyLattice} and conclude that $\mathcal{W} \pers \Omega$. In particular, this means $\mathcal{W} \in \seg(\con^-_{\mathcal{S},\mathcal{T}})$ if and only if $v' \neq v$. It follows that if $[w^-,w]$ is a prime chain, then $(z,y)$ is the minimal representative of $v$ in $\con^-(\Gamma',\Delta')$ and that $v' = v^-_{\con^-(\Gamma',\Delta')}$. In particular, we have $[v',v] \in \seg(\con^-(\Gamma',\Delta'))$.

By symmetry, there is a bijection
    $$\Phi: \left\{w \in \con^-(\Delta',\Gamma') \mid w \neq w^-_{\con^-(\Delta',\Gamma')}\right\} \rightarrow \left\{v \in \con^-(\Gamma',\Delta') \mid v \neq v^-_{\con^-(\Gamma',\Delta')}\right\}$$
so that $\left[w^-_{\con^-(\Delta',\Gamma')}, w\right] \pers \left[\Phi(w)^-_{\con^-(\Gamma',\Delta')}, \Phi(w)\right]$ given by sending the minimal representative $w = y_{\Delta'}^- \vee (y \wedge z)$ to $z^-_{\Gamma'} \vee (y \wedge z)$. We conclude that $\con^-(\Delta',\Gamma')$ and $\con^-(\Gamma',\Delta')$ are up-down perspective as claimed.
\end{proof}

From this, we deduce our ``Jordan--H\"older-like theorem''.

\begin{theorem}[Theorem \ref{thm:intro:Jordan--Holder}]\label{thm:Jordan--Holder}\
    \begin{enumerate}
        \item Let $\L$ be a weakly Jordan--H\"older--Schreier (complete bounded modular) lattice. Then
        \begin{enumerate}
            \item There exists a prime chain in $\L$.
            \item Let $\Delta, \Gamma \subseteq \L$ be prime chains. Then there is a unique bijection $\Phi: \left\{y \in \Delta \ \middle| \ y \neq y^-_\Delta\right\} \rightarrow \left\{z \in \Gamma \ \middle| \ z \neq z^-_\Gamma\right\}$ which induces an up-down perspectivity between $\Delta$ and $\Gamma$.
        \end{enumerate}
        \item Let $\mathcal{A}$ be a skeletally small abelian category and let $X$ be a weakly Jordan--H\"older--Schreier object in $\mathcal{A}$. Then there exists a composition series for $X$ and any two composition series for $X$ are subfactor equivalent.
    \end{enumerate}
\end{theorem}

\begin{proof}
    It suffices to prove (1), as (2) will then follow from Propositions~\ref{prop:comp_cover} and~\ref{prop:subfactor_proj}.

    The statement of (1a) is already included in Theorem~\ref{prop:compExists}. Thus let $\Delta, \Gamma$ be two prime chains. Then $\Delta$ and $\Gamma$ are maximal by Theorem~\ref{prop:compExists}. In particular, both are bicomplete and we have $\con^-(\Delta,\Gamma) = \Delta$ and $\con^-(\Gamma,\Delta) = \Gamma$ by Lemma~\ref{prop:conflation}. It then follows from Theorem~\ref{thm:schreier} that $\Delta$ and $\Gamma$ are up-down perspective. Finally, the fact that the bijection inducing the bijective equivalence is unique follows from Lemma~\ref{lem:unique_proj}.
\end{proof}

\begin{remark}
    Recall from Example~\ref{ex:refineableCompSeries} that an upper or lower composition series may admit a proper refinement. Thus, our proof of Theorem~\ref{thm:Jordan--Holder} cannot be readily modified to prove a ``Jordan--H\"older-like theorem'' for upper or lower composition series. Moreover, we have indeed made use of both the axioms the (AB5) property and the (AB5$^*$) property in our proof of Theorem~\ref{thm:Jordan--Holder}. For example, both were necessary in deducing Lemma~\ref{lem:maxSuccessor} from Lemmas~\ref{lem:bicompleteConflation} and~\ref{lem:maxRep}.
\end{remark}


\subsection{The Stitching Lemma}\label{sec:stitching}

In this section, we examine how one may ``stitch together'' filtrations from multiple intervals (or subquotients in an abelian category) to obtain a filtration of a given lattice or object. In particular, this leads to Theorem~\ref{thm:intro:inducedInclusion}, which shows that categories in which every object is (weakly) JHS
offer a natural generalization of length categories.

In light of Theorem~\ref{thm:Jordan--Holder}, it makes sense to talk about the composition factors of an object $X$ in $\mathcal{A}$, rather than just of a specific composition series. To that end, we establish the following notational conventions.

\begin{notation}\label{nota:subfact_of_obj}\
    \begin{enumerate}
        \item Suppose $\L$ is weakly JHS, and let $\Delta \subseteq \L$ be a prime chain. We denote $\cov_\pi(\L) := \cov_\pi(\Delta)$. Note that $\cov_\pi(\L)$ does not depend on $\Delta$ by Theorem~\ref{thm:Jordan--Holder}.
        \item Let $X$ be a weakly JHS object in $\Cat$. Let $\Delta$ be any composition series of $X$. We denote $\fac(X) := \fac(\Delta)$. Note that (up to isomorphism) $\fac(X)$ does not depend on $\Delta$ by Theorem~\ref{thm:Jordan--Holder}.
    \end{enumerate}
\end{notation}

To prepare for the main results of this section, we also need the following straightforward lemma.

\begin{lemma}[``The Stitching Lemma'']\label{lem:stitching}\
    \begin{enumerate}
        \item Suppose $\L$ is complete. Let $\mathfrak{S}, \mathfrak{T} \subseteq 2^{(2^\L)}$ be sets of segments which are closed with respect to up-down perspectivity. If $\L$ has an $\mathfrak{S}$-filtration $\Delta_\mathfrak{S}$ and each $\mathcal{S} \in \seg(\Delta_\mathfrak{S})$ has a $\mathfrak{T}$-filtration $\Gamma_\S$, then $\Delta_\mathfrak{T} := \Delta_\mathfrak{S} \cup \left(\bigcup_{\S \in \seg(\Delta_\mathfrak{S})} \Gamma_\S\right)$ is a $\mathfrak{T}$-filtration (of $\L$) which refines $\Delta_\mathfrak{S}$. Moreover, this filtration satisfies $\seg(\Delta_\mathfrak{T}) = \bigsqcup_{\S \in \seg(\Delta_\mathfrak{S})}\seg(\Gamma_Y)$.
        \item Let $X$ be a subobject complete object in $\Cat$. Let $\mathcal D$ and $\mathcal E$ be subcategories of $\mathcal A$ which are closed under isomorphisms. Suppose that $X$ has a $\mathcal{D}$-filtration $\Delta_\mathcal{D}$, and that each $W_Y:= Y/Y^-_{\Delta_\mathcal{D}}$, has a $\mathcal{E}$-filtration $\Gamma_Y$. Given $j_{Y,V}: V \hookrightarrow W_Y$ in $\Gamma_Y$, denote by $h_{Y,V}: Z_{Y,V} \rightarrow Y$ the kernel of the map $Y \twoheadrightarrow W_Y \twoheadrightarrow W_Y/V$, and let $\iota_{Y,V} = \iota_Y \circ h_{Y,V}$. Then
        $$\Delta_\mathcal{E}:= \Delta_{\mathcal{D}} \cup \left(\bigcup_{W_Y \in \fac(\Delta_\mathcal{D})} \{\iota_{Y,V}:Z_{Y,V} \hookrightarrow X \mid V \in \Gamma_Y\}\right)$$
        is an $\mathcal{E}$-filtration of $X$ which refines $\Delta_\mathcal{D}$. Moreover, this filtration satisfies $$\fac(\Delta_\mathcal{E}) = \bigsqcup_{Y \in \fac(\Delta_\mathcal{D})}\seg(\Gamma_Y).$$
    \end{enumerate}
\end{lemma}

\begin{proof}
    We prove only (1), as in light of Proposition~\ref{prop:subQuotients} the proof of (2) is analogous.
    
    It is clear from the construction that $\Delta_\mathfrak{T}$ is a chain which refines $\Delta_\mathfrak{S}$. We next show that $\Delta_\mathfrak{T}$ is bicomplete. Let $\Gamma \subseteq \Delta_\mathfrak{T}$ be nonempty and note that $\bigvee \Gamma$ exists because $\L$ is complete. Now if there exists $\S \in \seg(\Delta_\mathfrak{S})$ such that $\bigvee \Gamma \in \Gamma_\S$, we are done. Otherwise, denote
    $$\Gamma' := \left\{y \in \Delta_\mathfrak{S} \ \middle| \ \exists y' \in \Gamma \text{ such that } y^-_{\Delta_\mathfrak{S}} \leq y' \leq y\right\}.$$
    The fact that each $\Delta_\mathfrak{S}$ and all of the $\Gamma_\S$ are complete then implies that $\bigvee \Gamma = \bigvee \Gamma' \in \Delta_\mathfrak{S} \subseteq \mathfrak{T}$. The proof that $\Delta_\mathfrak{T}$ is cocomplete is similar.
    
    It remains so show that $\seg(\Delta_\mathfrak{T}) = \bigsqcup_{\S \in \seg(\Delta_\mathfrak{S})}\seg(\Gamma_Y)$. To see this, let $y \in \Delta_\mathfrak{T}$. If $y \in \Delta_\mathfrak{S}$ and $y_{\Delta_\mathfrak{S}}^- = y$, then clearly $y_{\Delta_\mathfrak{T}}^- = y$ as well. Otherwise, there exists a unique segment $\S \in \seg(\Delta_\mathfrak{S})$ such that $y \in \S \setminus \{\bigwedge \S\}$. Once again, we will then have $y_{\Delta_\mathfrak{T}}^- = y_{\Delta_\S}^-$. This implies the result.
\end{proof}

\begin{remark}
    Note that in Lemma~\ref{lem:stitching}(1), it is necessary to explicitly include $\Delta_\mathfrak{S}$ as a subset of $\Delta_\mathfrak{T}$, or we will lose those $y \in \Delta_\mathfrak{S}$ which satisfy $y^-_{\Delta_\mathfrak{S}} = y$. The situation is analogous in Lemma~\ref{lem:stitching}(2).
\end{remark}

As a consequence, we obtain the final result of this section. Note in particular that Theorem~\ref{cor:inducedInclusion} generalizes well-known properties of length categories. 

\begin{theorem}[Theorem~\ref{thm:intro:inducedInclusion}]\label{cor:inducedInclusion}\
    \begin{enumerate}
        \item Let $\L$ be weakly Jordan--H\"older--Schreier, and let $y \in \L$. Let $\Delta_0 \subseteq [\hat{0},y]$ and $\Delta_1 \subseteq [y,\hat{1}]$ be prime chains in the sublattices $[\hat{0},y]$ and $[y,\hat{1}]$, respectively. Then
    \begin{enumerate}
        \item $\Delta_0 \cup \Delta_1 \subseteq \L$ is a prime chain. Moreover, this prime chain induces a decomposition
        $$\cov_\pi(\L) = \cov_\pi[\hat{0},y] \sqcup \cov_\pi[y,\hat{1}].$$
        \item Suppose $[y,\hat{1}]$ is Jordan--H\"older--Schreier. Then $y = \hat{1}$ if and only if the inclusion $\cov_\pi[\hat{0},y] \subseteq \cov_\pi(\L)$ is an equality.
        \item Suppose $[\hat{0},y]$ is Jordan--H\"older--Schreier. Then $y = \hat{0}$ if and only if the inclusion $\cov_\pi[y,\hat{1}] \subseteq \cov_\pi(\L)$ is an equality.
    \end{enumerate}
    \item Let $X$ be a Jordan--H\"older--Schreier object in a well-powered abelian category $\Cat$. Let $Y \subob X$ be objects in $\Cat$ and consider the exact sequence
        $$0 \rightarrow Y \xrightarrow{\iota} X \xrightarrow{q} X/Y \rightarrow 0.$$
    Let $\Delta_Y$ and $\Delta_{X/Y}$ be composition series of $Y$ and $X/Y$, respectively. Then:
    \begin{enumerate}
        \item There is a composition series of $X$ given by $\Delta_Y \bigcup q^{-1}(\Delta_{X/Y})$. Moreover, this composition series induces a bijection
        $$\fac(Y)\sqcup \fac(X/Y) \cong \fac(X).$$
        \item If $X/Y$ is Jordan--H\"older--Schreier, then $Y = X$ if and only if the induced inclusion $\fac(Y) \hookrightarrow \fac(X)$ is a bijection.
        \item If $Y$ is Jordan--H\"older--Schreier, then $Y = 0$ if and only if the induced inclusion $\fac(X/Y) \hookrightarrow \fac(X)$ is a bijection.
    \end{enumerate}
    \end{enumerate}
\end{theorem}

\begin{proof}
    We prove only (1) as the proof of (2) is similar. Suppose $\L$ is weakly JHS and let $y$, $\Delta_0$, and $\Delta_1$ be as in the theorem. Recall that $[\hat{0},y]$ and $[y,\hat{1}]$ are themselves weakly JHS by Proposition~\ref{prop:subquotientJHS}, and thus it makes sense to talk about $\cov_\pi[\hat{0},y]$ and $\cov_\pi[y,\hat{1}]$. In particular, (1a) is the special case of Lemma~\ref{lem:stitching}(1) where $\mathfrak{S}$ is the set of segments which are projectively equivalent to either $[\hat{0},y]$ or $[y,\hat{1}]$ and $\mathfrak{T}$ is the set of prime segments. For (1b), we have that the inclusion $\cov_\pi[\hat{0},y] \subseteq \cov_\pi\L$ is an equality if and only if $\cov[y,\hat{1}] = \emptyset$. By the assumption that $[y,\hat{1}]$ is JHS, this is equivalent to $y = \hat{1}$. Finally, (1c) is the dual of (1b).
\end{proof}

\section{Subobject Completeness and Categorical Limits and Colimits}\label{sec:limits_colimits}

In this section, we discuss the relationship between the completeness of the lattice $\Sub(X)$ and the existence of limits and colimits of directed systems in $\Cat$. In particular, this highlights some subtle changes which can occur when one switches to working in the category $\Cat^{op}$. Finally, we discuss some consequences of assuming that all objects in $\mathcal{A}$ satisfy the (weakly) JHS axioms.

\subsection{ Bounded Direct Limits and Subobject Completeness}\label{sec:cat_complete}

The purpose of this section is to compare supremums and infimums in the lattice $\Sub(X)$ with colimits and limits in the category $\Cat$. The following definition is the starting point of our comparison. Recall that a functor $F: I \rightarrow \mathcal{A}$ for $I$ a directed set is called a \textdef{directed diagram} (of shape $I$ in $\mathcal{A}$).

\begin{definition}\label{def:bounded_diagrams}\
    Let $I$ be a directed set and $F: I \rightarrow \mathcal{A}$ a directed diagram.
        \begin{enumerate}
            \item Let $X$ be an object in $\mathcal{A}$. We say that $F$ is \textdef{$X$-bounded} if it factors through the forgetful functor $\Sub(X) \rightarrow \mathcal{A}$.
            \item We say that $F$ is \textdef{bounded} if there exists an object $X$ in $\mathcal{A}$ such that $F$ is $X$-bounded.
        \end{enumerate}
\end{definition}

All of the directed diagrams considered in this section will be bounded. Thus we will adopt the common convention of identifying a diagram with its image. For example, let $X$ be an object of $\mathcal{A}$, $I$ a directed set, and $F:I \rightarrow \Sub(X)$ a directed diagram in $\Sub(X)$. For $\alpha \in I$, we denote $X_\alpha = F(\alpha)$. We then can then identify the image of $F$ with a direct set $\Delta = \{X_\alpha\mid \alpha \in I\} \subseteq \Sub(X)$. By construction, we have $\varinjlim_{\alpha} F(\alpha) = \bigvee \Delta$ (in the sense that the direct limit exists if and only if the supremum exists, and if both exist then they coincide). Similarly, if we assume that $I$ is downward directed, then by construction we have $\varprojlim_\alpha F(\alpha) = \bigwedge \Delta$.

We now consider the $X$-bounded diagram $U\circ F$, where $U: \Sub(X) \rightarrow \mathcal{A}$ is the forgetful functor. Note in particular that the $X$-bounded property means that, for $\alpha \leq \beta \in I$ (identified with the unique morphism $\alpha \rightarrow \beta$), we have that $U\circ F(\alpha\leq \beta) = \iota_{\alpha,\beta}: X_\alpha \rightarrow X_\beta$ is the inclusion map.

We first suppose that $I$ is downward directed and consider the (inverse) limit of the diagram $U\circ F$, which we denote $\lim\Delta := \varprojlim_\alpha U\circ F(\alpha)$. Suppose there exists an object $Z$ of $\mathcal A$ with morphisms $g_{\alpha}:Z\to X_\alpha$ for each $X_=\alpha\in\Delta$ such that $g_{\beta} = \iota_{\alpha,\beta}\circ g_{\alpha}$ whenever $\alpha \leq \beta$. Then $\im g_{\alpha} = \im g_{\beta}$ whenever $\alpha \leq \beta$. Now let $G=\im g_{\alpha}$ for some $X_\alpha \in \Delta$, and let $\kappa:G\hookrightarrow Y$ be the induced inclusion. Then each $g_{\beta}$ factors through $G$, and $(G,\iota_{\alpha}\circ \kappa)$ is in $\Sub(X)$. From this it follows that if $\lim \Delta$ or $\bigwedge \Delta$ exists then they both exist and $\lim \Delta \cong U(\bigwedge \Delta)$. For emphasis, we record this observation for use in Section~\ref{sec:examples}.
    
\begin{proposition}\label{prop:limit}
    Let $X$ be an object in $\Cat$, and suppose that $X$ is subobject complete. Then for any downward directed $\Delta \subseteq \Sub(X)$, the limit $\lim\Delta$ exists and is equal to $\bigwedge \Delta$.
\end{proposition}
    
    On the other hand, returning to the case where $I$ is directed, it may be the case that $\displaystyle \colim\Delta := \lim_{\rightarrow_\alpha} U\circ F(\alpha) \neq \bigvee \Delta$, or even that $\bigvee \Delta$ exists while $\colim\Delta$ does not. Indeed, suppose $\colim\Delta$ exists.
    Since $U$ is a solution to $U\circ F$, there will be a unique morphism $\colim \Delta \rightarrow X$. If this is a monomorphism, we will say that $\colim\Delta$ is a subobject of $X$, and we will indeed have that $\colim\Delta \cong U(\bigvee \Delta)$ in this case. The following shows that this map may not, however, be a monomorphism.
    
    \begin{example}\label{ex:coGroth}
    Let $\Veck$ be the category of vector spaces over a field $\field$, and let $X = \bigoplus_{i = 1}^\infty \field$. For $n \in \N$, denote $\bigoplus_{i = 1}^n\field$ the quotient of $X$ obtained by restricting to the first $n$ coordinates. Then
       \begin{displaymath}
       \prod_{i=1}^\infty \field = \lim \left(\cdots \twoheadrightarrow \bigoplus_{i=1}^n \field \twoheadrightarrow \bigoplus_{i=1}^{n-1} \field \twoheadrightarrow \cdots \twoheadrightarrow \field\oplus \field\twoheadrightarrow \field\to 0\right)\not\cong \bigoplus_{i=1}^\infty \field = X.
    \end{displaymath}
    In particular, $\bigoplus_{i=1}^\infty \field$ does not surject onto $\prod_{i=1}^\infty \field$. If we view the chain of quotient objects above as a subobject chain $\Delta$ in $\Veck^{op}$, we then have $\bigvee \Delta = X$ and $\colim\Delta = \prod_{i = 1}^\infty \field$, which is not a subobject of $X$.
    \end{example}

This discussion leads us to the following definitions, which are standard in the literature.

\begin{definition}\label{def:exact_limits}
    Let $X$ be an object in $\mathcal{A}$.
    \begin{enumerate}
        \item We say that $\mathcal{A}$ has \textdef{($X$-)bounded direct limits} if every ($X$-)bounded diagram in $\mathcal{A}$ has a limit.
        \item We say that $\mathcal{A}$ has \textdef{exact ($X$-)bounded direct limits} if it has ($X$-)bounded direct limits and, given a compatible system of monomorphisms $(\kappa_\alpha: X_\alpha \rightarrow Y_\alpha)_{\alpha \in I}$ between two ($X$-)bounded diagrams, the unique compatible morphism $\varinjlim_\alpha\kappa_\alpha: \varinjlim_\alpha X_\alpha \rightarrow \varinjlim_\alpha Y_\alpha$ is a monomorphism.
    \end{enumerate}
\end{definition}

\begin{remark}\label{rem:catComp1}
        It follows from taking $Y_\alpha = X$ for all $\alpha$ that if $\mathcal{A}$ has exact $X$-bounded direct limits then $\colim \Delta$ exists and is equal to $\bigvee \Delta$ for any directed set $\Delta \subseteq \Sub(X)$. In particular, this says means that if $\mathcal{A}$ has exact $X$-bounded direct limits, then $X$ is subobject complete.
\end{remark}

Note that Example~\ref{ex:coGroth} shows that having exact bounded direct limits is indeed a stronger condition than being subobject complete.


\subsection{Bounded Direct Limits and the JHS Axioms}\label{sec:selfDual}

Recall that a cocomplete abelian category $\mathcal{A}$ is said to satisfy Grothendieck's axiom (AB5) if it has exact direct limits. Recall also from Remark~\ref{rem:moduleJHS3} that this axiom is equivalent to requiring that $Z \cap (\sum \Delta) = \sum(Z \cap \Delta)$ for all objects $Z$ and directed systems $\Delta$. The purpose of this section is to establish a ``bounded'' analog of this result (Theorem~\ref{thm:catComp_AB5} and Corollary~\ref{cor:catComp_AB5_2}). Note in particular that the results of this section do not require the category to be cocomplete.

We first recall that, as we have seen, the existence of limits and colimits of subobject chains do not imply one another. This lack of self-duality is absent in Sections~\ref{sec:subobjectChains} and~\ref{sec:JHSobjects}, since the notion of subobject completeness is self-dual. For example, the following result can be compared with Proposition~\ref{prop:bicompleteIsLattice}(2), which shows that subobject completeness is inherited by all subquotients. See also Proposition~\ref{prop:subQuotients}.

\begin{proposition}\label{prop:subCatComplete}\
    Let $X$ be an object of $\Cat$.
    \begin{enumerate}
        \item If $\mathcal{A}$ has exact $X$-bounded direct limits, then $\mathcal{A}$ has $X'$-bounded direct limits for every subobject $X' \in \Sub(X)$.
        \item If $\mathcal{A}^{op}$ has exact $X$-bounded direct limits, then $\mathcal{A}^{op}$ has $X'$-bounded direct limits for every quotient object $X' \in \Sub^{op}(X)$.
    \end{enumerate}
\end{proposition}

\begin{proof}
    We prove only (1) as (2) is dual. Suppose $\mathcal{A}$ has exact $X$-bounded direct limits. Let $Y \subob X$ with inclusion map $\iota: Y \hookrightarrow X$. Let $\emptyset \neq \Delta \subseteq \Sub(Y) \subseteq \Sub(X)$ be a directed set. By assumption, we have that $\colim\Delta$ exists and there are natural maps $c_X: \colim\Delta \rightarrow X$ and $c_Y: \colim\Delta \rightarrow Y$, with $c_X$ being a monomorphism. By the universal property of the colimit, it follows that $c_X = \iota \circ c_Y$, and so $c_Y$ is also a monomorphism. This proves the result.
\end{proof}

We also need the following, the proof of which we defer to the appendix.

\begin{lemma}[Lemma~\ref{lem:cat_complete_appendix}]\label{lem:catComplete}
    Let $X$ be an object in $\Cat$. Then the following are equivalent.
    \begin{enumerate}
        \item For every chain $\Delta \subseteq \Sub(X)$, the direct limit $\colim\Delta$ exists (in $\mathcal{A}$) and is a subobject of $X$.
        \item For every well-ordered chain $\Delta \subseteq \Sub(X)$, the direct limit $\colim \Delta$ exists (in $\Cat$) and is a subobject of $X$.
        \item $\mathcal{A}$ has exact $X$-bounded direct limits.
    \end{enumerate}
\end{lemma}

We are now prepared to prove the first main theorem of this section.

\begin{theorem}\label{thm:catComp_AB5}
Let $X$ be an object in $\Cat$. Then the following are equivalent.
\begin{enumerate}
    \item $\mathcal{A}$ has exact $X$-bounded direct limits.
    \item For every directed subset $\Delta \subseteq \Sub(X)$, the direct limit $\colim\Delta$ exists (in $\Cat$) and is (AB5).
    \item For every chain $\Delta \subseteq \Sub(X)$, the direct limit $\colim\Delta$ exists (in $\Cat$) and is (AB5).
    \item For every well-ordered chain $\Delta \subseteq \Sub(X)$, the direct limit $\colim\Delta$ exists (in $\Cat$) and is (AB5).
 \end{enumerate}
   In particular, if $\mathcal{A}$ has exact $X$-bounded direct limits, then $X$ is (AB5).
\end{theorem}

\begin{proof}
    The implications (2)$\implies$(3)$\implies$(4) are trivial. Moreover, given the equivalence between (1) and (2), one can show the ``in particular'' part by applying (2) to the directed subset $\{X\} \subseteq \Sub(X)$. Thus it remains to prove that (1)$\implies$(2) and that (4)$\implies$(1).
    
    (1)$\implies$(2): Suppose (1) holds, and let $\Delta \subseteq \Sub(X)$ be a directed subset. Then $\colim\Delta$ exists and is a subobject of $X$ by (1) and Lemma~\ref{lem:catComplete}. In particular, $\mathcal{A}$ has exact $(\colim\Delta)$-bounded direct limits by Lemma~\ref{prop:subCatComplete}.
    
    It remains to show that $\colim\Delta$ is (AB5). Let $\Delta' \subseteq \Sub(\colim\Delta)$ be a subobject chain, and let $Z \subob \colim\Delta$ be a subobject. Then for each $Y \in \Delta$, there is a commutative diagram
    $$\begin{tikzcd}
        Z \cap Y \arrow[r,hook] \arrow[d,hook] & Y\arrow[d,hookrightarrow]\arrow[ddr,hook',bend left]\\
        Z \arrow[r,hook]\arrow[drr,hook,bend right] & Z + Y\arrow[dr,hook]\\
        &&\colim\Delta
    \end{tikzcd}$$
    where the outer square is a pullback and the inner square is bicartesian. Now as $\Cat$ has exact ($\colim\Delta$)-bounded direct limits, taking colimits (with respect to $Y \in \Delta$) yields a commutative diagram
        $$\begin{tikzcd}
        \bigvee(Z \cap \Delta') \arrow[r,hook] \arrow[d,hook] & \bigvee \Delta'\arrow[d,hookrightarrow]\arrow[ddr,hook',bend left]\\
        Z \arrow[r,hook]\arrow[drr,hook,bend right] & \bigvee(Z + \Delta')\arrow[dr,hook]\\
        &&\colim\Delta
    \end{tikzcd}$$
    where, since pushouts are colimits, the inner square remains a pushout. The fact that all of the maps are monomorphisms then implies that the inner square is bicartesian. Since the morphism $\bigvee(Z + \Delta') \rightarrow \colim(\Delta)$ is a monomorphism, this means the outer square must be a pullback as well. However, the pullback of $Z \hookrightarrow \colim\Delta \hookleftarrow \bigvee \Delta'$ is by definition $Z \cap (\bigvee \Delta')$. We conclude that $\bigvee(Z\cap \Delta') = Z \cap (\bigvee \Delta')$ as desired.
    
    (4)$\implies$(1): Suppose (4) holds, and let $\Delta \subseteq \Sub(X)$ be a well-ordered chain. We have that $\colim\Delta$ exists by (4), so in light of Lemma~\ref{lem:catComplete} we need only show that the canonical morphism $c: \colim\Delta \rightarrow X$ is a monomorphism. For each $Y \in \Delta$, let $\iota_Y: Y \hookrightarrow X$ and $\varepsilon_Y: Y \hookrightarrow \colim\Delta$ be the inclusion maps. Note in particular that $c \circ \varepsilon_Y = \iota_Y$, which is why $\varepsilon_Y$ must be mono. Moreover, this implies that $\varepsilon_Y(Y) \cap \ker(c) = 0$ for all $Y \in \Delta$. Finally, we note that $\Gamma := \{\varepsilon_Y(Y) \mid Y \in \Delta\} \subseteq \Sub(\colim\Delta)$ is a subobject chain which satisfies $\bigvee \Gamma = \colim\Delta$ (where the supremum is computed in $\Sub(\colim\Delta)$). Since $\colim\Delta$ is (AB5) by the assumption of (2), this means
    $$0 = \bigvee\left(\ker(c) \cap \Gamma\right) = \ker(c) \cap \left(\bigvee\Gamma\right) = \ker(c).$$
    We conclude that $c$ is a monomorphism as desired.
\end{proof}

As an immediate corollary, we conclude the following.

\begin{corollary}\label{cor:catComp_AB5}
    Let $X$ be an object in $\Cat$.
    \begin{enumerate}
        \item If $\mathcal{A}^{op}$ has exact $X$-bounded direct limits, then $X$ is (AB5$^*$).
        \item If both $\mathcal{A}$ and $\mathcal{A}^{op}$ have $X$-bounded direct limits, then $X$ is weakly JHS.
    \end{enumerate}
\end{corollary}

Another consequence of Theorem ~\ref{thm:catComp_AB5} is that if \emph{every} object of $\mathcal{A}$ is (AB5), then $\mathcal{A}$ has exact bounded direct limits if and only if every subobject chain (of every object in $\mathcal{A}$) has a direct limit. The goal of the remainder of this section is to prove that the existence of these direct limits is actually a consequence of the global (AB5) assumption. In particular, this will imply that if every object of $\mathcal{A}$ is (AB5), then for any $X \in \mathcal{A}$ and $\Delta \subseteq \Sub(X)$, we will have that $\bigvee \Delta$ is the direct limit of $\Delta$.

A proof of the following technical lemma can be found in the appendix.

\begin{lemma}[Lemma~\ref{lem:inclusion_gluing_appendix}]\label{lem:inclusion_gluing}
    Let $X$ and $Y$ be objects in $\mathcal{A}$, and suppose that $X\oplus Y$ is (AB5). Let $\Delta = (X_\alpha)_{\alpha \in I} \subseteq \Sub(X)$ and $\Gamma = (Y_\alpha)_{\alpha \in I} \subseteq \Sub(Y)$ be directed diagrams of the same shape. For $\alpha \leq \beta \in I$, denote by $x_{\alpha,\beta}: X_\alpha \rightarrow X_\beta$, $y_{\alpha,\beta}: Y_\alpha \rightarrow Y_\beta$, $x_\alpha: X_\alpha \rightarrow X$, and $y_\beta: Y_\beta \rightarrow X$ the inclusion maps. Suppose that $\bigvee \Delta = X$ (in $\Sub(X)$) and $\bigvee \Gamma = Y$ (in $\Sub(Y)$), and that there is a system of isomorphisms $(f_\alpha: X_\alpha \rightarrow Y_\alpha)_{\alpha \in I}$ such that $f_\beta \circ x_{\alpha,\beta} = y_{\alpha,\beta} \circ f_\alpha$ for all $\alpha \leq \beta \in I$. Then there is an isomorphism $f: X\rightarrow Y$ which satisfies $(\star)$ $f \circ x_\alpha = y_\beta \circ f_\alpha$ for all $\alpha \in I$. Furthermore, $f$ is the unique morphism $X \rightarrow Y$ which satisfies $(\star)$.
\end{lemma}

We are now prepared to prove the final main results of this section.

\begin{theorem}\label{thm:AB5_sup_colim}
    Let $\mathcal{A}$ be a well-powered abelian category, and suppose that every object of $\mathcal{A}$ is (AB5). Given an object $X$ in $\mathcal{A}$ and a directed set $\Delta = (X_\alpha)_{\alpha \in I} \subseteq \Sub(X)$, the supremum $\bigvee \Delta \in \Sub(X)$ (together with the inclusion maps $X_\alpha \rightarrow \bigvee \Delta$) is the direct limit of $\Delta$ in $\mathcal{A}$.
\end{theorem}

\begin{proof}
    We assume without loss of generality that $\bigvee \Delta = X$. For $\alpha \leq \beta \in I$, we denote by $\iota_{\alpha,\beta} : X_i=\alpha \rightarrow X_\beta$ and $\iota_{\alpha,X}: X_\alpha \rightarrow X$ the inclusion maps.

    Let $Y$ be an object of $\mathcal{A}$ and let $(f_\alpha: X_\alpha \rightarrow Y)_{\alpha \in I}$ be a compatible system of morphisms (so that $f_\beta \circ \iota_{\alpha,\beta} = f_\alpha$ for all $\alpha \leq \beta \in I$). Denote $Z := \bigvee_{\alpha \in I} \im(f_\alpha) \in \Sub(Y)$ and $W:= \bigvee_{\alpha \in I} \ker(f_\alpha) \in \Sub(X)$. Now for each $\alpha \in I$, denote $\overline{X}_\alpha := (X_\alpha + W)/W \cong X_\alpha/\ker(f_\alpha)$. Then there are induced isomorphisms $\overline{f}_\alpha: \overline{X}_\alpha \rightarrow \im(f_\alpha)$ and inclusion maps $\overline{\iota}_{\alpha,\beta}: \overline{X}_\alpha \rightarrow \overline{X}_\beta$ and $\iota'_{\alpha,\beta}: \im(f_\alpha) \rightarrow \im(f_\beta)$ which satisfy $f_\beta \circ \overline{\iota}_{\alpha,\beta} = \iota'_{\alpha,\beta} \circ f_{\alpha}$ for all $\alpha \leq \beta \in I$. Denoting by $\iota_{\alpha,Z}: \im(f_\alpha) \rightarrow Z$ and $\overline{\iota}_{\alpha,X}: \overline{X}_\alpha \rightarrow X/W$ the relevant inclusion maps, Lemma~\ref{lem:inclusion_gluing} then implies that there is a unique morphism $\overline{f}: X/W \rightarrow Z$ which satisfies $\overline{f}\circ\overline{\iota}_{\alpha,X} = \iota_{\alpha,Z} \circ \overline{f}_\alpha$ for all $\alpha \in I$, and moreover that $\overline{f}$ is an isomorphism.

    Now let $q: X \twoheadrightarrow \overline{X}$ and $\iota_{Z}: Z \rightarrow Y$ be the quotient and inclusion maps, and define $f := \iota_{Z} \circ \overline{f} \circ q$. It is straightforward to see that $f \circ \iota_{\alpha,X} = \iota_{\alpha,Y} \circ f_\alpha$ (where $\iota_{\alpha,Y}: \im(f_\alpha) \rightarrow Y$ is the inclusion map) for all $\alpha \in I$. It remains to show that $f$ is the unique morphism with this property.

    Let $g: X \rightarrow Y$ be a morphism which satisfies $g \circ \iota_{\alpha,X} = \iota_{\alpha,Y} \circ f_\alpha$ for all $\alpha \in I$. Then $\im(g) = \bigvee_{\alpha \in I} \im(f_\alpha) = \im(f)$ (see e.g. \cite[Corollary~70]{abelian lecture notes}). Moreover, we have $\ker(f) = \bigvee_{\alpha \in I} \ker(f_\alpha) \subob \ker(g)$. Thus there is an induced morphism $\overline{g}: \overline{X} \rightarrow Z$ such that $g = \iota_Z \circ \overline{g} \circ q$. We conclude that $\overline{g} = \overline{f}$, and hence $g = f$, by Lemma~\ref{lem:inclusion_gluing}.
\end{proof}

\begin{corollary}\label{cor:catComp_AB5_2}
    Let $\mathcal{A}$ a well-powered abelian category. Then:
    \begin{enumerate}
        \item $\mathcal{A}$ has exact bounded direct limits if and only if every object of $\Cat$ is (AB5).
        \item $\mathcal{A}^{op}$ has exact bounded direct limits if and only if every object of $\Cat$ is (AB5$^*$).
        \item Both $\Cat$ and $\Cat^{op}$ have exact bounded direct limits if and only if every object of $\Cat$ is weakly JHS.
    \end{enumerate}
\end{corollary}

\begin{proof}
    We prove only (1), as (2) is dual and (3) follows from (1) and (2). The forward implication is contained in Theorem~\ref{thm:catComp_AB5}, so suppose that every object of $\Cat$ is (AB5). Let $X$ be an object of $\mathcal{A}$ and $\Delta = (X_\alpha)_{\alpha \in I} \subseteq \Sub(X)$ an $X$-bounded diagram. By assumption, we have that $\Sub(X)$ is complete, and so $\bigvee \Delta$ exists. Moreover, $\bigvee \Delta = \colim \Delta$ is the direct limit of $\Delta$ (in $\mathcal{A}$) by Theorem~\ref{thm:AB5_sup_colim}. It thus follows from Theorem~\ref{thm:catComp_AB5} that $\mathcal{A}$ has exact $X$-bounded direct limits. Since $X$ was arbitrary, we conclude that $\mathcal{A}$ has exact bounded direct limits.
\end{proof}


\section{Examples and Discussion}\label{sec:examples}

In this section, we provide several examples and non-examples of (weakly) Jordan--H\"older--Schreier objects.

\subsection{Preliminary examples}\label{sec:preliminary}

First note that, as observed in Remark~\ref{rem:moduleJHS3}(3), in the category of all (left) modules over a ring, the weakly JHS objects are precisely the (AB5$^*$)-modules. In more general abelian categories, we also have all of the following.

\begin{example}\label{ex:preliminary}
    Let $X$ be an object in $\Cat$.
    \begin{enumerate}
        \item If $X$ has finite length, then $X$ is JHS.
        \item If $X$ is uniserial (that is, if $\Sub(X)$ is totally ordered), then $X$ is weakly JHS.
        \item If $X$ is completely distributive; that is, if $\Sub(X)$ is a complete distributive lattice, then $X$ is weakly JHS.
        \item If $\Cat$ is an (AB5) category and $X$ is artinian, then $X$ is weakly JHS.
    \end{enumerate}
\end{example}

    As a special case of Example~\ref{ex:preliminary}(4) above, note that if $R$ is commutative and noetherian, then the injective envelope $E(S)$ of a simple $R$-module $S$ will be artinian, and thus also JHS. This includes the Pr\"ufer $p$-group $\mathbb{Z}(p^\infty)$. Similarly, there is a class of Pr\"ufer modules over Leavitt path algebras that are uniserial and artinian. These are used to produce injective envelopes of simple modules (see, for example, work by Abrams, Mantese, and Tonolo \cite{amt19,amt22}). On the other hand, in Example~\ref{ex:injective} we give an example of an injective envelope of a simple object (not in a module category) which is JHS but not artinian. It is an interesting question to determine the generality in which the injective envelopes of simple objects are JHS.

    It is generally well-known that if a topological module over a topological ring is  linearly compact, then it is (AB5*).
    Taking the topology to be discrete, this is a generalization of the artinian property.
    If we consider the injective envelope of a simple module, we may ask if it is JHS; i.e., if it is (AB5*).
    One possibility is that the injective envelope is linearly compact.
    Since a simple module $S$ is essential in its injective envelope $E$, this is equivalent to asking if $E/S$ is linearly compact.

\subsection{Functor categories}\label{sec:pwf}
In this section, we consider functor categories in which every object in the target is (weakly) Jordan--H\"older--Schreier. In particular, this includes the category of pointwise finite-dimensional (pwf) persistence modules over any field and small category. For certain small categories, we are further able to describe the composition factors explicitly.

\begin{proposition}\label{prop:FunctorCategoryIsJHS}
    Let $\mathcal{C}$ be a small category and let $\Cat$ be a well-powered abelian category in which every object is weakly JHS. Then every object in $\Fun(\mathcal{C},\Cat)$ is weakly JHS.
\end{proposition}

\begin{proof}
    Let $M \in \Fun(\mathcal{C},\Cat)$ and let $\emptyset \neq \Delta \subseteq \Sub(M)$ be a subobject chain. 
    For $x \in \mathcal{C}$, we have that $\colim(\Delta(x)) = \bigvee \Delta(x) \in \Sub(M(x))$ by Theorem~\ref{thm:AB5_sup_colim}.It follows that $\colim(\Delta)$ exists and that the natural map $\colim(\Delta)\rightarrow M$ is a monomorphism. Therefore, $\Fun(\mathcal{C},\Cat)$ has exact $M$-bounded direct limits, and so $M$ is (AB5) by Theorem~\ref{thm:catComp_AB5}.
    
    Now let $N\subob M$ be an arbitrary subfunctor of $M$. Once again, for $x \in \mathcal{C}$, we have $\lim(\Delta(x) + N(x)) = \lim(\Delta(x))+ N(x))$ since $M(x) \in \Cat$ is (AB5$^*$). (We are using here the fact that inverse limits and infimums always coincide. See Proposition~\ref{prop:limit}.) It again follows immediately that $\lim(\Delta + N) = \lim(\Delta)+ N$ in $\Fun(\mathcal{C},\Cat)$, and so $\Fun(\mathcal{C},\Cat)$ is~(AB5$^*$).
\end{proof}

We now wish to apply Proposition~\ref{prop:FunctorCategoryIsJHS} to pointwise finite-dimensional persistence modules. To do so, we recall that a small category $\mathcal{C}$ is called directed if (a) for all objects $x\neq y \in \mathcal{C}$, at least one of $\Hom_\mathcal{C}(x,y)$ and $\Hom_\Cat(y,x)$ is empty and (b) for all objects $x \in \mathcal{C}$, $\Hom_\mathcal{C}(x,x) = \{1_x\}$. For example, any poset category is directed.

We also fix the following notation. Given an object $X$ in an abelian category $\Cat$ and an object $x$ of a directed category $\mathcal{C}$, we denote $\widetilde{X}_x:\mathcal{C}\rightarrow\Cat$ the functor which has $\widetilde{X}_x(x) = X$ and $\widetilde{X}_x(y) = 0$ for all other objects $y$ in $\mathcal{C}$. We note that $\widetilde{X}_x$ is simple in $\Fun(\mathcal{C},\Cat)$ if and only if $X$ is simple in $\Cat$.
Moreover, given a functor $M:\mathcal{C}\rightarrow \Cat$, we denote by $\supp(M)$ the set of objects $x$ of $\mathcal{C}$ for which $M(x) \neq 0$.
Finally, recall that the disjoint union of multisets is again a multiset where the multiplicity of elements is additive in the sense of cardinalities.
    
\begin{theorem}\label{thm:pwf}
    Let $\mathcal{C}$ be a small directed category, let $\Cat$ be a well-powered abelian category in which every object of $\Cat$ is JHS, and let $M$ be a functor in $\Fun(\mathcal{C},\Cat)$. Then $M$ is JHS and satisfies
    $$\fac(M) = \coprod_{x\in\supp(M)} \left\{ \wS_x \mid S\in\fac(M(x)) \right\}.$$
\end{theorem}

\begin{proof}
    First note that $M$ is weakly JHS by Proposition~\ref{prop:FunctorCategoryIsJHS}. Since $\supp(M) = \emptyset$ if and only if $M = 0$, we thus need only show the subfactors of $M$ are as described.
    
    We first choose some total order $\leq$ on the objects of $\supp(M)$ so that $\Hom_\mathcal{C}(x,y) = \emptyset$ whenever $y \lneq x$ and $x,y\in\supp(M)$.
    For each $x\in \supp(M)$, we then choose a composition series $\Omega_x$ of $M(x)$.
    Now for each $X\in\Omega_x$ we define a functor $M_{x,X}$ in $\Fun(\mathcal{C},\Cat)$ as follows.
    For $y$ an object in $\mathcal{C}$,
    we set $$M_{x,X}(y) = \begin{cases} M(y) & x \lneq y\\ X & y=x\\ 0 & y\lneq x, \end{cases}$$
    
    For
    $f: y \rightarrow y'$ a morphism in $\mathcal{C}$, we then set $M_{x,X}(f) = M(f)|_{M_{x,X}(y)}$.
    It is straightforward that each $M_{x,X}$ is well defined.
    
    Furthermore, we see that $M_{x,X} \subob M_{y,Y}$ if and only if either (a) $y \lneq x$ in $\mathcal{C}$, or (b) $x = y$ in $\mathcal{C}$ and $X\subob Y$ in $\Cat$.
    Now recall that, in a total ordered set, an interval $I$ is closed above (respectively, below) if $y\in I$ and $y\leq z$ (respectively, $z\leq y$) implies $z\in I$.
    For each interval $I$ of objects in $\supp(M)$ which is closed above (with respect to $\leq$), we define a functor $M_I$ in the following way.
    For each object $x$ in $\mathcal{C}$, define
    $$M_I(y) = \begin{cases} M(y) & y\in I \\ 0 & y\notin I.\end{cases}
    $$
    For each morphism $f:y\to y'$ in $\mathcal{C}$, we set $M_I(f)= M(f)|_{M_I(y)}$.
    We note that if $I$ has an infimum (with respect to $\leq$), then there are two possibilities. Either $\inf I \notin I$, and so $M_I = M_{x,0}$, or $\inf I \in I$, and so $M_I=M_{x,M(x)}$. If, on the other hand, $I$ does not have an infimum, then $M_I$ is not equal to $M_{x,X}$ for any object $x$ in $\mathcal{C}$ and $X\in\Omega_x$.
    
    We now see that
    $$\Delta = \{M_{x,X}\mid x\in\supp(M), X\in\Omega_X\} \cup \{M_I \mid I\subseteq\supp M \text{ is closed above}\}$$ is a subobject chain of $M$.
    It remains to show that $\Delta$ is a composition series and that the subfactors of $\Delta$ are as described.
    
    Let $x\in \supp(M)$
    and recall that $\Omega_x$ is complete.
    If $X\in\Omega_x$ has a predecessor $X^-_{\Omega_x}$, denote $Y:= X/X_{\Omega_x}^-$. Then $(M_{x,X})^-_\Delta = M_{x,X^-_{\Omega_x}}$, and so $M_{x,X}/(M_{x,X})^-_\Delta$ is the simple functor $\widetilde{Y}_x$.
    If $X\in\Omega_x$ does not have a predecessor, then $(M_{x,X})^-_\Delta = M_{x,X}$.
    Now suppose that $x$ has an immediate successor $y \in \supp(M)$ under the restriction of $\leq$; that is, $y$ is such that if $x \leq z \lneq y$ and $z \in \supp(M)$, then $z = x$. In this case, it follows that $M_{x,0} = M_{y,M(y)}$.
    Otherwise, we have $(M_{x,0})^-_\Delta = M_{x,0}$.
    Finally, suppose that $I$ is an interval of objects in $\supp(M)$ that is closed above such that $M_I\neq M_{x,X}$ for any object $x\in\mathcal{C}$ and $X\in\Omega_X$.
    Note that the intervals that are closed above also form a totally ordered set under inclusion.
    Then there exists no interval $I'$ such that $I'$ is the predecessor or successor to $I$ by inclusion.
    Thus, $(M_I)^-_\Delta = M_I$ and $(M_I)^+_\Delta = M_I$.

    It follows from the above paragraph that $$\fac(M) = \coprod_{x\in\supp M} \left\{\wS_x \mid S\in\fac(M(x)) \right\}$$ as claimed.
    This means $\Delta$ is a composition series if it is bicomplete. To see this, let $\Delta' \subseteq \Delta$ and, for each object $x$ in $\supp(M)$, denote by $I_x$ the interval of objects in $\supp(M)$ which is closed above and has minimal element $x$.
    Let
    \begin{align*}
        \Be =& \{I\subseteq \supp(M) \text{ closed above} \mid \exists x\in\supp(M), X\in\Omega_x : M_{x,X}\in\Delta', I\subseteq I_x\} \\
        &\cup \{I\subseteq \supp(M) \text{ closed above} \mid \exists J\subseteq \supp(M) \text{ closed above}: M_J\in \Delta', I\subseteq J\}.
    \end{align*}
    In the poset of intervals of $\supp(M)$ (ordered by inclusion), we note that $\Be$ itself is an interval that is closed below.
    In particular, $\be = \bigcup_{I\in\Be} I$ is the supremum of $\Be$.
    If $\be$ has no infimum in $\supp(M)$, $\colim\Delta' = M_\be \in \Delta$.
    Otherwise, there are two possibilities.
    If $\inf\be\notin\be$, then $\be = I_x\setminus\{x\}$ for some object $x$ in $\supp(M)$ and $\colim\Delta' = M_\be = M_{x,0} \in \Delta$.
    If $\inf(\be)\in\be$, then $\be=I_x$ for some object $x$ in $\supp(M)$.
    Then let $\Omega_x' = \{X\in \Omega_x \mid M_{x,X}\in \Delta'\}$.
    Since $\inf(\be)\in\be$, we see $\Omega_x'\neq\emptyset$.
    It follows that $\colim\Delta' = M_{x,\colim\Omega'_x} \in \Delta$. We conclude that $\Delta$ is complete. The argument that $\Delta$ is cocomplete is similar.
\end{proof}

We now specialize Proposition~\ref{prop:FunctorCategoryIsJHS} and Theorem~\ref{thm:pwf} to obtain Theorem~\ref{thm:intro:pwf} from the introduction.

\begin{corollary}[Theorem~\ref{thm:intro:pwf}]\label{cor:pwf}\
\begin{enumerate}
    \item Let $\mathcal{C}$ be a directed small category, let $\field$ be a field, and let $M:\mathcal{C}\rightarrow \veck$ be a pointwise finite-dimensional $\mathcal{C}$-persistence module. Then $M$ is Jordan--H\"older--Schreier and satisfies
    $$\fac(M) = \coprod_{x\in\supp M} \left( \coprod_{i=1}^{\dim_\field M(x)}\left\{ \widetilde{\field}_x\right\}\right).$$
    \item Let $\Cat$ be an abelian category such that every object of $\Cat$ is Jordan--H\"older--Schreier and let $X$ be a topological space. Then any presheaf on $X$ with values in $\mathcal{A}$ is Jordan--H\"older--Schreier.
\end{enumerate}
\end{corollary}
\begin{proof}
(1) This follows immediately from Theorem~\ref{thm:pwf}.
(2) The source category $\mathcal{T}op(X)^{op}$ of a presheaf is the opposite category of a directed category, and so itself is directed.
\end{proof}

\begin{example}\label{ex:injective}
    Consider the category $\Fun(\mathbb{R},\veck)$, where $\mathbb{R}$ is considered as a poset category as in Section~\ref{sec:fun cats}. Then for $x \in \mathbb{R}$, the functor $M_{[x,x]}$ is simple with injective envelope $M_{(-\infty,x]}$. The functor $M_{(-\infty,x]}$ is JHS (and uniserial) by Corollary~\ref{cor:pwf}. However, it is not artinian since $\Sub\left(M_{(-\infty,x]}\right)$ is isomorphic (as a lattice) to $(-\infty,x] \subseteq \mathbb{R}$. Note also that one can obtain non-uniserial examples by replacing the standard poset structure of $\mathbb{R}$ with the relation that $x \preceq y$ if either $0 \leq x \leq y$ or $y\leq x \leq 0$.
\end{example}

\begin{remark}\label{rmk:sheaf}
We note that the proof technique for Corollary \ref{cor:pwf} will not work for all sheaves $\mathcal{O}$ of a space $X$ with values in $\Cat$.
Suppose there is a covering $\{U_i\}$ of $U$ in $\mathcal{T}op(X)$ and a pair $s_j\in \mathcal{O}(U_j)$ and $s_k\in\mathcal{O}(U_k)$ such that $s_j|_{U_j\cap U_k}=s_k|_{U_j\cap U_k} \neq 0$. Then $\mathcal{O}_{U,0}$ is not a sheaf.
\end{remark}

It is also natural to consider whether quasi-coherent sheaves are weakly JHS. Specifically, consider a scheme $\mathbb{X}$ with locally ringed space $\mathcal{O}_{\mathbb{X}}$. We consider the category $\mathrm{Qch}(\mathbb{X})$ of quasi-coherent sheaves of $\mathcal{O}_\mathbb{X}$ modules. It is well-known that this category is Grothendieck, and so every object $\mathcal{O}$ in $\mathrm{Qch}(\mathbb{X})$ is (AB5). In particular, $\mathcal{O}$ satisfies the hypotheses of Theorem~\ref{prop:compExists}(2), and thus admits a composition series in $\mathrm{Qch}(\mathbb{X})$.

One could then ask whether the property of being (AB5$^*$) (and thus also being weakly JHS) is a local property. More precisely, let $\mathfrak R$ be the class of rings in $\mathcal{O}_{\mathbb{X}}$.
Suppose a property $\mathfrak P$ of modules of rings in a class $\mathfrak R$ satisfies: if $\{M_i\}_{i=1}^n$ are modules over $\{R_i\}_{i=1}^n$ that satisfy $\mathfrak P$, then the module $M_1\times\cdots\times M_n$ satisfies $\mathfrak P$ as a module over $R_1\times\cdots\times R_n$.
Then we say $\mathfrak P$ is \textdef{compatible with direct products of rings in $\mathfrak R$}.
We will use the following definition and lemma from \cite{EGT}.
\begin{definition}[Definition 3.4 in \cite{EGT}]
    Let $\mathfrak P$ be a property of modules of rings in $\mathfrak R$ compatible with direct products of rings in $\mathfrak R$.
    \begin{enumerate}
        \item Suppose that for any flat $\phi:R\to S$ in $\mathfrak R$ and $R$-module $M$, if $M$ satisfies $\mathfrak P$ then $M\otimes_R S$ satisfies $\mathfrak P$ as an $S$-module.
        Then we say $\mathfrak P$ \textdef{ascends in $\mathfrak R$}.
        \item Suppose that for any faithfully flat $\phi:R\to S$ in $\mathfrak R$ and $R$-module $M$, if $M\otimes_R S$ satisfies $\mathfrak P$ as an $S$-module then $M$ satisfies $\mathfrak P$.
        Then we say $\mathfrak P$ \textdef{descends in $\mathfrak R$}.
    \end{enumerate}
    If $\mathfrak P$ both ascends and descends in $\mathfrak R$ then we say $\mathfrak P$ is an \textdef{ascent-descent} property.
\end{definition}

\begin{lemma}[Lemma 3.5 in \cite{EGT}]
    Let $\mathfrak R$ be a class of commutative rings and $\mathfrak P$ be a property of modules of rings in $\mathfrak R$ compatible with direct products of rings in $\mathfrak R$.
    If $\mathfrak P$ is an ascent-descent property then the notion of a $\mathfrak P$ quasi-coherent sheaf of $\mathcal{O}_{\mathbb{X}}$-modules is local.
\end{lemma}

Let $\mathfrak R$ be a class of commutative rings.
For two (modular) lattices $\mathcal L$ and $\mathcal L'$, we denote by
$\mathcal{L}\times\mathcal{L}'$ the product of $\mathcal{L}$ and $\mathcal{L}'$, where $\leq$, $\wedge$, and $\vee$ are given component-wise.

If $\mathcal L$ and $\mathcal L'$ are both (weakly) JHS, then so is $\mathcal{L}\times\mathcal{L}'$.
When we consider lattices of submodules, we see that the property (AB5$^*$) is compatible with products of rings in $\mathfrak R$.

Thus, whether or not being (weakly) JHS is a local property of quasi-coherent sheaves of $\mathcal{O}_{\mathbb{X}}$-modules might be solved by asking if (AB5$^*$) is an ascent-descent property in $\mathfrak{R}$.


\subsection{Igusa and Todorov's representations of $\R$}\label{sec:ITreps}
In this section we show that the (nonzero) representations examined by Igusa and Todorov in \cite{IT15} are weakly JHS, but are not JHS. We recall from Section \ref{sec:fun cats} that for $\field$ an arbitrary field, we denote by $\Fun(\mathbb{R},\veck)$ the category of (covariant) functors from $\mathbb{R}$, considered as a poset category, to the category of finite-dimensional $\field$-vector spaces. Moreover, given $M \in \Fun(\mathbb{R},\veck)$ and $x \leq y \in \mathbb{R}$, we denote by $M(x,y)$ the result of applying $M$ to the unique morphism $x\rightarrow y$ in $\mathbb{R}$.

\begin{definition}\label{def:ARc}
    We denote by $\IT$ the full subcategory of $\Fun(\mathbb{R},\veck)$ consisting of functors $M:\R\to \veck$ be a functor which satisfy the following.
    \begin{enumerate}[label={(IT\arabic*)}]
        \item\label{IT1} For all $x\in\R$, we have $\lim_{y>x}V(y)=V(x)$ and
        \item\label{IT2} For all but finitely many $x \in \mathbb{R}$, there exists $\e > 0$ such that $M(x,x+\e)$ is an isomorphism.
    \end{enumerate}
\end{definition}

In \cite{IT15}, Igusa and Todorov studied functors $\R^{op}\to \veck$ that satisfied the symmetric condition to~\ref{IT1} and the same condition~\ref{IT2} in Definition \ref{def:ARc}.
We may then use their results.

\begin{theorem}[Adapted from \cite{IT15}]\label{thm:igusaTodorv}\
    \begin{enumerate}
    \item The category $\IT$ is an abelian and extension closed subcategory of $\Fun(\R,\veck)$.
    \item For every interval $I$ that is closed below and open above, there is an interval indecomposable $M_I$ in $\IT$ given by
    \begin{align*}
        M_I(x) &= \begin{cases}
            \field & x\in I \\
            0 & x\notin I
        \end{cases}
        &
        M_I(x,y) &= \begin{cases}
            1_\field & x\leq y\in I \\
            0 & \text{otherwise}.
        \end{cases}
    \end{align*}
    \item Every representation in $\IT$ is isomorphic to a finite direct sum of interval indecomposables. This direct sum is unique up to permuting the direct summands.
    \end{enumerate}
\end{theorem}
We note that (2) includes intervals of the form $[a,+\infty)$ for every $a\in\R$. By results in \cite{C-B, GR,Continuous A 1}, each object in $\Fun(\R,\veck)$ is also a direct sum of interval indecomposables, but the intervals may be of any form and the sum may be infinite.

We have already shown in Corollary~\ref{cor:pwf} that each object $M$ in $\IT$ is JHS when considered as an object in $\Fun(\mathbb{R},\veck)$. However, as a subcategory of $\Fun(\mathbb{R},\veck)$, $\IT$ is not closed under subobjects. As a result, the colimits and limits of subobject chains will generally differ depending on which category we are working in. The following lemma describes this relationship precisely.
To increase readability in the proof of the lemma, we use $\bigoplus$ for direct sums of objects and $\sum$ for sums of morphisms using the abelian structure of $\IT$.

\begin{lemma}\label{lem:ARcIsBicomplete}
    Let $X$ be an object in $\IT$ and $\emptyset \neq \Delta$ a subobject chain of $X$ in $\IT$.
    \begin{enumerate}
        \item Let $\colim_{\Fun}\Delta :=\bigoplus_\lambda M_{I_\lambda}$ be the colimit of $\Delta$ in $\Fun(\R,\veck)$.
        Then the colimit of $\Delta$ in $\IT$ is given by $\colim_{\mathsf{IT}}\Delta:= \bigoplus_\lambda M_{I_\lambda\cup\inf(I_\lambda)}.
        $
        \item The limit $\lim_\Fun\Delta$ of $\Delta$ in $\Fun(\R,\veck)$ is also the limit of $\Delta$ in $\IT$.
    \end{enumerate}
\end{lemma}
\begin{proof}
    By Proposition~\ref{prop:FunctorCategoryIsJHS}, we know $\colim_\Fun\Delta$ and $\lim_\Fun\Delta$ exist and are subfunctors of $M$. By \cite[Theorem~3.0.1]{Continuous A 1}, we know that $M$, and thus $\colim_\Fun\Delta$ and $\lim_\Fun\Delta$, are finitely generated in $\Fun(\mathbb{R},\veck)$. In particular, this means $\colim_\Fun\Delta$ and $\lim_\Fun\Delta$ are finite direct sums of interval indecomposables. It follows that both $\lim_\Fun\Delta$ and $\colim_\Fun\Delta$ satisfy~\ref{IT2}.
    
    To see that $\lim_\Fun\Delta$ satisfies~\ref{IT1}, we note that for all $x \in \mathbb{R}$, we have that (in $\veck$) $$\lim_{y > x}\left(\mathrm{lim}_\Fun \Delta(y)\right) = \lim_{N \in \Delta}\left(\lim_{y > x} N(y)\right).$$
    Since each $N \in \Delta$ is in $\IT$, this implies that $\lim_{y > x}\left(\lim_\Fun \Delta(y)\right) = \lim_\Fun \Delta(x)$, as desired.
    
    Now write $\colim_\Fun \Delta = \bigoplus_{i = 1}^m M_{I_i}$ and $M = \bigoplus_{k = 1}^p M_{K_k}$, where each $M_{I_i}$ and $M_{K_k}$ is an interval indecomposable. Let $\iota:\colim_\Fun \Delta \rightarrow M$ be the inclusion map and write $\iota = \sum_{i = 1}^m \sum_{k = 1}^p \iota_{i,k}$, with each $\iota_{i,k}:M_{I_i}\rightarrow M_{K_k}$. We claim that each $I_i$ is open on the right. Indeed, if $I_i$ is closed on the right, then there exists some $K_k$ so that $(\iota_{i,k})_{\sup I_i}: M_{I_i}(\sup I_i)\rightarrow M_{K_k}(\sup I_i)$ is nonzero. But this can only happen if $\supp I_i = \supp K_k \in K_k$, a contradiction. In particular, this implies that $\colim_\IT\Delta$ is indeed an object in $\IT$.
    
    It remains to show that $\colim_\IT\Delta$ satisfies the universal property of the colimit. To see this, let $N = \oplus_{j = 1}^n M_{J_j}$ be an object in $\IT$ and let $g = \sum_{i = 1}^m \sum_{j = 1}^n g_{i,j}: M_{I_i}\rightarrow M_{J_j}$ be a morphism in $\Fun(\mathbb{R},\veck)$. Since each $J_j$ is closed on the left by assumption, it follows that each $g_{i,j}$ factors through $M_{I_i\cup\inf I_i}$. Thus, $g$ factors through $\colim_{\IT}\Delta$. 
    This proves the result.
\end{proof}

\begin{proposition}\label{prop:ARcIsWJHSNotJHS}
    Every nonzero object $X$ in $\IT$ is weakly JHS but not JHS.
\end{proposition}
\begin{proof}
    We first show every object is weakly JHS and then show that $\IT$ has no simple objects. This will imply that $\fac(M) = \emptyset$ for all objects $M$ of $\IT$, and so no nonzero object in $\IT$ is~JHS.
    
    Let $M$ be an object of $\IT$. We note that  $\IT$ has exact $M$-bounded direct limits and that $M$ is (AB5$^*$) in $\IT$ as an immediate consequence of Proposition~\ref{prop:FunctorCategoryIsJHS} and Lemma~\ref{lem:ARcIsBicomplete}.
    Thus, $M$ satisfies~(AB5) in $\IT$ by Theorem~\ref{thm:catComp_AB5}.
    
    Now we show that $\IT$ contains no simple objects.
    For contradiction, suppose $S$ is simple in $\IT$.
    Then $S\cong M_I$ for some interval $I$ that is closed below and open above.
    However, this means that $\inf I\neq \sup I$, and so there exists $\e>0$ such that $\inf I < \e+\inf I < \sup I$.
    Then $M_{[\e+\inf I,\sup I)}$ is a subobject of $S$, a contradiction.
\end{proof}

\begin{example}\label{ex:CompSeriesInARc}
    Let $I=[a,b)$ for a pair of real numbers $a<b$.
    Then
    \begin{displaymath}
        \Delta = \{M_{[x,b)} \mid a\leq x < b\}\cup\{0\}
    \end{displaymath}
    is a composition series of $M_I$ in $\IT$ which satisfies $\fac\Delta = \emptyset$.
\end{example}


\subsection{Non-example: Infinite products and coproducts}\label{sec:grothendieck}
We conclude by showing that infinite products and coproducts are often not JHS, and may even admit two composition series with different cardinalities of subfactor multisets. While the results of this section hold more generally, we work only in module categories to simplify the arguments.

\begin{proposition}
    Let $R$ be a ring, $\Cat$ the category of left $R$-modules, and $S$ a simple object in $\Cat$. Then
    \begin{enumerate}
        \item $\coprod_\mathbb{N} S$ is not (AB5$^*$).
        \item $\prod_{\mathbb{N}}S$ is not (AB5$^*$).
        \item $\prod_{\mathbb{N}}S$ admits composition series $\Delta$ and $\Gamma$ with $$|\fac(\Delta)| = |\mathbb{N}| \neq \left|\End_\Cat(S))\right|^{|\mathbb{N}|} = |\fac(\Gamma)|.$$
    \end{enumerate}
\end{proposition}

\begin{proof}
    Denote $C:= \coprod_{\mathbb{N}}S$ and $P:= \prod_{\mathbb{N}}S$. For $n \in \mathbb{N}$, we denote by $\iota_n: S \rightarrow C$ and $q_n: P\rightarrow S$ the $n$-th inclusion and projection maps. Moreover, we denote $C_n:= \coprod_{i = 0}^n S$ and $C_n^c := \coprod_{\mathbb{N}\setminus\{[0,n]\}}$, and likewise for $P_n$ and $P_n^c$. We consider each $C_n$ and $C_n^c$ as subobjects of $C$ and each $P_n$ and $P_n^c$ as quotient objects of $P$ via the inclusion or projection to the relevant coordinates. Finally, we denote by $\natural: C \rightarrow P$ the natural (inclusion) map.

    (1) Let $\sigma: C \rightarrow S$ be the summation map; i.e., the unique map for which $\sigma \circ \iota_n = 1_S$ for all $n$. We claim that $\Delta_C:= \{C_n^c\mid n \in \mathbb{N}\}$ and $\ker \sigma$ fail to satisfy the (AB5$^*$) condition. Indeed, first note that $\bigwedge \Delta_C = 0$. Now let $m, n \in \mathbb{N}$. If $m > n$, then $\mathrm{im}(\iota_m) \subob C^c_n$, and if $m \leq n$, then $\iota_m = (\iota_m - \iota_{n+m}) + \iota_{n+m}$ and $\mathrm{im}(\iota_m - \iota_{n+m}) \subob \ker \sigma$. It follows that $C = \bigvee_m \{\mathrm{im}(\iota_m)\} = C^c_n + \ker \sigma$. Thus
    $$\bigwedge\left(\Delta_C + \ker \sigma\right) = C \neq \ker\sigma = \left(\bigwedge \Delta_C\right) + \ker \sigma.$$
    We conclude that $C$ is not (AB5$^*$).

    (2) Let $\Delta_P = \{P_n^c\mid n \in \mathbb{N}\}$, where each $P_n^c$ is considered as a subobject of $P$ via inclusion of the relevant coordinates. Then in particular $P_n^c + \mathrm{im}(\natural) = P$ for all $n \in \mathbb{N}$. Moreover, we have $\bigwedge \Delta_P = 0$. We conclude that
    $$\bigwedge\left(\Delta_P + \mathrm{im}(\natural)\right) = P \neq \mathrm{im}(\natural) = \left(\bigwedge \Delta_P\right) + \mathrm{im}(\natural),$$
    and so $P$ is not (AB5$^*$).

    (3) We first claim that $\Delta := \Delta_P \cup \{0\}$ is a composition series of $P$. Indeed, it is clear that $\Delta \cup \{0\}$ is bicomplete, since any $\Delta' \subseteq \Delta$ contains a maximal element and either contains a minimal element or satisfies $\bigwedge\Delta' = \bigwedge\Delta_0 = 0$. Moreover, we have $P^-_{\Delta} = P^c_0$ and $P/P^c_0 \cong S$. Likewise, for $n \in \mathbb{N}$, we have $(P^c_n)^-_{\Delta} = P^c_{n+1}$ and $P^c_n/P^c_{n+1} \cong S$. We conclude that all of the subfactors of $\Delta$ are simple, and therefore $\Delta$ is a composition series which satisfies $|\fac(\Delta)| = |\mathbb{N}|$.

    We now construct a second composition series of $P$ with larger cardinality. To begin, we note that since $S$ is simple, the ring $\mathrm{End}_\Cat(S)$ a division ring. Thus we have that
        $$\Hom_\Cat\left(S,P\right)\cong \prod_\mathbb{N}\mathrm{End}_\Cat(S)^{op}$$
    is a free $\mathrm{End}_\Cat(S)^{op}$-module. It is well-known that the $(\End_\Cat(S)^{op})$-dimension of this module is $|\mathrm{End}_\Cat(S)|^{|\mathbb{N}|}$, which in particular is uncountable. See e.g.~\cite[Theorem XI.5.2]{jacobson}. Thus choose some basis $B$ of $\Hom_\mathcal{S}(S,P)$ and choose a well-order $\leq$ on $B$ for which $B$ contains a minimal element $b_0$ and maximal element $b_1$.

    Now for any subset $B'\subseteq B$, denote
    $$X_{B'} := \sum_{b \in B'} \mathrm{im}(b).$$
    We note that if $B' \subsetneq B'' \subseteq B$, then there is a proper inclusion $X_{B'}\hookrightarrow X_{B''}$. Moreover, we have $X_B = P$. Since $B$ is well-ordered, we then have a subobject chain
    $$\Gamma = \left\{0,X_{\{b_0\}},X_B\right\} \cup \left\{ X_{[b_0,b)}\mid b \in B\setminus\{b_0\}\right\}$$
    of $P$. Now let $\Gamma'\subseteq \Gamma$. If $0 \in \Gamma'$, then clearly $\bigwedge\Gamma' = 0 \in \Gamma$. Otherwise, define
        $$J_0 := \bigcap_{I \subseteq B: X_I \in \Gamma'} I\qquad\quad\text{ and } \quad\qquad J_1 := \bigcup_{I \subseteq B: X_I \in \Delta'} I.$$
        It follows that $\bigwedge\Gamma' = X_{J_0} \in \Gamma$ and $\bigvee\Gamma' = X_{J_1} \in \Gamma$. We conclude that $\Delta$ is bicomplete. Moreover, we have $0^+_\Gamma = X_{\{b_0\}}$ and $X_{\{b_0\}}/0 \cong S$. Likewise, for $b \in B\setminus\{b_0\}$, we have
    $\left(X_{[b_0,b)}\right)^+_\Gamma = X_{[b_0,b]}$ and $X_{[b_0,b]}/X_{[b_0,b)} \cong S$. We conclude that $\Gamma$ is a composition series and $|\fac(\Gamma)| = |B| \gneq |\mathbb{N}|$.
\end{proof}

\section{On Direct Sum Decompositions}\label{sec:directSum}

While we have focused primarily on the existence and uniqueness of composition series, another result which is fundamental to the study of abelian length categories is the \emph{Krull--Remak--Schmidt Theorem}. This result states that every object of finite (nonzero) length can be expressed as a direct sum of (finitely many) indecomposable subobjects, each with local endomorphism ring. Furthermore, this decomposition is unique up to isomorphism and permutation of the direct summands. See e.g. \cite[Theorem~X.7.5]{Lang} for a proof in module categories.

As with the Jordan--H\"older and Schreier Theorems, there have been a plethora of generalizations of the Krull--Remak--Schmidt Theorem in the past century. One of the most famous examples is the Krull--Remak--Schmidt--Azumaya Theorem \cite{KRSA}, which states that if $X = \bigoplus_{\alpha \in I} Y_\alpha$ is a decomposition of a module $X$ into a direct sum of indecomposable modules with local endomorphism rings and $Z \subob X$ is an indecomposable direct summand of $X$, then there exist $W \subob X$ and $\alpha \in I$ such that $X = W\oplus Z = W\oplus Y_\alpha$. This in particular implies that there is a bijection $\Phi:I \rightarrow J$ such that $Y_\alpha \cong Z_{\Phi(\alpha)}$ for all $\alpha \in I$. An even stronger notion of uniqueness is obtained by asking the decomposition $X = \bigoplus_{\alpha \in I} Y_\alpha$ to ``complement direct summands'', see Definition~\ref{def:complements}.

It has long been known that many results on direct sum decompositions can be formulated in the language of lattice theory. In particular, much attention has been devoted to establishing decomposition theories in meet- and/or join-continuous lattices (both modular and not). See  \cite{DC,sch,sem,simmons,wal} and the references therein for some examples of historical and modern approaches\footnote{Note that some of these references focus on \emph{irredundant decompositions} using \emph{completely join-irreducible elements}. These generally differ from the \emph{independent decompositions} using \emph{indecomposable elements} considered in the present paper. For example, every completely join-irreducible element is indecomposable, but the converse may not hold.}. This includes \emph{algebraic} lattices (see Definition~\ref{def:local_fin_gen}), which inherit many nice properties from module theory.

The recent paper \cite{IY}, Ibrahim and Yousif show that every (AB5$^*$) module $X$ (over any ring) can be expressed as a direct sum $X = \bigoplus_{\alpha \in I} Y_\alpha$ of indecomposable submodules. If moreover $X$ satisfies the ``finite exchange property'' (see Definition~\ref{def:exchange} below), then each $Y_\alpha$ has local endomorphism ring and the decomposition complements direct summands. In light of the fact that the classes of weakly JHS modules and (AB5$^*$)-modules  coincide (see Remark~\ref{rem:moduleJHS3}(2)), some of the proofs from \cite{IY} are readily adapted to the large classes of weakly JHS objects and lattices. In this section, we outline a plan for future work in this direction and prove some preliminary results. 

As in many of the earlier sections, we maintain the convention that $\mathcal{L}$ denotes a bounded modular lattice, which we assume to be complete, and that $\Cat$ denotes a well-powered abelian category.


\subsection{Independent sets and decompositions}

In this section, we recall the lattice-theoretic notions which generalize direct sum decompositions in abelian categories. The following definitions are motivated by the notion of (Von Neumann) independence, see e.g. \cite[Section~V.1.6]{gratzer_book} or \cite[Section~V.1.5]{birkhoff_book}.

\begin{definition}\label{def:independent}
    Let $\Delta \subseteq \mathcal{L}$ and suppose that $\hat{0} \notin \Delta$.
    \begin{enumerate}
        \item We say that $\Delta$ is \textdef{weakly independent} if for all $y \in \Delta$ and for all finite subsets $\Delta' \subseteq \Delta \setminus \{y\}$, one has $y \wedge \left(\bigvee (\Delta'\setminus\{y\})\right) = \hat{0}$.
        \item We say that $\Delta$ is \textdef{independent} if for all $\Delta' \subseteq \Delta$, one has
        $\left(\bigvee \Delta'\right) \wedge \left(\bigvee (\Delta\setminus\Delta')\right) = \hat{0}$. In this case, we write $\biguplus \Delta := \bigvee \Delta$; i.e., if we write $\biguplus\Delta$, we are assuming that $\Delta$ is independent.
        \item We say that $\Delta$ is a \textdef{independent decomposition} of $\mathcal{L}$ (respectively of $\bigvee \Delta$) if $\Delta$ is independent and $\biguplus \Delta = \hat{1}$ (respectively $\biguplus \Delta \leq \hat{1}$). In this case, we say each $y \in \Delta$ is a \textdef{direct summand} of $\mathcal{L}$ (respectively of $\bigvee \Delta = \biguplus \Delta)$.
        \item For $y \in \Delta$, we say that $y$ is \textdef{indecomposable} if given $z, w \in \mathcal{L}$ with $z \vee w = y$ and $z \wedge w = \hat{0}$, one must have that $z = \hat{0}$ and $w = y$ or vice versa.
        \item We say that $\Delta$ is an independent decomposition of $\mathcal{L}$ (respectively of $\bigvee \Delta$) \textdef{into indecomposables} if $\Delta$ is an independent decomposition of $\mathcal{L}$ (respectively of $\bigvee \Delta$) and every $y \in \Delta$ is indecomposable.
    \end{enumerate}
\end{definition}

\begin{remark}\label{rem:directSumDefs}
    Let $X$ be a subobject complete object in $\Cat$, and let $\emptyset \neq \Delta \subseteq \mathcal{L}$. Then $\Delta$ is independent if and only if $\bigvee\Delta = \bigoplus_{Y \in \Delta} Y$, where here $\bigoplus$ denotes the internal direct sum. Moreover, a subobject $Y \in \Delta$ is indecomposable (respectively a direct summand of $\Sub(X)$) in the sense of Definition~\ref{def:independent} if and only if it is indecomposable (respectively a direct summand of $X$) as an object in $\Cat$.
\end{remark}

\begin{notation}
    In light of Remark~\ref{rem:directSumDefs}, we will use $\biguplus \Delta$ and $\bigoplus_{Y \in \Delta} Y$ interchangeably when working with objects in abelian categories.
\end{notation}

We note that $\mathcal{L}$ always admits an independent decomposition $\hat{1} = \biguplus\left\{\hat{1}\right\}$. On the other hand, $\mathcal{L}$ may not admit an independent decomposition into indecomposables. For example, let $X = \prod_{i = 1}^\infty \mathbb{Z}$, considered as a $\mathbb{Z}$-module. It is well-known that $\mathbb{Z}$ is not a free module, and thus can not be expressed as a direct sum of indecomposable submodules. It follows that the lattice $\mathcal{L} = \Sub(X)$ does not admit an independent decomposition into indecomposables.

The following allows us to simplify the definition of independence when dealing with (AB5) lattices. In particular, the analogous statement is well-known in Grothendieck categories.

\begin{lemma}\label{lem:indep}
    Suppose $\mathcal{L}$ is (AB5), and let $\Delta \subseteq \mathcal{L}$ be a nonempty subset. Then $\Delta$ is independent if and only if it is weakly independent.
\end{lemma}

\begin{proof}
    Suppose $\Delta$ is weakly independent, and let $\Delta' \subseteq \Delta$. Denote by $\mathfrak{F}$ and $\mathfrak{G}$ the sets of finite subsets of $\Delta'$ and $\Delta \setminus \Delta'$, respectively. Then $\Omega_1 := \{\bigvee \Gamma \mid \Gamma \in \mathfrak{F}\}$ and $\Omega_2:= \{\bigvee \Gamma' \mid \Gamma' \in \mathfrak{G}\}$ are (upward) directed subsets of $\mathcal{L}$ which satisfy $\bigvee \Omega_1 = \bigvee \Delta'$ and $\bigvee \Omega_2 = \bigvee(\Delta\setminus\Delta')$. Since $\mathcal{L}$ is (AB5) and $\Delta$ is weakly independent, it the follows that
    $$
        \left(\bigvee \Delta'\right) \wedge \left(\bigvee (\Delta \setminus \Delta')\right) = \left(\bigvee \Omega_1\right) \wedge \left(\bigvee \Omega_2 \right)
            = \bigvee \left(\Omega_1 \wedge \Omega_2\right)
           = \hat{0}.
    $$
    We conclude that $\Delta$ is independent. The reverse implication is trivial.
\end{proof}


\subsection{Existence of direct sum decompositions}

In this section, we discuss the existence of direct sum decompositions, and more generally independent decompositions into indecomposables, for weakly JHS objects and lattices. We first consider the following non-example.

\begin{example}\label{ex:boolean}
    Let $\mathcal{L}$ be a non-trivial atomless complete Boolean algebra. (A typically example is the algebra generated by all half-open intervals $[a,b)$ with $a \in \mathbb{Q}$ and $b \in \mathbb{Q} \cup \{\infty\}$. See e.g. \cite[Chapter~16]{atomless}.) Since Boolean algebras are completely semidistributive, they are modular and weakly JHS. Moreover, let $\hat{0} \neq x \in \mathcal{L}$. Since $\mathcal{L}$ is atomless, we can choose $\hat{0} < y < x$. Denoting by $y^c$ the complement of $y$, this yields an independent decomposition $\{y, x \wedge y^c\}$ of $x$. Thus $\mathcal{L}$ has no indecomposable elements. 
\end{example}

As shown in the example, the existence of complements in a nontrivial atomless complete Boolean algebra $\mathcal{L}$ implies that every independent decomposition of $\mathcal{L}$ can be ``further decomposed'', but that $\mathcal{L}$ does not admit an independent decomposition into indecomposables. We will call lattices with this property \textdef{superdecomposable}. It will be a consequence of Lemma~\ref{lem:countable} that every nontrivial weakly JHS lattice which does not contain an indecomposable direct summand is superdecomposable.

\begin{remark}
    Example~\ref{ex:boolean} can also be modified into an example of a superdecomposable weakly JHS object in an abelian (even Grothendieck) category. Indeed, in \cite[Theorem~3]{roos}, it is shown that for every boolean algebra $\L$ one can associate a ``locally distributive spectral category'' $\Cat_\L$. This is a Grothendieck category in which every short exact sequence is split. Moreover, if $\L$ is atomless (or more generally non-atomic), then $\Cat_\L$ contains a nontrivial superdecomposable object $X$ such that $\Sub(X)$ is completely distributive (and thus weakly JHS), see \cite[Remark~2]{roos} and \cite[Proposition~6.7]{stenstrom}.
\end{remark}

Example~\ref{ex:boolean} stands in contrast to this situation in module categories. Indeed, it is shown in \cite{IY} that every (AB5$^*$) module can be expressed as a direct sum of indecomposable submodules. The proof in that paper relies on the fact that every (AB5$^*$)-module contains an indecomposable direct summand, a result we have already shown does not extend to generally (AB5$^*$) lattices in Example~\ref{ex:boolean}. Towards that end, we consider the following definitions.

\begin{definition}\label{def:local_fin_gen}\
Let $\mathcal{L}$ be a complete lattice.
   \begin{enumerate}
        \item An element $x \in \mathcal{L}$ is called \textdef{compact} if, given a directed subset $\Delta \subseteq \mathcal{L}$ for which $x \leq \bigvee \Delta$, there exists $y \in \Delta$ such that $x \leq y$. We denote by $\mathcal{L}^c$ the set of compact element of $\mathcal{L}$.
        \item We say that the lattice $\mathcal{L}$ is \textdef{weakly algebraic} if $\hat{1} = \bigvee \L^c$.
        \item The lattice $\mathcal{L}$ is said to be \textdef{algebraic} (also known as \emph{compactly generated}) if for every $x \in \mathcal{L}$ one has $x = \bigvee \{y \in \mathcal{L}^c \mid y \leq x\}$.
   \end{enumerate}
\end{definition}

\begin{definition}\label{def:local_fin_gen_2}\
Let $X$ be a subobject complete object in $\Cat$.
   \begin{enumerate}
        \item We say a subobject $Y \subob X$ is \textdef{finitely generated} if $Y$ is compact as an element of the lattice $\Sub(X)$.
        \item We say that $X$ is \textdef{locally finitely generated} if $\Sub(X)$ is weakly algebraic.
        \item We say that $X$ is \textdef{locally finitely presented} if $\Sub(X)$ is algebraic.
   \end{enumerate}
\end{definition}

We note that the lattice-theoretic definitions of the notions in Definition~\ref{def:local_fin_gen_2} intentionally avoid some technicalities. Indeed, the more general definition of a finitely generated object $X$ in a cocomplete abelian category $\mathcal{A}$ asks that the natural map $\varinjlim_\alpha \Hom(X,Y_\alpha) \rightarrow \Hom\left(X,\varinjlim Y_\alpha\right)$ be a monomorphism for every directed system $(Y_\alpha)_{\alpha \in I}$. This coincides with the definition given above in the case of a Grothendieck category. Likewise, a cocomplete abelian category $\Cat$ is said to be locally finitely generated (resp. presented) if the full subcategory of finitely generated (resp. presented) objects is skeletally small and every object can be expressed as a direct limit of finitely generated (resp. presented) objects, see e.g. \cite{CB,PR}. On the other hand, it is known that a complete lattice is algebraic if and only if it is locally finitely presented as a category, see \cite{porst}.

\begin{example}\label{ex:locally_fp}\
\begin{enumerate}
    \item Let $\mathcal{A}$ be the category of (left) modules over some ring $R$. Then every object in $\mathcal{A}$ can be expressed as the sum of its finitely-generated submodules, and thus is locally finitely generated.
    \item Let $\mathcal{C}$ be a small category and $\mathcal{A}$ an abelian category. Then a functor $M \in \Fun(\mathcal{C},\Cat)$ is locally finitely generated if and only if $M(c)$ is locally finitely generated for all $c \in \mathcal{C}$. Indeed, given objects $c_1,\ldots,c_n$ in $\mathcal{C}$ and finitely-generated subobjects $N_i \subob M(c_i)$, the smallest subfunctor $N \langle(c_1,N_1),\ldots,(c_n,N_n)\rangle \subob M$ which satisfies $N_i \subob N(c_i)$ for all $i$ will be finitely generated in $\Fun(\mathcal{C},\Cat)$. Then $M$ is the supremum of such objects over all finite subsets of objects in $\mathcal{C}$ and all choices of finitely-generated $N_i \subob M(c_i)$. In particular, pointwise finite-dimensional $\mathcal{C}$-persistence modules over a field $\field$ and presheaves taking values in locally finitely generated abelian categories are both examples of locally finitely generated objects. See also the introduction of \cite{BC-B}, where it is discussed that the category of $\mathcal{C}$-persistence modules over $\field$ is always locally finitely presented.
\end{enumerate}
\end{example}

It is well-known that every algebraic lattice is meet-continuous, and therefore that every bounded modular algebraic lattice is (AB5). In particular, it follows immediately from Corollary~\ref{cor:catComp_AB5_2} that if every object in $\mathcal{A}$ is locally finitely generated then $\mathcal{A}$ has exact bounded direct limits.

In the remainder of this section, we prove the following. We then discuss uniqueness properties of decompositions in Section~\ref{sec:directSumUnique}. Note that, by Example~\ref{ex:locally_fp}, this generalizes the results from \cite{IY} on (AB5$^*$)-modules and from \cite{BC-B} on persistence modules.

\begin{theorem}\label{thm:decomp_existence}
    Suppose $\mathcal{L}$ is weakly JHS and weakly algebraic. Then $\mathcal{L}$ admits an independent decomposition of into indecomposables. In particular, if $X$ is an object in $\mathcal{A}$ which is both locally finitely generated and weakly JHS, then $X$ can be expressed as a direct sum of indecomposable submodules.
\end{theorem}

We begin with the following definition, referring to \cite[Definition~13.1]{CLVW} for the module-theoretic version.

\begin{definition}\label{def:local}
    Let $\Delta \subseteq \mathcal{L}$ be an independent set. We say that $\Delta$ is a \textdef{local direct summand} of $\mathcal{L}$ (respectively of $\bigvee \Delta$) if for every finite subset $\emptyset \neq \Delta' \subseteq \Delta$, the supremum $\bigvee \Delta' (= \biguplus \Delta')$ is a direct summand of $\mathcal{L}$ (respectively of $\bigvee \Delta$).
\end{definition}

The following is clear from the definitions.

\begin{lemma}\label{lem:summand_of_summand} Let $x \in \mathcal{L}$ be a nonzero element.
\begin{enumerate}
    \item If $y$ is a direct summand of $x$, then $y$ is a direct summand of $\mathcal{L}$.
    \item If $\Delta$ is a local direct summand of $x$, then $\Delta$ is a local direct summand of $\mathcal{L}$.
\end{enumerate}
\end{lemma}

In the theory of modules, one has that the supremum of every local direct summand of $X$ is itself a direct summand of $X$ if and only if the set of direct summands of $X$ is closed under directed unions, see \cite[13.2]{CLVW}. In \cite[Theorem~2.3]{IY}, it is shown that (AB5$^*$) modules satisfy these conditions. The statement and proof of their result can be readily adapted to arbitrary weakly JHS lattices as follows.

\begin{lemma}\label{lem:local}
    Suppose $\mathcal{L}$ is weakly JHS.
    \begin{enumerate}
        \item Let $\Delta \subseteq \mathcal{L}$ be a chain. If every $y \in \Delta$ is a direct summand of $\mathcal{L}$, then $\bigvee \Delta$ is a direct summand of $\mathcal{L}$.
        \item Let $\Delta \subseteq \mathcal{L}$ be a local direct summand of $\mathcal{L}$. Then $\biguplus \Delta$ is a direct summand of $\mathcal{L}$. In particular, if $\bigvee \Delta = \hat{1}$, then $\hat{1} = \biguplus \Delta$ is an independent decomposition of $\mathcal{L}$.
    \end{enumerate}
\end{lemma}

\begin{proof}
    (1) We adapt the proof of \cite[Theorem~2.3]{IY}. If $\bigvee \Delta = \hat{1}$, there is nothing to show. Otherwise, denote $$\mathcal{F} := \left\{(\Gamma,z) \in 2^\Delta \times \mathcal{L} \ \middle| \ \hat{1} = \biguplus\{z,\bigvee\Gamma\}\right\}.$$
    We partial order $\mathcal{F}$ via the relation $(\Gamma,z) \preceq (\Gamma', z')$ if $\Gamma \subseteq \Gamma'$ and $z' \leq z$. Note in particular that $\mathcal{F}$ is nonempty since $y = \bigvee\{y\}$ is a direct summand of $\mathcal{L}$ for every $y \in \Delta$.
    
    Now let $\{(\Gamma_\be,z_\be)\}_{\be \in \Be}$ be a totally ordered subset of $\mathcal{F}$. Let $\Gamma = \bigcup_{\be \in \Be} \Gamma_\be$ and $z:= \bigwedge \{z_\be \mid \be \in \Be\}$. We will show that $(\Gamma,z) \in \mathcal{F}$, and so Zorn's lemma applies.
    
    First note that, since $(\bigvee \Gamma) \vee z_\be = \hat{1}$ for all $\be \in \Be$, the fact that $\mathcal{L}$ is (AB5$^*$) implies that $(\bigvee \Gamma) \vee z = \hat{1}$. In particular, the assumption that $\bigvee \Gamma \neq \hat{1}$ means that $z \neq \hat{0}$. Moreover, since $\mathcal{L}$ is (AB5), we have that
    $$\left(\bigvee\Gamma\right) \wedge z = \bigvee\left\{\Gamma_\be \wedge z\mid \be \in \Be\right\} = \hat{0}.$$
    We conclude that $(\Gamma,z) \in \mathcal{F}$, as claimed.
    
    Now let $(\Gamma,z) \in \mathcal{F}$ be maximal. To finish the proof, we must show that $\Gamma = \Delta$. Indeed, suppose this is not the case and let $y \in \Delta \setminus \Gamma$. We then have two cases to consider.
    
    First suppose that $\bigvee \Gamma \leq y$. Then $y = \bigvee (\Gamma \cup \{y\})$ and, using a straightforward adaptation of the argument in \cite[Lemma~2.2]{IY}, there exists $z' \leq z$ such that $\hat{1} = \biguplus\{y,z'\}$. This contradicts the maximality of $(\Gamma,z)$. Since $\Delta$ is a chain, the only possibility is that there exists $y' \in \Gamma$ such that $y \leq y'$. But then $\hat{1} = \left\{z, \bigvee (\Gamma \cup \{y\})\right\}$, again contradicting the maximality of $(\Gamma,z)$.
    
    (2) Under the assumption of (subobject) completeness, the proof of \cite[13.2]{CLVW} is valid in any abelian category (not necessarily a concrete one), and can be translated directly into lattice-theoretic terms. Thus (2) is an immediate consequence of (1). In particular, the proof of \cite{CLVW} makes use of the following fact: Suppose $Z$ is a direct summand of $Y$ and $Y$ is a direct summand of $X$, and write $X = Z \oplus W$. Then $Y = Z \oplus (Y \cap W)$. In lattice-theoretic terms, this still holds since $Y = Z \vee (Y \wedge W)$ by the modular law and $Z \wedge (Y \wedge W) = \hat{0}$ since $W$ and $Z$ are independent.
\end{proof}

\begin{lemma}\label{lem:countable}
    Suppose that $\mathcal{L}$ is weakly JHS and does not contain an indecomposable direct summand. Then for every direct summand $y$ of $\mathcal{L}$ there exists a countably infinite subset $\Delta_y \subseteq \mathcal{L}$ such that $y = \biguplus \Delta_y$. In particular, either $|\mathcal{L}| = 1$ or $\mathcal{L}$ is superdecomposable.
\end{lemma}

\begin{proof}
    Suppose $\mathcal{L}$ does not contain an indecomposable direct summand. We suppose without loss of generality that $y = \hat{1}$ (otherwise we pass to the sublattice $[\hat{0},y]$). Since $\hat{1}$ is not indecomposable, we can write $\hat{1} = \biguplus\{y_0, y_0'\}$ with both $y_0$ and $y_0'$ nonzero. Repeating this argument, for $i \in \mathbb{N}$ we iteratively write $y_i' = \biguplus \{y_{i+1},y_{i+1}'\}$ with $y_{i+1}$ and $y_{i+1}'$ both nonzero. Now denote $\Delta := \{y_i \mid i \in \mathbb{N}\}$. It is straightforward to show that $\Delta$ is weakly independent, and thus independent by Lemma~\ref{lem:indep}. Thus $\Delta$ is a local direct summand of $\mathcal{L}$. It then follows from Lemma~\ref{lem:local}(2) that there exists $z \in \mathcal{L}$ such that $\hat{1} = \biguplus(\Delta \cup \{z\})$. This concludes the proof.
\end{proof}

 If $X$ is an (AB5$^*$) module, one can use Zorn's Lemma and the results above to express $X$ as a direct sum of indecomposable modules, see \cite[13.3]{CLVW}. A key component of this proof is that every direct summand of $X$ must itself have an indecomposable direct summand. On the other hand, Example~\ref{ex:boolean} shows that there exist weakly JHS lattices which do not contain indecomposable elements. The following shows that this cannot happen when we restrict to those lattices which are algebraic.

\begin{lemma}\label{lem:indecSummand}
    Suppose that $|\mathcal{L}| > 1$ and that $\L$ is weakly JHS and weakly algebraic. Then $\L$ contains an indecomposable direct summand.
\end{lemma}

\begin{proof}
    Suppose for a contradiction that $\L$ contains no indecomposable direct summand. Thus we can write $\hat{1} = \biguplus \{y,z\}$ with $y \neq \hat{0} \neq z$. By the assumption that $\L$ is weakly algebraic, there exists a compact element $x \in \L^c$. Moreover, we can assume without loss of generality that $x \not \leq y$ since $y \wedge z = \hat{0}$. Now let $\mathcal{F}$ be the set of local direct summands $\emptyset \neq \Delta \subseteq \L$ of $\L$ for which $x \not\leq \bigvee \Delta$. Order $\mathcal{F}$ by containment, and note $\mathcal{F}$ is nonempty since $\{z\} \in \mathcal{F}$.
    
    Let $\{\Delta_\be\}_{\be \in \Be}$ be a chain in $\mathcal{F}$. Then
    $x \not\leq \bigvee_{\be \in \Be} \left\{\bigvee \Delta_\be\right\} = \bigvee\left(\bigcup_{\be \in \Be} \Delta_\be\right)$. Moreover, it is clear that $\bigcup_{\be \in \Be}\Delta_\be$ is weakly independent, and thus independent by Lemma~\ref{lem:indep}; that is, we have that $\bigcup_{\be \in \Be} \in \mathcal{F}$ and Zorn's lemma applies.
    
    Now let $\Delta$ be maximal in $\mathcal{F}$. By Lemma~\ref{lem:local}, this means there exists $\hat{0} \neq w \in \L$ such that $\hat{1} = \biguplus \left\{w, \bigvee \Delta\right\}$. By our assumption that $\hat{1}$ does not contain an indecomposable direct summand, Lemma~\ref{lem:countable} implies that there are countably many nonzero elements $w_i \leq w$ such that $w = \biguplus_{i \in \mathbb{N}} w_i$. Since $x$ is compact, there exists a finite subset $\mathcal{S} \subseteq \mathbb{N}$ such that $x \leq \biguplus \left(\left\{\bigvee\Delta\right\} \cup \{w_i \mid i \in \mathcal{S}\}\right)$. Then for $j \in \mathbb{N} \setminus \mathcal{S}$, the definition of independence implies that $x \not\leq \biguplus \left\{w_j,\bigvee \Delta\right\}$. We conclude that  $\Delta \cup \{w_j\} \in \mathcal{F}$, which contradicts the maximality of $\Delta$.
\end{proof}

We also need the following.

\begin{lemma}\label{lem:summand_loc_fg}
    Let $\L$ be a weakly algebraic bounded modular lattice, and let $y$ be a direct summand of $\L$. Then the sublattice $[\hat{0},y]$ is weakly algebraic.
\end{lemma}

\begin{proof}
    Write $\hat{1} = \biguplus \{y,z\}$. Since $\L$ is modular, there is a lattice isomorphism between the intervals $[\hat{0},y]$ and $[z,\hat{1}]$ (given by $x \mapsto x \vee z$). Thus it suffices to show that $[z,\hat{1}]$ is weakly algebraic.

    Let $x \in \L^c$. We claim that $x \vee z$ is compact in $[z,\hat{1}]$. Indeed, let $\Delta \subseteq [z,\hat{1}]$ be directed and suppose $x \vee z \leq \bigvee \Delta$. Since $x$ is compact in $\L$, there exists $w \in \Delta$ such that $x \leq w$. But $z \leq w$ by assumption, and so $x \vee z \leq w$. This proves the claim.

    Now since $\L$ is weakly algebraic, we have that $\hat{1} = \bigvee \L^c = \bigvee (\L^c \vee z)$, where $\L^c \vee z$ is the set of compact elements of $[z,\hat{1}]$ by the claim. We conclude that $[z,\hat{1}]$ is weakly algebraic.
\end{proof}

We are now prepared to prove Theorem~\ref{thm:decomp_existence}

\begin{proof}[Proof of Theorem~\ref{thm:decomp_existence}]
    We adapt the proof of \cite[13.3]{CLVW}. If $|\L| = 1$, then $\Delta = \emptyset$ is an independent decomposition into indecomposables. Thus suppose $|\L| > 1$. Let $\mathcal{F}$ be the set of local direct summands $\Delta$ of $\mathcal{L}$ such that each $y \in \Delta$ is indecomposable. We order $\mathcal{F}$ by containment and note that $\mathcal{F}$ is nonempty by Lemma~\ref{lem:indecSummand}. Moreover, as in the proof of Lemma~\ref{lem:indecSummand}, we have that the union of a chain in $\mathcal{F}$ also lies in $\mathcal{F}$, and so Zorn's lemma applies.
    
    Let $\Delta$ be a maximal element of $\mathcal{F}$, and suppose for a contradiction that $\bigvee \Delta \neq \hat{1}$. Then by Lemma~\ref{lem:local}(2), there exists a nonzero element $y \in \mathcal{L}$ such that $\hat{1} = \biguplus\{y, \bigvee \Delta\}$. But the sublattice $[\hat{0},y]$ is also weakly JHS and weakly algebraic by Lemma~\ref{lem:summand_loc_fg}. So, $y$ has some indecomposable direct summand $y'$, which is also an indecomposable direct summand of $\mathcal{L}$ by Lemma~\ref{lem:summand_of_summand}. It then follows that $\Delta \cup \{y'\} \in \mathcal{F}$, contradicting the maximality of $\Delta$.
\end{proof}


\subsection{Uniqueness of direct sum decompositions}\label{sec:directSumUnique}

In this section, we lay out a plan for future work addressing the uniqueness of independent decompositions into indecomposables. We begin by defining three notions of uniqueness, see e.g. \cite[Definition~12.4]{CLVW} for the module-theoretic versions.

\begin{definition}\label{def:complements}
    Let $\emptyset \neq \Delta \subseteq \mathcal{L}$ be an independent decomposition of $\mathcal{L}$ into indecomposables.
    \begin{enumerate}
        \item We say that $\Delta$ \textdef{complements direct summands} if for every direct summand $y$ of $\mathcal{L}$ there exists a subset $\Delta' \subseteq \Delta$ such that $\Delta' \cup \{y\}$ is an independent decomposition of $\mathcal{L}$.
        \item We say that $\Delta$ \textdef{complements maximal direct summands} if for every indecomposable direct summand $y$, there exists a direct summand $z \in \mathcal{L}$ and an indecomposable direct summand $w \in \Delta$ such that both $\{z,y\}$ and $\{z,w\}$ are independent decompositions of $\mathcal{L}$.
        \item If $\emptyset \neq \Gamma \subseteq \mathcal{L}$ is another independent decomposition of $\mathcal{L}$ into indecomposables, then we say that $\Delta$ and $\Gamma$ are \textdef{Krull--Remak--Schmidt--Azumaya equivalent} (or \textdef{KRSA equivalent} if there exists a bijection $\Phi:\Delta \rightarrow \Gamma$ such that $[\hat{0},y] \pers [\hat{0},\Phi(y)]$ for all $y \in \Delta$.
    \end{enumerate}
\end{definition}

We recall that in categories of modules, the Krull--Remak--Schmidt--Azuyama Theorem shows that if $X$ can be expressed as a direct sum of indecomposable modules with local endomorphism rings, then this decomposition complements maximal direct summands and is unique up to KRSA equivalence. We emphasize that there are indecomposable modules whose endomorphism rings are not local, and so there is not an obvious generalization of this statement to the language of complete bounded modular lattices.

We also note that requiring a decomposition to complement direct summands is indeed a stronger than asking it only to complement maximal direct summands, as shown in the seminal work \cite{CJ}. Thus it is an interesting question to determine which modules admit such decompositions. In many works, including the work \cite{IY} on (AB5$^*$) modules, this is done by considering the following properties.

\begin{definition}\label{def:exchange}\
    \begin{enumerate}
        \item We say that $\mathcal{L}$ has the \textdef{finite exchange property} if for any direct summand $y$ of $\mathcal{L}$ and for any independent decomposition $\emptyset \neq \Delta \subseteq \mathcal{L}$ such that $|\Delta| < \infty$, there exists a subset $\Delta' \subseteq \Delta$ and nonzero elements $w_z \leq z$ for every $z \in \Delta$ such that $\{y\} \cup \{w_z \mid z \in \Delta'\}$ is an independent decomposition of $\mathcal{L}$.
        \item We say that $\mathcal{L}$ has the \textdef{(full) exchange property} if for any direct summand $y$ of $\mathcal{L}$ and for any independent decomposition $\emptyset \neq \Delta \subseteq \mathcal{L}$, there exists a subset $\Delta' \subseteq \Delta$ and nonzero elements $w_z \leq z$ for every $z \in \Delta$ such that $\{y\} \cup \{w_z \mid z \in \Delta'\}$ is an independent decomposition of $\mathcal{L}$.
    \end{enumerate}
\end{definition}

It is unknown whether the finite exchange property implies the full exchange property in full generality, but this has been shown in many cases. For example, let $X$ be a (left) module over an arbitrary ring and suppose that $X = \bigoplus_{Y \in \Delta} Y$ is a decomposition of $X$ into a direct sum of indecomposables. If $M$ has the exchange property, it follows from \cite[12.14]{CLVW} and \cite[Theorem~2.25]{MM} that $X$ also has the full exchange property, that the endomorphism ring of each $Y \in \Delta$ is local, and that the decomposition complements direct summands. In light of these facts, we conclude with the following conjecture.

\begin{conjecture}\label{conj:decomp_unique} Suppose $\mathcal{L}$ is an (AB5) lattice.
    \begin{enumerate}
        \item If $\mathcal{L}$ admits an indecomposable decomposition into indecomposables which complements maximal direct summands, then all indecomposable decompositions of $\mathcal{L}$ are KRSA equivalent.
        \item Suppose further that $\mathcal{L}$ is weakly JHS. Then $\mathcal{L}$ has the full exchange property if and only if it has the finite exchange property. Moreover, if $\mathcal{L}$ has these properties, then every independent decomposition of $\mathcal{L}$ into indecomposables complements direct summands
    \end{enumerate}
\end{conjecture}

The module-theoretic analog of Conjecture~\ref{conj:decomp_unique} is proved in \cite{CLVW,MM} using the theory of ``local semi-T-nilpotency'', which lies beyond the scope of this paper. Nevertheless, we note that one can adopt this definition to the setting of weakly algebraic lattices by replacing ``elements of a module'' with ``compact elements of a lattice''. This yields a potential strategy for resolving Conjecture~\ref{conj:decomp_unique}.

\appendix\section{Assumptions on chains and directed sets}

Let $\mathcal{L}$ be a bounded lattice. In this appendix, we discuss the relationship between several classes of (upwards and downward) directed subsets of $\mathcal{L}$.
We begin with the following definition.

\begin{definition}\label{def:directed}
    Let $\L$ be an arbitrary lattice, and let $\emptyset \neq \Delta \subseteq \L$. We say that $\Delta$ is
    \begin{itemize}
        \item[(1)] \textdef{(upward) directed} if for all $x, y \in \Delta$ there exists $z \in \Delta$ such that $x \vee y \leq z$.
        \item[(1')] \textdef{downward directed} if for all $x, y \in \Delta$ there exists $z \in \Delta$ such that $z \leq x \wedge y$.
        \item[(3)] \textdef{noetherian} (respectively \textdef{artinian}) if any totally ordered subset of $\Delta$ contains a maximal (resp. minimal) element.
        \item[(4)] of \textdef{finite length} if it is both noetherian and artinian.
        \item[(5)] a \textdef{continuous well-ordered chain} if $\Delta = \{y_\alpha\}_{\alpha \in I}$ is well-ordered (equivalently an artinian chain) and for any limit ordinal $\lambda \in I$ one has $\bigvee_{\alpha < \lambda} \{y_\alpha\} = y_\lambda$.
    \end{itemize}
\end{definition}
Note in particular that any chain is both upward and downward directed. Moreover, we have that a subset $\Delta \subseteq \L$ is downward directed (respectively noetherian) if and only if it is an upward directed (respectively artinian) subset of $\L^{op}:= (\L, \leq^{op})$.

A proof of the following can be found as part of \cite[Lemma~1.2.14]{GT}.

\begin{lemma}\label{lem:ordinal_chain}
    Let $\L$ be an arbitrary lattice, let $\Delta \subseteq \L$ be an infinite upward directed set, and let $2^{\Delta}$ be the lattice of subsets of $\Delta$ (ordered by inclusion). Then there exists a continuous and well-ordered chain $\{\Delta_\alpha \mid \alpha \in I\} \subseteq 2^{\Delta}$ such that:
    \begin{enumerate}
        \item Each $\Delta_\alpha$ is an upward directed subset of $\L$.
        \item The cardinality of each $\Delta_\alpha$ is strictly smaller than that of $\Delta$.
        \item $\Delta = \bigcup_{\alpha \in I} \Delta_\alpha$.
    \end{enumerate}
\end{lemma}

As a consequence, we obtain the main results of this appendix.

\begin{samepage}

\begin{proposition}\label{prop:complete}
    Let $\L$ be an arbitrary bounded lattice. Then the following are equivalent.
    \begin{itemize}
        \item[(1)] $\L$ is complete.
        \item[(1')] $\L$ is cocomplete.
        \item[(2)] every upward directed set of $\L$ has a supremum.
        \item[(2')] every downward directed set of $\L$ has an infimum.
        \item[(3)] every chain in $\L$ has a supremum.
        \item[(3')] every chain in $\L$ has an infimum.
        \item[(4)] every well-ordered continuous chain in $\L$ has a supremum.
        \item[(4')] every well-ordered continuous chain in $\L^{op}$ has a supremum (with respect to $\leq^{op})$.
    \end{itemize}
\end{proposition}

\end{samepage}

\begin{proof}
    We have already seen that (1) is equivalent to (1'). Thus by duality it suffices to show that (1), (2), (3), and (4) are equivalent. The implications $(1) \implies (2) \implies (3) \implies (4)$ are clear.
    
    $(4 \implies 2)$: Let $\Delta\subseteq \L$ be (upward) directed. We proceed by induction on the cardinality $|\Delta|$. For the base case, we recall that, in any lattice, $\bigvee \Delta$ always exists when $|\Delta| < \infty$.
    
    Now suppose $|\Delta| = \omega$ for some infinite ordinal $\omega$, and suppose the result holds for all smaller ordinals. Choose some $\{\Delta_\alpha\}_{\alpha \in I}$ as in Lemma~\ref{lem:ordinal_chain}. By the induction hypothesis, we have that $\bigvee \Delta_\alpha$ exists for all $\alpha \in I$. We claim that $\left\{\bigvee \Delta_\alpha\right\}_{\alpha \in I}$ is a continuous well-ordered chain in $\L$. From this, it will follow that $$\bigvee \Delta = \bigvee\left(\bigcup_{\alpha \in I}\Delta_\alpha\right) = \bigvee_{\alpha \in I} \left\{\bigvee \Delta_\alpha\right\}$$
    exists by (4).
    
    We now prove the claim. We first recall that every subset of a well-ordered set is well-ordered, and in particular has a minimal element. Now, for $\alpha \in I$, denote $\Omega(\alpha) := \{\beta \in I \mid \bigvee \Delta_\alpha = \bigvee \Delta_\beta\}$, and denote $I_{\mathrm{min}} := \{\min \Omega(\alpha) \mid \alpha \in I\}$. Then there exists a well ordered set $J$ and an order-preserving bijection $J \rightarrow I_{\mathrm{min}}$, which we denote by $j \mapsto i_j$. For any limit ordinal $\lambda \in J$, we then have
    $$\bigvee_{j < \lambda}\left\{\bigvee \Delta_{i_j}\right\} = \bigvee_{\alpha < i_\lambda}\left\{\bigvee \Delta_\alpha\right\} = \bigvee\left(\bigcup_{\alpha < i_\lambda} \Delta_\alpha \right) = \bigvee \Delta_{i_\lambda}.$$
    This proves the claim.
    
    $(2 \implies 1)$: Let $\Delta \subseteq \L$ be arbitrary, and let
    $$\widetilde{\Delta} = \left\{\bigvee \Delta' \mid \Delta' \subseteq \Delta \text{ and } |\Delta'| < \infty\right\}.$$
    By construction, $\widetilde{\Delta}$ is a directed subset of $\L$, and so $\bigvee \widetilde{\Delta} = \bigvee \Delta$ exists by (2).
\end{proof}

As a categorical analog of Proposition~\ref{prop:complete}, we have the following.

\begin{lemma}[Lemma~\ref{lem:catComplete}]\label{lem:cat_complete_appendix}
    Let $X$ be an object in $\Cat$. Then the following are equivalent.
    \begin{enumerate}
        \item For every chain $\Delta \subseteq \Sub(X)$, the direct limit $\colim\Delta$ exists (in $\mathcal{A}$) and is a subobject of $X$.
        \item For every well-ordered chain $\Delta \subseteq \Sub(X)$, the direct limit $\colim \Delta$ exists (in $\Cat$) and is a subobject of $X$.
        \item $\mathcal{A}$ has exact $X$-bounded direct limits.
    \end{enumerate}
\end{lemma}

\begin{proof}
    The implications $(3)\implies(1)\implies(2)$ are clear, so we prove that $(2)\implies(3)$. Suppose (2) holds, and let $\Delta \subseteq \Sub(X)$ be (upward) directed. As in the proof of Proposition~\ref{prop:complete}, we proceed by induction on the cardinality $|\Delta|$. For the case case, we recall that $\colim\Delta = \sum\Delta = \bigvee \Delta$ whenever $|\Delta| < \infty$.
    
    Now suppose $|\Delta| = \omega$ for some infinite ordinal $\omega$, and suppose the result holds for all smaller ordinals. Choose some $\{\Delta_\alpha\}_{\alpha \in I}$ as in Lemma~\ref{lem:ordinal_chain}. By the induction hypothesis, we have that $\colim\Delta_\alpha$ exists and is a subobject of $X$ for all $\alpha \in I$. In particular, we have $\colim\Delta_\alpha = \bigvee \Delta_\alpha$ for each $\alpha$, and so $\{\colim\Delta_\alpha\}$ is a continuous well-ordered chain in $\Sub(X)$ by the same argument as in Proposition~\ref{prop:complete}. From this, it follows that
    $\colim\Delta = \colim_{\alpha \in I}(\colim\Delta_\alpha) $
    exists and is a subobject of $X$ by (2).
\end{proof}

Recall the notational conventions from Notation~\ref{nota:wedge_set}. To simplify the statement of our last result, for $z \in \mathcal{L}$ and $\Delta \subseteq \mathcal{L}$, we consider the following equations
\begin{eqnarray}
    z \wedge \left(\bigvee \Delta\right) &=& \bigvee (z \wedge \Delta) \label{eqn:ab5}\\
    z \vee \left(\bigwedge \Delta \right) &=& \bigwedge (z \vee \Delta)\label{eqn:ab5*}.
\end{eqnarray}

\begin{proposition}\label{prop:AB5}
    Let $\mathcal{L}$ be a complete bounded modular lattice.
    \begin{enumerate}
        \item The following are equivalent:
        \begin{enumerate}
            \item $\mathcal{L}$ is (AB5).
            \item Equation~\ref{eqn:ab5} holds for every well-ordered continuous chain $\Delta \subseteq \mathcal{L}$ and every $z \in \mathcal{L}$.
            \item Equation~\ref{eqn:ab5} holds for (upward) directed set $\Delta \subseteq \mathcal{L}$ and every $z \in \mathcal{L}$.
        \end{enumerate}
        \item The following are equivalent:
        \begin{enumerate}
            \item $\mathcal{L}$ is (AB5$^*$).
            \item Equation~\ref{eqn:ab5*} holds for every chain $\Delta \subseteq \mathcal{L}$ which is continuous and well-ordered with respect to $\leq^{op}$ and every $z \in \mathcal{L}$.
            \item Equation~\ref{eqn:ab5*} holds for downward directed set $\Delta \subseteq \mathcal{L}$ and every $z \in \mathcal{L}$.
        \end{enumerate}
    \end{enumerate}
\end{proposition}

\begin{proof}
   By duality it suffices to prove only (1). The implications $(1c) \implies (1a) \implies (1b)$ are clear.
    
    $(1b \implies 1c)$: Let $\Delta \subseteq \mathcal{L}$ be (upward) directed and let $z \in \mathcal{L}$. We proceed by induction on the cardinality $|\Delta|$. We first note that if $|\Delta| < \infty$, then $\Delta$ has a maximal elements and thus Equation~\ref{eqn:ab5} automatically holds.
    
    Now suppose $|\Delta| = \omega$ for some infinite ordinal $\omega$, and suppose the result holds for all smaller ordinals. Choose some $\{\Delta_\alpha\}_{\alpha \in I}$ as in Lemma~\ref{lem:ordinal_chain}. By the induction hypothesis, we have that $z \wedge \left(\bigvee \Delta_\alpha\right) = \bigvee(z \wedge \Delta_\alpha)$ for all $\alpha \in I$. Moreover, as in the proof of Proposition~\ref{prop:complete}, we have that $\{\bigvee \Delta_\alpha\}_{\alpha \in I}$ is a continuous well-ordered chain in $\mathcal{L}$. By the assumption of (1b) it therefore follows that
    \begin{eqnarray*}
        z \wedge \left(\bigvee \Delta\right) &=& z \wedge \left(\bigvee_{\alpha \in I} \left(\bigvee \Delta_\alpha\right)\right)\\
        &=& \bigvee_{\alpha \in I} z \wedge \left(\bigvee \Delta_\alpha\right)\\
        &=& \bigvee_{\alpha \in I} \left(\bigvee \left(z \wedge \Delta_\alpha\right)\right)\\
        &=& \bigvee (z \wedge \Delta).
    \end{eqnarray*}
    This proves the result.
\end{proof}


\section{Technical results on biproducts}

In this appendix, we prove Lemma~\ref{lem:inclusion_gluing}. Recall that, given a morphism $f: X \rightarrow Y$ in an abelian category, there is a functor $f^{-1}: \Sub(Y) \rightarrow \Sub(X)$ which sends $(Z,\iota_Z) \in \Sub(Y)$ to the pullback of the diagram $X \xrightarrow{f} Y \xleftarrow{\iota_Z} Z$, and that this functor distributes over intersections.

\begin{lemma}\label{lem:biproduct_1}
    Let $f: X \rightarrow Y$ be an isomorphism between objects of $\mathcal{A}$. Denote $\iota_X: X\rightarrow X \oplus Y$ and $\pi_X: X \oplus Y \rightarrow X$ the inclusion and projection maps, and likewise for $Y$. Let $d: X \rightarrow X \oplus Y$ be the unique map which satisfies $\pi_X \circ d = -1_X$ and $\pi_Y \circ d = f$ (in matrix form, this map is $d = [-1_X, f]^\top$). Then the following hold in the lattice $\Sub(X\oplus Y)$:
    \begin{enumerate}
        \item $\im(d) \cap \im(\iota_X) = 0 = \im(d) \cap \im(\iota_Y)$.
        \item $\im(d) + \im(\iota_X) = X\oplus Y = \im(d) + \im(\iota_Y)$.
    \end{enumerate}
\end{lemma}

\begin{proof}
    Note that $\im(d) = \im(d\circ f^{-1})$ and that $-d \circ f^{-1}$ satisfies $-\pi_X \circ d \circ f^{-1} = f^{-1}$ and $-\pi_Y \circ d \circ f^{-1} = -1_Y$. Thus it suffices to prove the first half of each statement. Moreover, recall that there is a functor $d^{-1}: \Sub(X\oplus Y) \rightarrow \Sub(X)$ which sends $Z \in \Sub(X\oplus Y)$ to the pullback of the diagram $X \xrightarrow{d} X\oplus Y \xleftarrow{\iota_Z} Z$.

    (1) Since $d$ is a monomorphism, it suffices to show that $d^{-1}(\im(d) \cap \im(\iota_X)) = 0 \in \Sub(X)$. Since $\im(\iota_X) = \ker(\pi_Y)$ and $d^{-1}(\im(d) \cap \im(\iota_X)) \subob d^{-1}(\im(\iota_X))$, it follows that
        $$d^{-1}(\im(\iota_X)) = d^{-1}(\ker(\pi_Y)) = \ker(\pi_Y \circ d) = \ker(f) = 0.$$
        This proves (1).

    (2) Note that 
        $\pi_X \circ (\iota_X + d) \circ f^{-1}  = f^{-1} - f^{-1}  = \pi_X \circ \iota_Y$ and 
        $\pi_Y \circ (\iota_X + d)\circ f^{-1}  = 0 + 1_Y  = \pi_Y \circ \iota_Y$. Thus by the universal property of the product, we have that $\iota_Y = (\iota_X + d) \circ f^{-1}$. This means that $\iota_Y$ factors through the inclusion of $\im(\iota_X) + \im(d)$. We conclude that $X\oplus Y = \im(\iota_X) + \im(\iota_Y) \subob \im(\iota_X) + \im(D)$ This proves (2).
\end{proof}

\begin{lemma}[Lemma~\ref{lem:inclusion_gluing}]\label{lem:inclusion_gluing_appendix}
    Let $X$ and $Y$ be objects in $\mathcal{A}$, and suppose that $X\oplus Y$ is (AB5). Let $\Delta = (X_\alpha)_{\alpha \in I} \subseteq \Sub(X)$ and $\Gamma = (Y_\alpha)_{\alpha \in I} \subseteq \Sub(Y)$ be directed diagrams of the same shape. For $\alpha \leq \beta \in I$, denote by $x_{\alpha,\beta}: X_\alpha \rightarrow X_\beta$, $y_{\alpha,\beta}: Y_\alpha \rightarrow Y_\beta$, $x_\alpha: X_\alpha \rightarrow X$, and $y_\beta: Y_\beta \rightarrow X$ the inclusion maps. Suppose that $\bigvee \Delta = X$ (in $\Sub(X)$) and $\bigvee \Gamma = Y$ (in $\Sub(Y)$) and that there is a system of isomorphisms $(f_\alpha: X_\alpha \rightarrow Y_\alpha)_{\alpha \in I}$ such that $f_\beta \circ x_{\alpha,\beta} = y_{\alpha,\beta} \circ f_\alpha$ for all $\alpha \leq \beta \in I$. Then there is an isomorphism $f: X\rightarrow Y$ which satisfies $(\star)$ $f \circ x_\alpha = y_\beta \circ f_\alpha$ for all $\alpha \in I$. Furthermore, $f$ is the unique morphism $X \rightarrow Y$ which satisfies $(\star)$.
\end{lemma}

\begin{proof}
    Let $\iota_X,\iota_Y, \pi_X$, and $\pi_Y$ be as in Lemma~\ref{lem:biproduct_1}. For $\alpha \in I$, let $d_\alpha: X_\alpha \rightarrow X \oplus Y$ be the unique map which satisfies $\pi_X \circ d_\alpha = -x_\alpha$ and $\pi_Y \circ d_\alpha = y_\alpha \circ f_\alpha$. Denote $D = \bigvee_\alpha \im(d_\alpha) \in \Sub(X\oplus Y)$, and let $p: X \oplus Y \rightarrow (X \oplus Y)/D$ be the quotient map. We claim that $p \circ \iota_X$ and $p \circ \iota_Y$ are isomorphisms.

    We first show that $p\circ \iota_X$ is a monomorphism. Since $\iota_X$ is a monomorphism and $\ker(p) = D$, it suffices to show that $\im(\iota_X) \cap D = 0$. Indeed:
        \begin{align*}
        \im(\iota_X) \cap D &= \im(\iota_X) \cap \left(\bigvee_{\alpha \in I} D_\alpha\right) &\\
        &= \bigvee_{\alpha \in I}\left[\im(\iota_X) \cap \im(d_\alpha)\right] &\text{ by (AB5)}\\
        &= \bigvee_{\alpha \in I} \left(\bigvee_{\beta \in I} \left[\im(\iota_X \circ x_\beta) \cap 
        \im(d_\alpha)\right]\right) &\\
        &= \bigvee_{\alpha \in I} \left[\im(\iota_X \circ x_\alpha) \cap \im(d_\alpha)\right] & \text{because $\{(\alpha,\alpha) \mid \alpha \in I\}$ is cofinal in $I\times I$}\\
        &= 0 & \text{by Lemma~\ref{lem:biproduct_1}.}\end{align*}

We conclude that $p \circ \iota_X$ is a monomorphism. The fact that $p \circ \iota_Y$ is a monomorphism then follows from replacing each $X$ with $Y$ and each $x$ with $y$ in the above argument.

We next show that $p \circ \iota_X$ is an epimorphism. Since $\ker(p) = D$, we note that $\im(p \circ \iota_X) = \im\left(p|_{\im(\iota_X) + D}\right)$. Moreover, for all $\alpha \in I$, Lemma~\ref{lem:biproduct_1} implies that $\im(\iota_X) + \im(\iota_Y) = \im(\iota_X) + \im(d_\alpha) \subob \im(\iota_X) + D$. Taking supremums, this implies that $X \oplus Y \subob \im(\iota_X) + D$, and so $\im(\iota_X) + D = X \oplus Y$. Thus $p \circ \iota_X$ is an epimorphism since $p$ is. The fact that $p \circ \iota_Y$ is an epimorphism then follows from interchanging the roles of $X$ and $Y$.

We have shown that $p \circ \iota_X$ and $p \circ \iota_Y$ are both monic and epic, and thus they are both isomorphisms since $\mathcal{A}$ is abelian. Now denote $f := -(p \circ \iota_Y)^{-1} \circ p \circ \iota_X: X \rightarrow Y$. Then for $\alpha \in I$, we have that 
$$p \circ \iota_Y \circ y_\alpha \circ f_\alpha = p \circ \iota_Y \circ \pi_Y \circ d_\alpha = p \circ \iota_X \circ \pi_X \circ d_\alpha = -p \circ \iota_X \circ x_\alpha.$$
Thus $f \circ x_\alpha = y_\alpha\circ f_\alpha$; that is, $f$ satisfies $(\star)$.

It remains to show that $f$ is the unique morphism which satisfies $(\star)$. Let $g: X \rightarrow Y$ satisfy $(\star)$. Then $X_\alpha \subob \ker(f-g)$ for all $\alpha \in I$. But this means $X = \ker(f-g)$, and so $f = g$.
\end{proof}

\end{document}